\documentclass[11pt]{article}

\usepackage{amsmath,amssymb,amsthm,epsfig,color,verbatim,url}
\usepackage{amsfonts}
\usepackage{amssymb}
\usepackage{url}
\usepackage{epsfig}
\usepackage{multirow}
\usepackage{array}
\usepackage{subcaption}
\usepackage{float}

\usepackage{algorithm}
\usepackage{algorithmicx}
\usepackage{algpseudocode}
\usepackage[hidelinks]{hyperref}

\usepackage{fullpage,array}

\usepackage{graphicx} 

\usepackage{mathtools}
\usepackage{authblk}
\usepackage[table,xcdraw]{xcolor}
\usepackage{colortbl}

\DeclarePairedDelimiter{\ceil}{\lceil}{\rceil}

\newtheorem{theorem}{Theorem}[section]
\newtheorem{corollary}[theorem]{Corollary}
\newtheorem{proposition}[theorem]{Proposition}
\newtheorem{lemma}[theorem]{Lemma}

\newtheorem{convention}[theorem]{Convention}

\newtheorem*{theoremportrait}{Theorem~\ref{thm:word_from_portrait}}

\theoremstyle{definition}
\newtheorem{definition}[theorem]{Definition}

\newcommand{\A}{\mathcal A}
\newcommand{\B}{\mathcal{B}}

\newcommand{\G}{\mathcal G}
\newcommand{\N}{\mathcal{N}}

\newcommand{\Z}{\mathbb Z}
\newcommand{\Aut}{\mathop{\rm Aut}\nolimits}
\newcommand{\Sym}{\mathop{\rm Sym}\nolimits}

\DeclareMathOperator{\St}{St}

\usepackage{pgfplots}
\usepgfplotslibrary{colorbrewer}
\usetikzlibrary{pgfplots.statistics, pgfplots.colorbrewer} 
\pgfplotsset{compat=1.18}

\usepackage{pgfplotstable}
\usepackage{filecontents}

\makeatletter
\usepackage{tikz}

\usetikzlibrary{automata, positioning, arrows}

\tikzset{
        ->,  
        >=stealth',
        node distance=3cm, 
        every state/.style={thick, fill=gray!10}, 
        initial text=$ $,
        }

\usepackage[linguistics]{forest}

\newcommand*\circled[2][1.6]{\tikz[baseline=(char.base)]{
    \node[shape=circle, draw, inner sep=1pt, 
        minimum height={\f@size*#1},] (char) {\vphantom{WAH1g}#2};}}
\makeatother

\begin{document}

\title{Contracting Self-Similar Groups in Group-Based Cryptography}

\author[1]{Delaram Kahrobaei}
\author[2]{Arsalan Akram Malik}
\author[3]{Dmytro Savchuk}

\affil[1]{Computer Science and Mathematics Departments\\
        The City University of New York, Queens College, CUNY\\
        The Initiative for the Theoretical Sciences, CUNY Graduate Center\\
        Department of Computer Science, University of York (UK)\\
        Department of Computer Science and Engineering, Tandon School of Engineering, New York University\\
        \href{mailto:Delaram.Kahrobaei@qc.cuny.edu}{Delaram.Kahrobaei@qc.cuny.edu}}
\affil[2]{Department of Mathematical Sciences\\
        Indiana University South Bend\\
        1700 Mishawaka Ave\\
        South Bend, IN 46615\\
        \href{mailto:malikaa@iu.edu}{malikaa@iu.edu}}
\affil[3]{Department of Mathematics and Statistics\\
        University of South Florida\\
        4202 E Fowler Ave\\
        Tampa, FL 33620-5700\\
        \href{mailto:savchuk@usf.edu}{savchuk@usf.edu}}

\date{}

\maketitle
\begin{abstract}
We propose self-similar contracting groups as a platform for cryptographic schemes based on simultaneous conjugacy search problem (SCSP). The class of these groups contains extraordinary examples like Grigorchuk group, which is known to be non-linear, thus making some of existing attacks against SCSP inapplicable. The groups in this class admit a natural normal form based on the notion of a nucleus portrait, that plays a key role in our approach. While for some groups in the class the conjugacy search problem has been studied, there are many groups for which no algorithms solving it are known. Moreover, there are some self-similar groups with undecidable conjugacy problem. We discuss benefits and drawbacks of using these groups in group-based cryptography and provide computational analysis of variants of the length-based attack on SCSP for some groups in the class, including Grigorchuk group, Basilica group, and others.
\end{abstract}

\tableofcontents

\section{Introduction}

Group-based cryptography is a relatively new area in post-quantum cryptography that started to develop in 1990's and rapidly grew in the recent 20 years. Several sources covering main developments in the area include~\cite{myasnikov_su:group-based_cryptography08,myasnikov_su:non-commutative_crypto_book11,diekert:discrete_algebraic_methods16,kahrobaeinotices, Kahrobaei-BattarbeeBook}.

Among many cryptographic schemes that were proposed based on non-abelian groups, there are several that are based on the complexity of various problems related to conjugacy in groups, such as conjugacy search problem, simultaneous conjugacy search problem, and other problems related to them. The conjugacy decision problem asks whether or not two given elements $a,b$ of a group $G$ are conjugate. The conjugacy search problem (CSP), on the other hand, given two elements $a,b$ of a group $G$ that are conjugate, asks to find at least one $r\in G$ such that $a^r=r^{-1}ar=b$. And the simultaneous conjugacy search problem (SCSP) similarly asks, given two tuples $(a_1,a_2,\ldots, a_n)$ and $(b_1,b_2,\ldots, b _n)$ of elements in $G$, such that there is $r\in G$ satisfying $a_i^r=b_i$ for all $1\leq i\leq n$, to find at least one such $r$. The first protocol in this family was proposed by Anshel-Anshel-Goldfeld (AAG) commutator key exchange protocol was proposed in~\cite{AAG}. Later Ko, Lee et al in~\cite{KoLee00} introduced a non-commutative version of Diffie-Hellman key exchange protocol. This was followed by many other schemes that were based on variants and generalizations of the conjugacy search problem. We mention here a relatively recent WalnutDSA scheme~\cite{anshel2017walnutdsa,anshel2021walnutdsa} that was proposed for standardization at a 2016 NIST call for post-quantum cryptographic constructions. A (nonexhaustive) list of earlier proposed schemes include Stickel key exchange protocol~\cite{stickel2005new}, Shpilrain-Ushakov scheme based on the decomposition problem~\cite{shpilrain2005thompson}, Kurt key exchange primitive based on the triple decomposition problem~\cite{kurt2006new}, Shpilrain-Zapata scheme  based on subgroup membership search problem~\cite{shpilrain2006zapata},
Kahrobaei-Koupparis digital signature scheme~\cite{KK12}, Kahrobaei-Khan non-commutative ElGamal public key encryption scheme~\cite{KK06}. All of these schemes are generic protocols that require a choice of an appropriate platform group or a class of groups to operate. For example, the original Ko-Lee and AAG protocols were proposing to use braid groups as platform groups, while Shpilrain and Ushakov considered Thompson's group $F$ as a platform group in~\cite{shpilrain2005thompson}. Other generalizations of AAG and Ko-Lee protocols were recently proposed by Habeeb~\cite{habeeb:oblivious23} and very recently by Lin and Han~\cite{lin2026quantumsafekeyexchangescheme}. Several cryptanalysis attacks were suggested~\cite{kreuzer_mu:linear_algebra_attack14,ben2018cryptanalysis} that showed certain vulnerabilities of the schemes in the case when the group has efficient matrix representations.

Automaton groups have been proposed as a platform for some cryptographic protocols with history described in detail in a recent paper~\cite{KNR}.
The first protocol of this type was based on Grigorchuk groups~\cite{grigorchukcrypto}, although it was shown to be insecure in \cite{grigorchukcripto2}. In~\cite{grigorchuk2019keyagreement}, Grigorchuk and Grigoriev suggested some specific families of automaton groups, known as branch groups, as platforms for Anshel-Anshel-Goldfeld key-agreement. Among others, they suggest the aforementioned Grigorchuk group and the Basilica group~\cite{grigorch_z:basilica,bondarenko_gkmnss:full_clas32}.

In most cases the groups generated by automata possess very peculiar and interesting characteristics. They provide a range of new examples that essentially do not have a simpler description than the one given by automaton. They include groups with properties such as torsion infinite finitely generated, or being of intermediate growth, and thus non-linear by Tits alternative~\cite{tits:free_groups_in_linear_groups72}. The class of automaton groups is also very interesting from algorithmic viewpoint. These groups admit effective calculations with group elements as the word problem is decidable (with exponential time complexity in general and polynomial time complexity in the case of contracting groups). However, many standard algorithmic problems are either undecidable or their decidability status is unknown. For example, \v Suni\'c and Ventura showed that the conjugacy problem is undecidable~\cite{sunic_v:conjugacy_undecidable} and Gillibert showed that the finiteness of an automaton semigroups is undecidable~\cite{gillibert:finiteness14}, and so is the order problem~\cite{gillibert:order_undecidable18,bartholdi_m:order_undecidable}.

These algorithmic properties make the class of automaton groups a good candidate for group-based cryptography. Moreover, many automaton groups possess naturally commuting subgroups, which is a requirement for some protocols, such as Ko-Lee key exchange protocol. However, there are some positive results regarding the solutions of conjugacy (search) problem in automaton groups. For example, it was proved that the conjugacy problem in Grigorchuk group $\G$ is decidable by Leonov~\cite{leonov:conjugacy} and Rozhkov~\cite{rozhkov:conjugacy}. These ideas were used by Grigorchuk and Wilson to construct an algorithm solving the conjugacy problem in the certain class of branch groups~\cite{grigorch_w:conjugacy}. Later polynomial time algorithm solving the conjugacy problem in $\G$ was constructed by Lysenok, Myasnikov, and Ushakov~\cite{lysenok_mu:conjugacy_problem_in_grigorchuk_group_polynomial10} and very recently a linear time algorithm was developed by Modi, Seedhom, and Ushakov in~\cite{modi_su:linear_time_algorithm_conjugacy_grigorchuk21}. One more large group for which the conjugacy problem is decidable is the group $B_d$ of all bounded automata over a $d$-letter alphabet~\cite{bondarenko_bsz:conjugacy}. This group, as well as the notion of polynomially and exponentially growing automata, was introduced by Sidki~\cite{sidki:nofree} who proved that finite automata of polynomial growth  generate groups that do not contain free nonabelian subgroups. Note that group $B_d$ is very large and not finitely generated, and the solution of the conjugacy problem in $B_d$ does not yield a solution of this problem in any subgroup of $B_d$.

In this paper we study a possibility to use the class of contracting automaton groups for cryptographic protocols based on variants of conjugacy search problem. This class includes many famous groups, such as Grigorchuk group, Gupta-Sidki group, Basilica group, etc. As mentioned above, it allows for effective calculations with group elements since there is a polynomial time algorithm solving the word problem that goes back to calculations of Grigorchuk~\cite{grigorch:degrees} and whose complexity was estimated by the third author in~\cite{savchuk:wp}. The class of contracting groups was formally introduced by Nekrashevych who discovered a deep connection between automaton groups and holomorphic dynamics via the notion of iterated monodromy groups (see~\cite{nekrash:self-similar}). Contracting groups can either be defined through their actions on regular rooted trees or directly using the structure of automata defining group elements. 

Each element of a contracting group can be identified with a, so-called, nucleus portrait that is a finite tree whose non-leaf vertices are labeled by permutations of the input alphabet, and the leaves are labeled by the elements of a finite subset of a group called the nucleus. Such a representation can be used as an analogue of the normal form (which is generally not defined for automaton). Using nucleus portraits instead of words to represent elements of a group potentially offers additional security for cryptographic protocols. For example, it was pointed out by Dehornoy~\cite{dehornoy:braid-based_cryptography2004} that while using cryptographic protocols based on the conjugation in braid groups, the (word) length of the secret conjugator can be inferred from the publicly transmitted data as a consequence of the fact that the normal form in braid groups does not significantly change the length of any reduced word. The size and the structure of a nucleus portrait are much less predictable under conjugation in the case of contracting groups. Moreover, to the best of our knowledge there is no algorithm for a general contracting group that recovers some word in generators that represents a given nucleus portrait. However, we present a polynomial time algorithm in the case of regular branch contracting groups:
\begin{theoremportrait}
Let $G$ be a contracting self-similar regular branch group acting on $T_d$, generated by a symmetric set $S$ that contains the nucleus of $G$. There is an algorithm that, given the nucleus portrait of an element $g\in G$ that can be defined by an (unknown) word in $S$ of length up to $n$, constructs a word in $S$ of polynomial in $n$ length, representing $g$. The complexity of this algorithm is polynomial in $n$.
\end{theoremportrait}
Note, that the words obtained as the outcome of the algorithm in the above theorem are very different from the input words, so even in the case of regular branch contracting groups nucleus portraits significantly hide the information about the input words. The proposed algorithm is a generalization of a similar algorithm for Grigorchuk group described in~\cite{grigorch:solved,lysenok_mu:conjugacy_problem_in_grigorchuk_group_polynomial10} and uses the standard techniques in computational group theory such as Reidemeister-Schreier process and Schreier graphs. Theorem~\ref{thm:word_from_portrait} is a direct corollary of a more general Theorem~\ref{thm:word_from_portrait_depth} that also gives a way to check if a given portrait in the contracting closure $\mathcal NG$ of $G$ (as defined in Section~\ref{sec:word_from_portrait}) represents an element of $G$, thus addressing a membership problem for $G$ in $\mathcal NG$.

We explore the techniques used for cryptanalysis of AAG protocol known as length-based Attack (LBA) on SCSP. This attack was initially introduced by Hughes and Tannenbaum in~\cite{HughesTann02} in the context of braid groups and its various forms were studied in several classes of groups, including free groups and groups with generic free basis property~\cite{myasnikov_u:random_subgroups_and_analysis_of_length-based_and_quotient_attacks08,myasnikov_u:length-based_attack_on_braid_groups07} and Kahrobaei-Garber-Lam did the analysis for polycyclic groups~\cite{graber_kn:length-based_attack15}. The LBA on SCSP attempts to find a secret conjugator heuristically by subsequently finding a sequence of ``partial conjugators'' $r_k$, $k\geq1$ that decrease the ``length'' of a tuple $(a_1^{r_k}b_1^{-1},a_2^{r_k}b_2^{-1},\ldots, a_n^{r_k}b_n^{-1})$, where the length function can be defined in multiple ways. In our implementation we found that the most reasonable choice for the length function is the sum of the depths of nucleus portraits of the entries of a given tuple. There are several more adjustable parameters of the version of LBA that we implemented that are discussed in detail in Section~\ref{sec:LBA}.

Myasnikov and Ushakov suggested in~\cite{myasnikov_u:random_subgroups_and_analysis_of_length-based_and_quotient_attacks08} that the LBA should work well in groups that have generically free subgroups. By the result of Nekrashevych~\cite{nekrashevych:contracting_no_free} contracting groups have no free nonabelian subgroups, so they constitute a natural class where LBA may not be effective.

For our statistical analysis we have chosen several contracting groups and ran the LBA attack against SCSP for them. These include Grigorchuk group $\G$~\cite{grigorch:burnside}, Basilica group $\B\cong \textrm{IMG}(z^2-1)$~\cite{grigorch_z:basilica} and its generalizations $p$-Basilica groups introduced in~\cite{didomenico_fnt:p-basilica22}, the group of intermediate growth $\textrm{IMG}(z^2+i)$~\cite{bux_p:iter_monodromy,grigorch_ss:img}, Universal Grigorchuk group (acting on 6-ary regular rooted tree)~\cite[Problem~9.5]{grigorch:solved}, two virtually abelian groups, and 2 groups associated with iterated monodromy groups of quadratic rational functions.  We found that the LBA variant against SCSP that we studied is mostly ineffective for groups which are not virtually abelian, but works quite well (with success rate of more than fifty percent for our choice of parameters) for 2 groups that contain $\Z^2$ as a subgroup of index 2.

Analyzing all possible attacks against SCSP in the case of contracting groups is not in the scope of this paper. Instead, we are setting up the stage for further research, some of which is under way. Scuderi, Seekamp, and the third author are currently studying a new portrait-based heuristic attack against SCSP in the case of contracting groups acting on binary tree~\cite{portrait_based_attack:github} when the length of the secret conjugator is known. The paper containing the detailed result of this attack is currently under preparation.

The paper is organized as follows. In Section~\ref{sec:cryptosystems} we recall several group-based cryptosystems and the algorithmic problems they are based upon. Sections~\ref{sec:automaton_groups} and~\ref{sec:contracting} we introduce the basic notions related to automaton and contracting groups. Our main theoretical result is contained in Section~\ref{sec:word_from_portrait}, in which we develop an algorithm of recovering a word representing a given nucleus portrait of an element of a contracting regular branch group. Section~\ref{sec:LBA} describes in detail the variant of the length-based attack that we study, and Section~\ref{sec:statistics} contains the results of our statistical analysis of the effectiveness of the proposed length-based attack.

\section{Cryptosystems Based on Variants of the Conjugacy Search Problem in Groups}
\label{sec:cryptosystems}

In this section we describe some well-studied cryptography protocols that are based on variants of conjugacy search problem. The list is not exhaustive and we present it here to provide additional motivation for our work.

\subsection{Commutator Key Exchange: a.k.a. Anshel-Anshel-Goldfeld (AAG) protocol}
\label{AAG}
As usual, we use two entities, called Alice and Bob, for presenting the two parties which plan to communicate over an insecure channel.

Let $G$ be a group with generators $g_1, \ldots, g_n$. First, Alice chooses her public set $\overline{a}=(a_1,\ldots,a_{N_1})$, where $a_i \in G$, and Bob chooses his public set $\overline{b}=(b_1,\ldots,b_{N_2})$, where $b_i \in G$. They both publish their sets. Alice then chooses her private key $A=a_{s_1}^{\varepsilon_1} \cdots a_{s_L}^{\varepsilon_L}$, where $a_{s_i} \in \overline{a}$ and $\varepsilon_i \in \{\pm1\}$. Bob also chooses his private key $B=b_{t_1}^{\delta_1} \cdots b_{t_M}^{\delta_M}$, where $b_{t_i} \in \overline{b}$ and $\delta_i \in \{\pm1\}$. Alice computes $b'_i=A^{-1}b_iA$ for all $b_i \in \overline{b}$ and sends it to Bob. Bob also computes $a'_i=B^{-1}a_iB$ for all $a_i \in \overline{a}$ and sends it to Alice. Now, the shared secret key is $K=A^{-1}B^{-1}AB$. Alice can computes this key by:
\begin{eqnarray}
K_A & = & A^{-1} (a'^{\varepsilon_1}_{s_1} \cdots a'^{\varepsilon_L}_{s_L}) = A^{-1} (B^{-1} a_{s_1} B)^{\varepsilon_1} \cdots (B^{-1} a_{s_L} B)^{\varepsilon_L} = \nonumber \\
& = & A^{-1}B^{-1} (a_{s_1}^{\varepsilon_1} \cdots a_{s_L}^{\varepsilon_L}) B = A^{-1}B^{-1}AB = K. \nonumber
\end{eqnarray}
Similarly, Bob can compute $K_B=B^{-1} (b'^{\delta_1}_{t_1} \cdots b'^{\delta_M}_{t_M}) = B^{-1}A^{-1}BA$, and then he knows the shared secret key by $K=K_B^{-1}$.

We note that solving the simultaneous conjugacy problem is generally not enough to break AAG protocol as the secret conjugators $A$ and $B$ must be found in the respective public subgroups $\langle a_1,\ldots,a_{N_1}\rangle$ and $\langle b_1,\ldots,b_{N_2}\rangle$. Thus, attacking the protocol requires solving the membership search problem in addition to the SCSP. The security of the AAG protocol is based on the assumption that the subgroup-restricted simultaneous conjugacy search problem is hard. For more details, see~\cite{AAG}.

\subsection{Non-commutative Diffie-Hellman: a.k.a. Ko-Lee protocol}
Originally specified by Ko, Lee, et al.~\cite{KoLee00} using braid groups, their non-commutative analogue of Diffie-Hellman key exchange can be generalized to work over other platform groups. Let $G$ be a finitely presented group, with $A,B \leq G$ such that all elements of $A$ and $B$ commute. 

An element $g\in G$ is chosen, and $g, G, A, B$ are made public. A shared secret can then be constructed as follows:
\begin{itemize}
\item Alice chooses a random element $a\in A$ and sends $g^{a}$ to Bob.
\item Bob chooses a random element $b\in B$ and sends $g^{b}$ to Alice.
\item The shared key is then $g^{ab}$, as Alice computes $(g^b)^a$, which is equal to Bob's computation of $(g^a)^b$ as $a$ and $b$ commute.
\end{itemize}

The security of Ko-Lee rests upon solving the conjugacy search problem within the subgroups $A, B$.

\subsection{Non-Commutative ElGamal Key-Exchange: a.k.a. Kahrobaei-Khan Protocol}
In the 2006 paper by the first author and Khan~\cite{KK06}, the authors proposed two adaptations of the ElGamal asymmetric key encryption algorithm for use in non-commutative groups. Let $S,T$ be finitely generated subgroups such that all elements of $S$ and $T$ commute. In any exchange, the triple $\langle G, S, T\rangle$ is made public.


\begin{itemize}
\item Bob chooses $s\in S$ as his private key, a random element $b\in G$, and publishes as his public key the tuple $\langle b, c\rangle$, with $c=b^s$.
\item To create a shared secret $x\in G$, Alice chooses $x$ and a $t\in T$. Using Bob's public key, she publishes $\langle h, E\rangle$, with $h=b^t$ and $E=x^{c^t}$.
\item To recover $x$, Bob first computes $h^s$, which, as elements of $S$ and $T$ commute, yields
$$
h^s=(b^t)^s=(b^s)^t=c^t.
$$
Bob can then calculate $x=E^{(c^t)^{-1}}$.
\end{itemize}

The security of this scheme relies upon the conjugacy search problem in $G$.

\section{Automaton Groups}
\label{sec:automaton_groups}

This section provides an overview of automaton groups. We discuss the motivation behind their various geometric descriptions and establish the notation required to analyze them.

\subsection{Trees}
Automaton groups act by automorphisms on regular rooted trees. We begin with setting up the terminology.

\begin{definition}
    A connected graph with no cycles is called a \emph{tree}. A tree $T$ is called \textit{rooted} if one of its vertices, called the \textit{root} is selected. 
\end{definition}

The trees are endowed with a combinatorial metric in which the distance between two vertices is equal to the number of edges in the shortest path connecting them.

\begin{definition}
    The \emph{$n$-th level} of a rooted tree $T$ is the set of vertices whose distance from the root is $n$.
\end{definition}

Given that trees have no cycles, each vertex $v$ at level $n$ is connected to the root by a unique path. The vertex $v'$ in this path that lies in level $n-1$ is called the \emph{parent} of $v$. The vertex $v$ in this case is called a \emph{child} of $v'$. Thus every vertex except the root has exactly one parent and may have several children. 

\begin{definition}
A rooted tree $T$ is called \emph{regular} (or $d$-regular, or $d$-ary) if there is $d\geq1$ such that each vertex of $T$ has exactly $d$ children. In case when $d=2$ such a tree is called \emph{binary}.
\end{definition}

Throughout our discussion, trees grow form top to bottom. Root is the highest vertex and the children of each vertex $v$ are located right under $v$.  We will use $T_d$ as the standard notation to denote the $d$-ary tree.  Figure~\ref{fig:binary_tree} depicts a binary tree.

\begin{definition}
The group of all graph automorphisms of the regular rooted tree $T_d$ with the operation of composition is called the \emph{full group of automorphisms of $T_d$} and is denoted $\Aut T_d$.
\end{definition}

\begin{figure}
\begin{center}
\epsfig{file=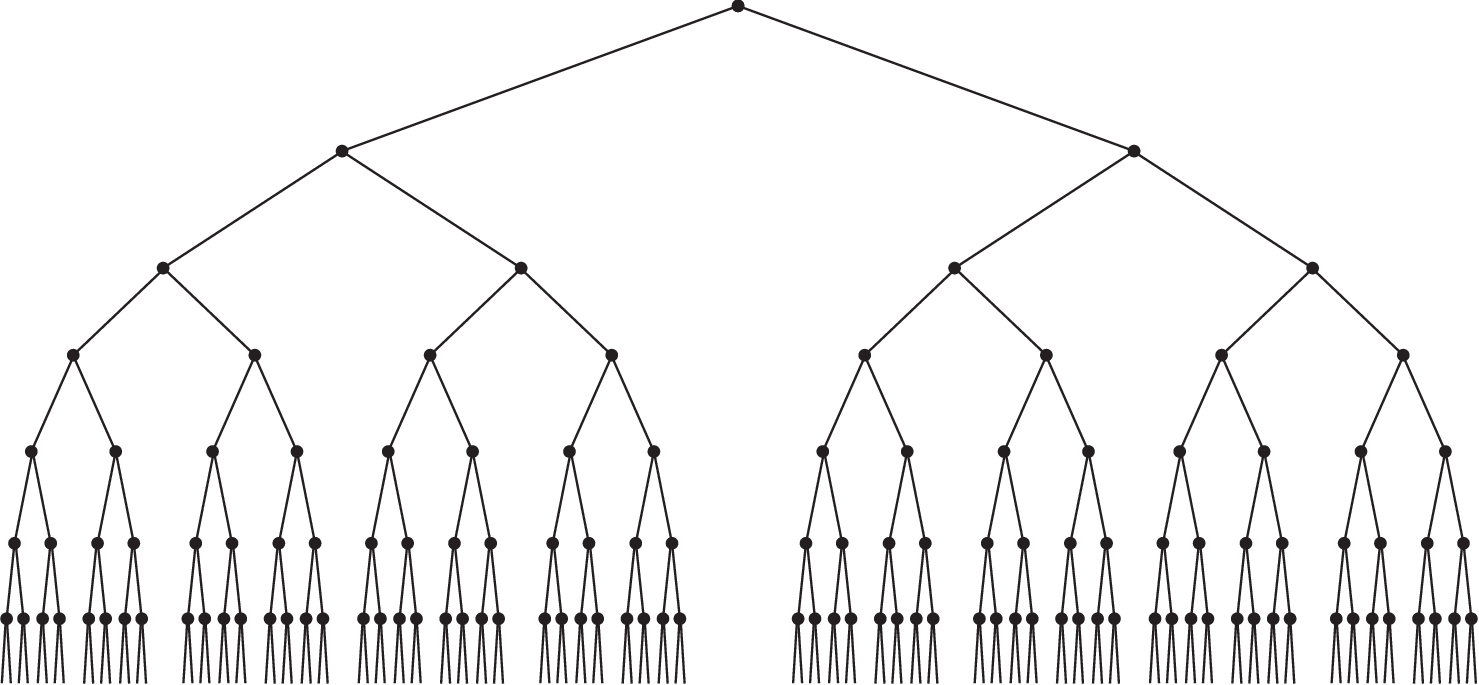,width=250pt}
\end{center}
\caption{Binary tree\label{fig:binary_tree}}
\end{figure}

We adopt the following convention throughout the paper: 
\begin{convention}[Right Actions]
\label{conv:right_actions}
    If $g$ and $h$ are elements of some group acting on a set $Y$ and $y\in Y$, then $gh(y)=h(g(y))$. In other words, we agree to consider only right actions of groups.
\end{convention}

Let $X$ be a finite set of cardinality $d$ and let $X^*$ denote the set of all finite words over $X$ (that can be thought as the free monoid generated by $X$). This set can be naturally endowed with a structure of a rooted $d$-ary tree by declaring that $v$ is adjacent to $vx$ for
any $v\in X^*$ and $x\in X$. So the vertices of $T_d$ are labeled by finite words over $X$. The empty word corresponds to the root
of the tree $T_d$ and $X^n$ corresponds to the $n$-th level of the tree. Given this interpretation, we will often write $\Aut X^*$ for $\Aut T_d$. Thus, we view  $X^*$ as a tree and consequently its automorphisms as those of a tree instead of a free monoid on $X$. 

\begin{definition}
\label{section-hom}
    For an automorphism $g$ of the tree $X^*$ and $x\in X$, we have that for any $v\in X^*$, $g(xv) = g(x)v'$ for some $v' \in X^*$. Then the map $g|_x \colon X^* \to X^*$ given by $g|_x(v) = v'$ defines an automorphism of $X^*$ called \textit{section} of $g$ at vertex $x$. Furthermore,  for any $w=x_1x_2\ldots x_n\in X^*$ we define the section $g|_w$ of $g$ at $w$ as 
\[g|_{w}=g|_{x_1}|_{x_2}\ldots|_{x_n}.\]
\end{definition}

\subsection{Automaton and self-similar groups}
Automorphisms of the tree $X^*$ can be described using the language of Mealy automata.

\begin{definition}
A \emph{Mealy automaton} (or simply \emph{automaton}) is a tuple
$(Q,X,\pi,\lambda)$, where $Q$ is a set (a set of \emph{states}), $X$ is a
finite alphabet, $\pi\colon Q\times X\to Q$ is the \emph{transition function}
and $\lambda\colon Q\times X\to X$ is the \emph{output function}. If the set
of states $Q$ is finite the automaton is called \emph{finite}. If
for every state $q\in Q$ the output function $\lambda(q,x)$ induces
a permutation of $X$, the automaton $\A$ is called \emph{invertible}.
Selecting a state $q\in Q$ produces an \emph{initial automaton}
$\A_q$.
\end{definition}

Automata are often represented by the \emph{Moore diagrams}. The
Moore diagram of an automaton $\A=(Q,X,\pi,\lambda)$ is a directed
graph in which the vertices are the states from $Q$ and the edges
have form

\begin{center}
    \begin{tikzpicture}
    \node[state] (q1) {\phantom{a}\hspace{1.6mm}$q$\hspace{1.6mm}\phantom{a}};
    \node[state, right of=q1] (q2) {\!\!$\pi(q,x)$\!\!};
    \draw (q1) edge[above] node{$x\,|\,\lambda(q,x)$} (q2);
    \end{tikzpicture}
\end{center}

\noindent for $q\in Q$ and $x\in X$. If the automaton is invertible, then it is
common to label vertices of the Moore diagram by the permutation
$\lambda(q,\cdot)$ and leave just first components from the labels of the edges. Figure \ref{fig:aut_grig} shows Moore diagram for automaton generating the Grigorchuk group $\G$. 

The action of an initial automaton $\A_q$ on $X^*$ is informally defined as follows. Given a word
$v=x_1x_2x_3\ldots x_n\in X^*$, the automaton $\A_q$ scans its first letter $x_1$ and outputs $\lambda(q,x_1)$. The rest of the word is handled in a similar fashion by the initial automaton $\A_{\pi(q,x_1)}$. Formally speaking, the functions $\pi$ and $\lambda$ can be extended to $\pi\colon
Q\times X^*\to Q$ and $\lambda\colon  Q\times X^*\to X^*$ via
\[\begin{array}{l}
\pi(q,x_1x_2\ldots x_n)=\pi(\pi(q,x_1),x_2x_3\ldots x_n),\\
\lambda(q,x_1x_2\ldots x_n)=\lambda(q,x_1)\lambda(\pi(q,x_1),x_2x_3\ldots x_n).\\
\end{array}
\]

The invertible initial automata define automorphisms of $X^*$. Moreover, any automorphism $g$ of $X^*$ can be encoded by the action of an invertible initial automaton with the following construction. Define an initial automaton $\A_g$ whose action on $X^*$ coincides with that of $g$ as follows. The set of states of $\A_g$ is the set $\{g|_v\colon  v\in X^*\}$ of different sections of $g$ at the vertices of the tree. The transition and output functions are defined by
\begin{equation}
\label{eqn:defn_homo}
\begin{array}{l}
\pi(g|_v,x)=g|_{vx},\\
\lambda(g|_v,x)=g|_v(x),
\end{array}
\end{equation}
and the initial state of $\A_g$ is $g$. We note that $\A_g$ does not need to be finite.

\begin{definition}
The group generated by all states of an invertible automaton $\A$ is
called an \emph{automaton group} and
denoted by $\mathbb G(\A)$.
\end{definition}

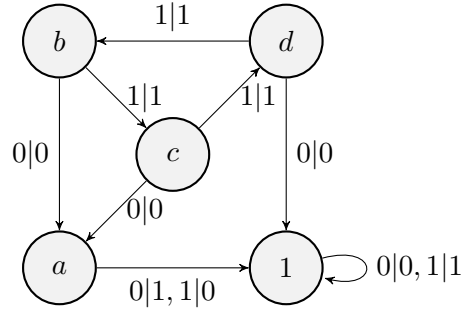
\begin{figure}
    \centering
    \begin{tikzpicture}

    \node[state] (q1) {$a$};
    \node[state, above of=q1] (q2) {$b$};
    \node[state, right of=q2] (q4) {$d$};
    \node[state, below of=q4] (q0) {$1$};
    \node[state] (q3) at (1.5, 1.5) {$c$};

    \draw (q1) edge[below] node{$0|1 , 1|0$} (q0);
    \draw (q2) edge[left] node{$0|0$} (q1);
    \draw (q2) edge[right] node{$1|1$} (q3);
    \draw (q4) edge[above] node{$1|1$} (q2);
    \draw (q4) edge[right] node{$0|0$} (q0);
    \draw (q3) edge[right] node{$0|0$} (q1);
    \draw (q3) edge[right] node{$1|1$} (q4);
    \draw (q0) edge[loop right] node{$0|0 , 1|1$} (q0);
    
    \end{tikzpicture}
\caption{Five-state Automaton $\A_\G$ generating the Grigorchuk group $\G$}
    \label{fig:aut_grig}
\end{figure}

Figure \ref{fig:aut_grig} shows the five-state automaton $\A_\G$ that generates the Grigorchuk group $\G$ acting on $X^*$ for $X = \{0,1\}$.

Often the language of self-similar groups is  used to describe automorphisms of a tree $X^*$ as well.

\begin{definition}
A subset $S$ of $\Aut X^*$ is \textit{self-similar} if all sections (as defined in Definition~\ref{section-hom}) of each element in $S$ are in $S$ again. A group $G<\Aut X^*$ is called \textit{self-similar} if it is self-similar as a set.
\end{definition}

A self-similar set $S$ has a structure of an automaton with transitions defined by:
\begin{center}
    \begin{tikzpicture}
    \node[state] (q1) {\phantom{!}$g$\phantom{!}};
    \node[state, right of=q1] (q2) {\!\!$g|_x$\!\!};
    \draw (q1) edge[above] node{$x\,|\,g(x)$} (q2);
    \end{tikzpicture}
\end{center}
for each $g\in S$ and $x\in X$.

An automaton group can be naturally defined as a group generated by a self-similar set. On the other hand, the whole self-similar group $G$ defines an automaton $\A(G)$ whose states are elements of the group.

The following trivial observation relates self-similar and automaton groups.

\begin{proposition}
A group $G$ is self-similar if and only if it is generated by (possibly infinite) automaton.
\end{proposition}

The group $\Aut(X^*)$ is naturally isomorphic to $\Aut(X^*)\wr\Sym(d)$ via the map
\begin{equation}
\label{eqn:wreath_embedding}
G\ni g\mapsto (g|_0,g|_1,\ldots,g|_{d-1})\sigma_g\in G\wr \Sym(X),
\end{equation}
where $\sigma_g\in\Sym(d)$ is the permutation of $X$ induced by the action of $g$ on the first level of $X^*$. We will often identify $g$ with its image under this embedding and write
\[g = (g|_0,g|_1,\ldots,g|_{d-1})\sigma_g.\]
The right-hand side of the last expression is called the \emph{wreath recursion} of $g$. If $\sigma_g$ is the trivial permutation, it is customary to omit it in the wreath recursion.

The language of wreath recursion allows us to do effective calculations with elements of $\Aut(X^*)$. We will use the Right Action Convention~\ref{conv:right_actions} for formulas below. Left actions are also used quite often in the literature in this context, so it is important to clearly state what notation we use. The translation from right to left actions consists just in reversing the order of all multiplications. We choose the right actions here to be compatible with the right actions of permutations in a symmetric group in GAP~\cite{GAP4}.

Now, if
$g=(g_0,g_1,\ldots,g_{d-1})\sigma_g$ and $h= (h_1,h_2,\ldots,h_d)\sigma_h$ are two elements of $\Aut X^*$, then by the definition of the wreath product we get:
\begin{equation}
\label{eqn:product}
gh=\left(g_0h_{\sigma_g(0)},g_1h_{\sigma_g(1)},\ldots,g_{d-1}h_{\sigma_g(d-1)}\right)\sigma_g\sigma_h
\end{equation}
and
\begin{equation}
\label{eqn:inverse}
g^{-1}=\left(g^{-1}_{\sigma_g^{-1}(0)},g^{-1}_{\sigma_g^{-1}(1)},\ldots,g^{-1}_{\sigma_g^{-1}(d-1)}\right)\sigma_g^{-1}.
\end{equation}
The last two equalities for the basis for effective calculations in groups generated by automata.

\subsection{Variants of conjugacy problem in automaton groups}

The study of the conjugacy problem in automaton groups goes back to pioneering works of Leonov~\cite{leonov:conjugacy} and Rozhkov~\cite{rozhkov:conjugacy}, who proved that the conjugacy problem is decidable in Grigorchuk group $\G$. These ideas were used by Grigorchuk and Wilson to construct an algorithm solving the conjugacy problem in the certain class of branch groups~\cite{grigorch_w:conjugacy}. Later polynomial time algorithm solving the conjugacy problem in $\G$ was constructed by Lysenok, Myasnikov, and Ushakov~\cite{lysenok_mu:conjugacy_problem_in_grigorchuk_group_polynomial10} and very recently a linear time algorithm was developed by Modi, Seedhom, and Ushakov in~\cite{modi_su:linear_time_algorithm_conjugacy_grigorchuk21}. The techniques used in the aforementioned papers in most cases allow to solve also the conjugacy search problem. For example, the linear-time algorithm from~\cite{modi_su:linear_time_algorithm_conjugacy_grigorchuk21} can be modified to produce a conjugator in polynomial time.

One more large group for which the conjugacy problem is decidable is the group $B_d$ of all bounded automata over a $d$-letter alphabet~\cite{bondarenko_bsz:conjugacy}. This group, as well as the notion of polynomially and exponentially growing automata, was introduced by Sidki~\cite{sidki:nofree} who proved that finite automata of polynomial growth generate groups that do not contain free nonabelian subgroups. Note that group $B_d$ is very large and not finitely generated, and the solution of the conjugacy problem in $B_d$ does not yield a solution of this problem in any subgroup of $B_d$.

In connection to the cryptographic applications the simultaneous conjugacy search problem and its variants are the most relevant. To the best of our knowledge, there are no results on SCSP in automaton groups. The exact variant of SCSP depends on the cryptosystem used.


\section{Contracting Groups}
\label{sec:contracting}
In this section we provide an overview of the class of contracting groups. In particular we recall the notion of a \emph{nucleus portrait} of an element of a contracting group that will play a role of the normal form of an element and that we propose to use as the way to transmit elements of these groups through communication channels. 
We refer the reader to~\cite{nekrash:self-similar} for details.


\subsection{Basic definitions and properties}

\begin{definition}
A self-similar group $G <\Aut X^* $ is called \textit{contracting} if there exists a finite set $\N \subset G$ such that for each $g \in G$ there is a level $l \geq 0$ so that for all $ k \geq l $ we have $g|_v \in \N$ for all $v \in X^k$. The minimal set $\N$ is called the \textit{nucleus} of $G$. Any finite subset of $G$ containing the nucleus is called \textit{quasinucleus}.
\end{definition}

The following proposition can be considered as a second definition of contracting groups (see \cite[Lemma~2.11.12]{nekrash:self-similar}) and justifies the term ``contracting''. For an element $g$ of a group $G$ generated by a set $S$ we will denote by $|g|_S$ the length of $g$ with respect to $S$, i.e., the length of the shortest word in $S^{\pm1}$ representing $g$.

\begin{proposition}
    \label{prop:contracting_length}
A self-similar group $G$ generated by a self-similar generating set $S$ is contracting if and only if there exist $C>0$ and level $N\geq 1$ such that for any element $g$ of $G$
\[\big|g|_v\big|_S<\frac12|g|_S+C\]
for each vertex $v\in X^N$.
\end{proposition}

For the first time contracting type arguments were applied by Grigorchuk in his construction of the first Burnside group of intermediate growth (a.k.a. Grigorchuk group) $\G$~\cite{grigorch:burnside,grigorch:degrees}. It is contracting with the nucleus consisting of five elements $\N = \{ 1,a,b,c,d \}$. We note that in this case the nucleus coincides with the set of states of the automaton generating $\G$, but in general it need not be the case. The term ``contracting'' was coined by Nekrashevych~\cite{nekrash:self-similar}, who developed the theory of Iterated Monodromy Groups, in which contracting groups play a key role.

Every element $g$ of $G$ can be uniquely described by its nucleus portrait that we define below. By definition, there is some minimal level $l \geq 0$ so that for all $k \geq l$ we have $g|_v \in \N $ for all $v \in X^k$. Therefore, we can identify $g$ with a finite tree of depth $l$ whose vertices are labeled by permutations of $X$ and leaves are labeled by the elements of the nucleus.

\begin{definition}
\label{def:portrait}
    Let $g$ be an element of a contracting group $G$ acting on a $d$-ary rooted tree $X^*$ with the nucleus $\mathcal N$. The \emph{nucleus portrait} $\pi(g)$ of $g \in G$ is a labeled tree obtained recurrently as follows. If $g\in\N$ then the nucleus portrait of $g$ is a single vertex labeled by $g$. Otherwise the nucleus portrait of $g$ starts with a single vertex labeled by the permutation induced by $g$ on the first level of $X^*$ that is connected to $d$ children on the first level that are roots of the portraits of $g|_x$ for $x\in X$.     
\end{definition}

Note that the contracting property of the groups guarantees that the nucleus portrait of every element is finite. If two words in the generators of a contracting group give rise to two identical nucleus portraits then they represent the same element of the group. 

When we will describe multiplication of nucleus portraits in Subsection~\ref{ssec:using_contracting} we will also need a bit more general notion of a \emph{portrait} of an element of a self-similar group, which is defined  similarly to the nucleus portrait, but without the requirements that the leaves must be labeled by the elements of the nucleus and that the development of the portrait must stop once an element of the nucleus is reached. In other words, a portrait of $g$ is any finite subtree of $X^*$ containing the root whose internal vertices are labeled by the permutations induced by $g$ at these vertices, and whose leaves are labeled by the sections of $g$ at the corresponding vertices. A portrait of $g$ is called an \emph{extension of the nucleus portrait} if all of its leaves are labeled by the elements of the nucleus. Such extensions are obtained from nucleus portraits by a finite number of operations that add children to a terminal node labeled by an element $n$ of the nucleus, label those children by sections of $n$, and replace the label $n$ of the original node by the permutation $\sigma_n$ induced by $n$ at the first level. Unlike in the case of nucleus portraits, there are many portraits (or extended nucleus portraits) representing the same element of a group.

One of the main computational advantages of contracting groups is that the word problem can be decided in polynomial time~\cite{nekrash:self-similar} (with the main idea of the algorithm consisting in comparing the nucleus portraits of the input words). Additionally, one can bound the degree of a polynomial solving the word problem by $(|\N|^2-1)\log_2d+\varepsilon$ as shown in~\cite{savchuk:wp}.

\begin{definition}
    For an element $g$ of a contracting group $G$ the \emph{depth} of $g$, denoted by $\partial(g)$, is the number of levels of $\pi(g)$). The \emph{portrait boundary size} of $g$, denoted by $s(g)$, is the number of leaves of $\pi(g)$.
\label{portrait-depth-def}
\end{definition}

The notion of the depth of an element was introduced by Sidki in~\cite{sidki:on_a_2-generated_infinite_3-group_the_subgroups_and_automorphisms87}. The nucleus portraits were introduced for the first time by Grigorchuk in~\cite{grigorch:solved} under the name of \emph{core portraits}. Additionally, Grigorchuk defined the portrait growth function and provided some bounds for this function for Grigorchuk group. \v Suni\'c and Uria-Albizuri in a recent paper~\cite{sunic_u:portrait_growth23} derived asymptotic behavior of this function for Grigorchuk group (that turns out to be double-exponential) and obtained recursive formulas for the portrait growth of any finitely generated contracting regular branch group.


If $G$ is a contracting group generated by a finite set $S$, there is a direct relation between the length $|g|_S$ of $g\in G$ with respect to $S$ and both $\partial(g)$ and $s(g)$.

\begin{proposition}
\label{prop:log_depth}
Suppose $G$ is a contracting group generated by a finite self-similar set $S$ containing the nucleus of $G$. Then there is $R>0$ such that for each $g\in G$ 
\[\partial(g)\leq \ceil{R\log_2 |g|_S+R}.\]
\end{proposition}

\begin{proof}
By applying Proposition~\ref{prop:contracting_length} $\ceil{\log_2|g|_S}$ times we obtain that for any vertex $v$ of level $\ceil{N\log_2|g|_S}$:
\[\big|g|_v\big|_S<\frac1{2^{\ceil{\log_2|g|_S}}}|g|_S+\sum_{i=0}^{\ceil{\log_2|g|_S}-1}\frac{C}{2^i}\leq 1+\sum_{i=0}^{\infty}\frac{C}{2^i}=1+2C.\]
Let 
\[M=\max\{\partial(h)\colon h\in G, |h|_s\leq 1+2C\}\] 
be the level at which all elements in $G$ of length up to $1+2C$ contract to the nucleus. Then $g$ will contract to the nucleus at level $\ceil{N\log_2|g|_S}+M=\ceil{N\log_2|g|_S+M}$. Thus, we can choose $R=\max\{N,M\}$.
\end{proof}

\begin{corollary}
\label{cor:log_depth}
Suppose $G$ is a contracting group generated by a finite self-similar set $S$ containing the nucleus of $G$. Then the portrait boundary size of an element $g\in G$ is at most polynomial in $|g|_S$.
\end{corollary}
\begin{proof}
The portrait boundary size $s(g)$ of $g\in G$ is bounded from above by the number of leaves of the $d$-ary tree of depth $\partial(g)$, so we get by Proposition~\ref{prop:log_depth} that there is $R>0$ such that 
\begin{multline*}s(g)\leq d^{\partial(g)}\leq d^{\ceil{R\log_2 |g|_S+R}}\leq d^{R\log_2 |g|_S+R+1}\\
=d^{R+1}d^{\frac{\log_d|g|_S^R}{\log_d2}}=d^{R+1}d^{\log_d|g|_S^{R\log_2d}}=d^{R+1}|g|_S^{R\log_2d}.
\end{multline*}
\end{proof}

The last corollary shows that the amount of information required to transmit the portrait is polynomial in the length of a word representing that element. However, in most cases the actual contracting depth of an element is much smaller than the theoretical bound from Proposition~\ref{prop:log_depth} and the number of the nodes in the portraits is very similar (and in many cases, like in Grigorchuk group is actually much smaller) than the length of the word defining the element.

We have also computed, in the case of Grigorchuk group, the average portrait depths and the ratios of portrait boundary sizes to the number of vertices at the contracting depth levels. More precisely, given a contracting group $G$ acting on a $d$-ary tree, for $g \in G$, we define the ratio $\mathcal{R}_g = \frac{\displaystyle s(g) }{\displaystyle d^{\partial(g)}} $. Table~\ref{ratio-grigrochuk} displays variation of average $\mathcal{R}_g$ and average portrait depth $\partial(g)$ against the reduced word length for randomly generated elements of Grigorchuk group. Each average value is calculated using 100 randomly generated elements for the respective word length.


\hspace{1cm}
\begin{table}[h]
\centering
\begin{tabular}{|p{4.0 cm}||c|c|c|c|c|c|c|}
    \hline
    Group & \multicolumn{7}{|c|}{Grigorchuk group} \\
    \hline
    Reduced word length & 100 & 200 & 500 & 1000 & 2000 & 5000 & 10000 \\
    \hline
    Average ratio $\mathcal{R}_g$
&0.47 
&0.41 
&0.38
&0.35
&0.33
&0.30
&0.32\\
    \hline 
    Average portrait depth 
&5.49
&6.18
&6.97
&7.64
&8.11
&8.95
&9.42 \\

    \hline
\end{tabular}
\caption{Variation in average $\mathcal{R}_g$ and average $\partial(g)$ against reduced word length for Grigorchuk group. Each average is calculated using 100 randomly generated elements for the respective word length. } 
\label{ratio-grigrochuk}
\end{table}


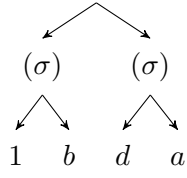
\begin{figure}
    \centering
\begin{forest}
   [ [$(\sigma)$ [1]
               [$b$]
                   ]   
            [$(\sigma)$ [$d$][$a$]]]  
\end{forest}
\caption{Nucleus portrait for $acab \in \G$}
\label{fig:nucleus portrait example}
\end{figure}


\subsection{Computations with nucleus portraits}
\label{ssec:using_contracting}

In this section, we show that the nucleus portraits can be used effectively for computations in contracting groups. Following a description of performing basic computations using nucleus portraits, we provide a specific simple example of deploying contracting groups as platforms for AAG cryptosystem. Each step of the key-generation process is demonstrated using examples from Grigorchuk group. The described routines are implemented in the GitHub repository~\cite{kahrobaei_ms:lba_github}.

Throughout this section let $G$ be a contracting group with nucleus $\N$ and let $g,h\in G$ be arbitrary elements given by their portraits $\pi(g)$ and $\pi(h)$. Below we show how to construct the portraits of $gh$ and $g^{-1}$.


For our computations we will represent the nucleus portraits of elements of contracting groups as nested lists defined as follows (in accordance with Definition~\ref{def:portrait} and as implemented in \verb|AutomPortrait| function from the \verb|AutomGrp| package~\cite{muntyan_s:automgrp}): 
\begin{itemize}
    \item If $g \in \N $ then
    
    \verb|AutomPortrait|$(g)$ = $[g]$,
    
    \item Otherwise, for $g =(g_1,\ldots,g_d)\sigma$,
    
    \verb|AutomPortrait|$(g)$ = [$\sigma$, \verb|AutomPortrait|$(g_1), \ldots,  \verb|AutomPortrait|(g_d)$].
\end{itemize}

The recursive algorithms for computing the product and the inverse of nucleus portraits are given in Algorithm~\ref{alg:portrait_product} and Algorithm~\ref{alg:portrait_inverse}, respectively. In the following, we give their informal descriptions and obtain bounds for their complexity.

\subsubsection*{Algorithm 1: Multiplication of two nucleus portraits $p=\pi(g)$ and $q=\pi(h)$.}
    \begin{enumerate}
        \item Precompute the decompositions of nucleus elements on the first level of the tree.
        \item Precompute nucleus portraits for the pairwise products $n_1n_2$ for $n_1,n_2\in\N$.
        \item If both $p$ and $q$ have length 1, so that $p=[g], q=[h]$, $g,h\in\N$, then $\pi(gh)$ was precomputed in the previous step.
        \item If exactly one of $p$ and $q$ has length 1, say, $p=[g]$ for $g\in\N$, then replace it with its decomposition computed in Step~(a): 
        \[\big[\sigma_g,[g|_1],[g|_2],\ldots,[g|_d]\big],\] 
        where $\sigma_g$ is the permutation of $X$ induced by $g$, and $[g|_i]$'s are nucleus portraits of the sections of $g$ (which also belong to $\N$). Then we can assume that the portraits of $p$ and $q$ have a form
        \[\begin{array}{l}
        p=[\sigma_g,p_1,p_2,\ldots,p_d],\\
        q=[\sigma_h,q_1,q_2,\ldots,q_d],\\
        \end{array}\]
        where $p_i$'s and $q_i$'s are the portraits of the sections of $g$ and $h$, respectively, on the first level.
        \item Compute the portrait (not necessarily the nucleus portrait) of $gh$ recursively as
        \begin{multline*}
            \textsc{Product}(p,q)\\=[\sigma_g\sigma_h, \textsc{Product}(p_1,q_{\sigma_g(1)}), \textsc{Product}(p_2,q_{\sigma_g(2)}),\ldots,\textsc{Product}(p_d,q_{\sigma_g(d)})].
        \end{multline*}
        \item (Portrait pruning) The portrait obtained in the previous step may not be the nucleus portrait of $gh$ as it may require some \emph{pruning}: For each tuple of $d$ consecutive leaves $(n_1,n_2,\ldots,n_d)$ whose labels are in $\N$ that are children of the same vertex $v$ check if there is an element $n\in\N$ whose wreath recursion at the first level coincides with $(n_1,n_2,\ldots,n_d)\sigma_v$, where $\sigma_v$ is the permutation of $gh$ at vertex $v$. If this is the case, delete leaves $(n_1,n_2,\ldots,n_d)$ from the portrait of $gh$ and label the new leaf at $v$ by $n$.
        \item Repeat the last step while possible.
        \item The obtained portrait is the nucleus portrait $\pi(gh)$ of $gh$.
    \end{enumerate}
\subsubsection*{Algorithm 2: Inverse of nucleus portrait $p=\pi(g)$ (see Figure~\ref{fig:portrait_inverse} for illustration).}
    \begin{enumerate}
        \item Act on the portrait of $p$ by permutations that are the labels of the internal vertices. Namely, at each internal vertex $v$ of $p$ labeled by a permutation $\sigma$ permute the portraits attached to this vertex according to this permutation.
        \item Replace all the labels (both internal labels and the labels of the leafs) by their inverses. Note that the nucleus is always closed under taking inverses since it is the self-similar closure of all the cycles in the automaton of the whole group.
        \item The obtained portrait is the nucleus portrait $\pi(g^{-1})$ of $g^{-1}$.
    \end{enumerate}

For complexity estimates we start from an easy lemma whose proof we leave as an exercise.

\begin{lemma}
\label{lem:non-terminal-nodes}
    Let $T$ be a finite rooted tree in which every vertex has either $d$ or $0$ children and let $l$ and $i$ denote the numbers of its terminal (with 0 children) and non-terminal (with $d$ children) nodes, respectively. Then $i = \frac{l-1}{d-1}$.
\end{lemma}

First of all, this lemma gives a natural way to compute the memory required to store a nucleus portrait.

\begin{corollary}
\label{cor:memory}
    Let $G$ be a contracting group with nucleus $\N$ acting on $T_d$. The amount of memory (in bits) required to store the contracting portrait $\pi(g)$ of $g\in G$ with the boundary size $s(g)$ is equal to 
    \begin{equation}
        \label{eqn:portrait_size}
        \left(\lceil\log_2|\N|\rceil+\frac{d(\lceil\log_2d\rceil+2)}{d-1}\right)s(g)+o(s(g)),\quad s(g)\to\infty.
    \end{equation}
\end{corollary}

\begin{proof}
    We observe that the amount of memory required for a portrait with boundary size $s(g)$ is given by
    \begin{equation}
        \label{eqn:portrait_size2}
        s(g)\lceil\log_2|\N|\rceil+\frac{s(g)-1}{d-1}d\lceil\log_2d\rceil+2\left(\frac{ds(g)-1}{d-1}\right)+o\left(\frac{ds(g)-1}{d-1}\right),\quad s(g)\to\infty.
    \end{equation}

    Indeed, the first term in expression~\eqref{eqn:portrait_size2} is responsible for storing the labels of $s(g)$ leaves of the portrait (which come from the set $\N$). The second term comes for storing the labels of $\frac{ds(g)-1}{d-1}$ internal vertices of the portrait (coming from the set $\Sym(d)$). Finally, the last 2 terms come from the balanced parenthesis encoding of a rooted tree with $\frac{ds(g)-1}{d-1}$ vertices.

    The expression~\eqref{eqn:portrait_size} is a trivial simplification of~\eqref{eqn:portrait_size2}. 
\end{proof}

For practical purposes, we provide the empirical relation between the length of the word representing an element of a contracting group and the average contracting portrait boundary size, as well as the average time required to compute these portraits. Figure~\ref{fig:grigorchuk_hist} describes the results for the Grigorchuk group, Figure~\ref{fig:basilica_hist} for the Basilica group, and Figure~\ref{fig:7basilica_hist} for the 7-Basilica group  (all for words of lengths 500, 1000, and 5000).

\tikzset{
        -,  
        >=stealth',
        node distance=3cm, 
        every state/.style={thick, fill=gray!10}, 
        initial text=$ $,
        }

\begin{figure}[H]
\centering
Word length 500:\\
\begin{minipage}[t]{9cm}
\vspace*{0mm}
\begin{tikzpicture}
\begin{axis}[
    title={},
    xlabel={Boundary size},
    ylabel={Frequency},
    ymajorgrids=true,
    grid style={dashed,gray!40},
    width=9cm,
    height=6.3cm,
    ymin=0,
    xtick distance=25,
    ytick distance=5,
]

\addplot[
    ybar,
    hist={
        bins=10
    },
    fill=blue!60,
    draw=black
]         table [row sep=\\,y index=0] {
            data\\
40\\ 47\\ 54\\ 36\\ 54\\ 41\\ 45\\ 45\\ 43\\ 41\\ 37\\ 55\\ 48\\ 50\\ 43\\ 36\\ 57\\ 43\\ 46\\ 42\\ 68\\ 51\\
35\\ 47\\ 50\\ 48\\ 52\\ 50\\ 53\\ 50\\ 40\\ 35\\ 49\\ 46\\ 46\\ 49\\ 44\\ 50\\ 49\\ 40\\ 50\\ 52\\ 56\\ 42\\
49\\ 34\\ 34\\ 41\\ 45\\ 55\\ 50\\ 57\\ 51\\ 59\\ 45\\ 61\\ 49\\ 55\\ 57\\ 50\\ 54\\ 46\\ 48\\ 48\\ 52\\ 41\\
49\\ 40\\ 44\\ 46\\ 54\\ 41\\ 39\\ 39\\ 45\\ 59\\ 49\\ 49\\ 54\\ 40\\ 45\\ 43\\ 48\\ 50\\ 50\\ 57\\ 47\\ 45\\
35\\ 49\\ 52\\ 52\\ 51\\ 46\\ 49\\ 40\\ 42\\ 50\\ 54\\ 48\\
    };
\end{axis}
\end{tikzpicture}
\end{minipage}
\begin{minipage}[t]{7cm}
\vspace*{0mm}
\begin{tikzpicture}
\begin{axis}[
    ytick={1},
    yticklabels={Boundary size\phantom(},
    yticklabel style = {align=center, font=\small, rotate=90},
    height=3cm,
    width=7cm,
    grid=major,
    grid style={dashed, gray!30}
    ]
    \addplot+ [boxplot, fill=blue!60]
        table [row sep=\\,y index=0] {
            data\\
40\\ 47\\ 54\\ 36\\ 54\\ 41\\ 45\\ 45\\ 43\\ 41\\ 37\\ 55\\ 48\\ 50\\ 43\\ 36\\ 57\\ 43\\ 46\\ 42\\ 68\\ 51\\
35\\ 47\\ 50\\ 48\\ 52\\ 50\\ 53\\ 50\\ 40\\ 35\\ 49\\ 46\\ 46\\ 49\\ 44\\ 50\\ 49\\ 40\\ 50\\ 52\\ 56\\ 42\\
49\\ 34\\ 34\\ 41\\ 45\\ 55\\ 50\\ 57\\ 51\\ 59\\ 45\\ 61\\ 49\\ 55\\ 57\\ 50\\ 54\\ 46\\ 48\\ 48\\ 52\\ 41\\
49\\ 40\\ 44\\ 46\\ 54\\ 41\\ 39\\ 39\\ 45\\ 59\\ 49\\ 49\\ 54\\ 40\\ 45\\ 43\\ 48\\ 50\\ 50\\ 57\\ 47\\ 45\\
35\\ 49\\ 52\\ 52\\ 51\\ 46\\ 49\\ 40\\ 42\\ 50\\ 54\\ 48\\
    };
\end{axis}
\end{tikzpicture}\\[0.4cm]
\begin{tikzpicture}
\begin{axis}[
    ytick={1},
    yticklabels={Time (ms)},
    yticklabel style = {align=center, font=\small, rotate=90},
    height=3cm,
    width=7cm,
    grid=major,
    grid style={dashed, gray!30}
    ]
    \addplot+ [boxplot, fill=blue!60]
        table [row sep=\\,y index=0] {
            data\\
78\\ 94\\ 94\\ 78\\ 156\\ 94\\ 125\\ 93\\ 94\\ 78\\ 78\\ 79\\ 109\\ 94\\ 140\\ 78\\ 94\\ 141\\ 109\\ 110\\
187\\ 125\\ 63\\ 125\\ 171\\ 79\\ 109\\ 109\\ 141\\ 141\\ 109\\ 109\\ 125\\ 141\\ 109\\ 94\\ 94\\ 78\\ 109\\
94\\ 78\\ 110\\ 109\\ 78\\ 94\\ 78\\ 94\\ 78\\ 94\\ 78\\ 78\\ 140\\ 110\\ 109\\ 125\\ 110\\ 125\\ 93\\ 78\\
94\\ 94\\ 94\\ 140\\ 125\\ 94\\ 141\\ 125\\ 109\\ 109\\ 157\\ 156\\ 78\\ 63\\ 62\\ 94\\ 94\\ 62\\ 109\\ 125\\
94\\ 125\\ 94\\ 78\\ 109\\ 79\\ 140\\ 156\\ 110\\ 125\\ 109\\ 94\\ 109\\ 110\\ 78\\ 94\\ 78\\ 78\\ 94\\ 109\\
125\\
    };
\end{axis}
\end{tikzpicture}
\end{minipage}\\
\hrulefill\\
Word length 1000:\\
\begin{minipage}[t]{9cm}
\vspace*{0mm}
\begin{tikzpicture}
\begin{axis}[
    title={},
    xlabel={Boundary size},
    ylabel={Frequency},
    ymajorgrids=true,
    grid style={dashed,gray!40},
    width=9cm,
    height=6.3cm,
    ymin=0,
    xtick distance=25,
    ytick distance=5,
]

\addplot[
    ybar,
    hist={
        bins=10
    },
    fill=blue!60,
    draw=black
]         table [row sep=\\,y index=0] {
            data\\
62\\ 86\\ 80\\ 66\\ 68\\ 65\\ 61\\ 48\\ 79\\ 67\\ 65\\ 75\\ 76\\ 74\\ 58\\ 64\\ 70\\ 61\\ 64\\ 67\\ 67\\ 67\\
62\\ 65\\ 74\\ 67\\ 83\\ 56\\ 72\\ 61\\ 69\\ 69\\ 52\\ 62\\ 56\\ 86\\ 50\\ 61\\ 67\\ 71\\ 62\\ 69\\ 53\\ 60\\
73\\ 60\\ 70\\ 68\\ 62\\ 78\\ 67\\ 81\\ 84\\ 59\\ 71\\ 72\\ 48\\ 67\\ 62\\ 59\\ 61\\ 52\\ 57\\ 76\\ 55\\ 58\\
67\\ 68\\ 55\\ 54\\ 75\\ 57\\ 75\\ 66\\ 63\\ 54\\ 57\\ 65\\ 57\\ 68\\ 55\\ 76\\ 69\\ 58\\ 57\\ 75\\ 74\\ 65\\
71\\ 79\\ 69\\ 67\\ 75\\ 65\\ 65\\ 55\\ 78\\ 64\\ 75\\ 79\\
    };
\end{axis}
\end{tikzpicture}
\end{minipage}
\begin{minipage}[t]{7cm}
\vspace*{0mm}
\begin{tikzpicture}
\begin{axis}[
    ytick={1},
    yticklabels={Boundary size\phantom(},
    yticklabel style = {align=center, font=\small, rotate=90},
    height=3cm,
    width=7cm,
    grid=major,
    grid style={dashed, gray!30}
    ]
    \addplot+ [boxplot, fill=blue!60]
        table [row sep=\\,y index=0] {
            data\\
62\\ 86\\ 80\\ 66\\ 68\\ 65\\ 61\\ 48\\ 79\\ 67\\ 65\\ 75\\ 76\\ 74\\ 58\\ 64\\ 70\\ 61\\ 64\\ 67\\ 67\\ 67\\
62\\ 65\\ 74\\ 67\\ 83\\ 56\\ 72\\ 61\\ 69\\ 69\\ 52\\ 62\\ 56\\ 86\\ 50\\ 61\\ 67\\ 71\\ 62\\ 69\\ 53\\ 60\\
73\\ 60\\ 70\\ 68\\ 62\\ 78\\ 67\\ 81\\ 84\\ 59\\ 71\\ 72\\ 48\\ 67\\ 62\\ 59\\ 61\\ 52\\ 57\\ 76\\ 55\\ 58\\
67\\ 68\\ 55\\ 54\\ 75\\ 57\\ 75\\ 66\\ 63\\ 54\\ 57\\ 65\\ 57\\ 68\\ 55\\ 76\\ 69\\ 58\\ 57\\ 75\\ 74\\ 65\\
71\\ 79\\ 69\\ 67\\ 75\\ 65\\ 65\\ 55\\ 78\\ 64\\ 75\\ 79\\
    };
\end{axis}
\end{tikzpicture}\\[0.4cm]
\begin{tikzpicture}
\begin{axis}[
    ytick={1},
    yticklabels={Time (ms)},
    yticklabel style = {align=center, font=\small, rotate=90},
    height=3cm,
    width=7cm,
    grid=major,
    grid style={dashed, gray!30}
    ]
    \addplot+ [boxplot, fill=blue!60]
        table [row sep=\\,y index=0] {
            data\\
234\\ 329\\ 343\\ 547\\ 266\\ 328\\ 328\\ 250\\ 375\\ 328\\ 328\\ 297\\ 391\\ 312\\ 344\\ 297\\ 328\\ 266\\
343\\ 313\\ 281\\ 250\\ 219\\ 344\\ 281\\ 234\\ 407\\ 359\\ 312\\ 219\\ 375\\ 344\\ 297\\ 250\\ 234\\ 360\\
390\\ 203\\ 328\\ 235\\ 250\\ 234\\ 281\\ 235\\ 219\\ 218\\ 282\\ 375\\ 203\\ 265\\ 250\\ 313\\ 375\\ 281\\
375\\ 297\\ 312\\ 235\\ 265\\ 204\\ 218\\ 219\\ 266\\ 250\\ 187\\ 219\\ 234\\ 235\\ 218\\ 204\\ 250\\ 312\\
266\\ 234\\ 234\\ 235\\ 219\\ 359\\ 234\\ 344\\ 281\\ 328\\ 219\\ 250\\ 219\\ 234\\ 297\\ 266\\ 250\\ 312\\
219\\ 219\\ 234\\ 422\\ 313\\ 296\\ 282\\ 234\\ 328\\ 313\\
    };
\end{axis}
\end{tikzpicture}
\end{minipage}\\
\hrulefill\\
Word length 5000:\\
\begin{minipage}[t]{9cm}
\vspace*{0mm}
\begin{tikzpicture}
\begin{axis}[
    title={},
    xlabel={Boundary size},
    ylabel={Frequency},
    ymajorgrids=true,
    grid style={dashed,gray!40},
    width=9cm,
    height=6.3cm,
    ymin=0,
    xtick distance=25,
    ytick distance=5,
]

\addplot[
    ybar,
    hist={
        bins=10
    },
    fill=blue!60,
    draw=black
]         table [row sep=\\,y index=0] {
            data\\
125\\ 139\\ 139\\ 151\\ 152\\ 143\\ 140\\ 157\\ 154\\ 143\\ 147\\ 149\\ 141\\ 142\\ 146\\ 146\\ 157\\ 144\\
139\\ 162\\ 160\\ 145\\ 162\\ 163\\ 139\\ 152\\ 132\\ 151\\ 126\\ 155\\ 143\\ 167\\ 166\\ 127\\ 146\\ 151\\
154\\ 158\\ 129\\ 168\\ 146\\ 133\\ 138\\ 137\\ 151\\ 132\\ 154\\ 129\\ 132\\ 150\\ 153\\ 141\\ 152\\ 148\\
141\\ 168\\ 158\\ 138\\ 145\\ 117\\ 148\\ 141\\ 128\\ 120\\ 142\\ 137\\ 162\\ 153\\ 161\\ 148\\ 144\\ 130\\
129\\ 145\\ 137\\ 133\\ 157\\ 134\\ 166\\ 168\\ 145\\ 154\\ 147\\ 147\\ 155\\ 141\\ 128\\ 164\\ 156\\ 158\\
147\\ 169\\ 157\\ 140\\ 136\\ 147\\ 136\\ 146\\ 123\\ 147\\
    };
\end{axis}
\end{tikzpicture}
\end{minipage}
\begin{minipage}[t]{7cm}
\vspace*{0mm}
\begin{tikzpicture}
\begin{axis}[
    ytick={1},
    yticklabels={Boundary size\phantom(},
    yticklabel style = {align=center, font=\small, rotate=90},
    height=3cm,
    width=7cm,
    grid=major,
    grid style={dashed, gray!30}
    ]
    \addplot+ [boxplot, fill=blue!60]
        table [row sep=\\,y index=0] {
            data\\
125\\ 139\\ 139\\ 151\\ 152\\ 143\\ 140\\ 157\\ 154\\ 143\\ 147\\ 149\\ 141\\ 142\\ 146\\ 146\\ 157\\ 144\\
139\\ 162\\ 160\\ 145\\ 162\\ 163\\ 139\\ 152\\ 132\\ 151\\ 126\\ 155\\ 143\\ 167\\ 166\\ 127\\ 146\\ 151\\
154\\ 158\\ 129\\ 168\\ 146\\ 133\\ 138\\ 137\\ 151\\ 132\\ 154\\ 129\\ 132\\ 150\\ 153\\ 141\\ 152\\ 148\\
141\\ 168\\ 158\\ 138\\ 145\\ 117\\ 148\\ 141\\ 128\\ 120\\ 142\\ 137\\ 162\\ 153\\ 161\\ 148\\ 144\\ 130\\
129\\ 145\\ 137\\ 133\\ 157\\ 134\\ 166\\ 168\\ 145\\ 154\\ 147\\ 147\\ 155\\ 141\\ 128\\ 164\\ 156\\ 158\\
147\\ 169\\ 157\\ 140\\ 136\\ 147\\ 136\\ 146\\ 123\\ 147\\
    };
\end{axis}
\end{tikzpicture}\\[0.4cm]
\begin{tikzpicture}
\begin{axis}[
    ytick={1},
    yticklabels={Time (ms)},
    yticklabel style = {align=center, font=\small, rotate=90},
    height=3cm,
    width=7cm,
    grid=major,
    grid style={dashed, gray!30}
    ]
    \addplot+ [boxplot, fill=blue!60]
        table [row sep=\\,y index=0] {
            data\\
2985\\ 2937\\ 3359\\ 3485\\ 3484\\ 3266\\ 3406\\ 3453\\ 3703\\ 3657\\ 3687\\ 3844\\ 3797\\ 3453\\ 3562\\
3438\\ 3750\\ 3406\\ 3359\\ 3672\\ 3485\\ 3422\\ 3250\\ 3875\\ 3593\\ 3547\\ 3641\\ 3265\\ 3016\\ 3250\\
3109\\ 3704\\ 3468\\ 3141\\ 3578\\ 3203\\ 3047\\ 3656\\ 2985\\ 3203\\ 3390\\ 2938\\ 2312\\ 3125\\ 2782\\
2672\\ 2968\\ 2735\\ 2484\\ 3359\\ 3750\\ 3219\\ 3188\\ 3359\\ 2750\\ 3500\\ 3453\\ 3016\\ 3047\\ 2812\\
3594\\ 3062\\ 3094\\ 3016\\ 2969\\ 3625\\ 3421\\ 3313\\ 3750\\ 3375\\ 5047\\ 6031\\ 6344\\ 6844\\ 6718\\
7672\\ 7688\\ 3172\\ 3468\\ 3375\\ 3016\\ 2750\\ 3422\\ 3500\\ 3281\\ 3797\\ 3219\\ 3250\\ 2890\\ 3219\\
3406\\ 3328\\ 3250\\ 3266\\ 2875\\ 3125\\ 3016\\ 2922\\ 2843\\ 2907\\
    };
\end{axis}
\end{tikzpicture}
\end{minipage}
\caption{The histograms of the sizes of boundaries of contracting portraits of 100 elements in Grigorchuk group $\G$ represented by words of length 500 (top), 1000 (middle), and 5000 (bottom), and corresponding boxplots for the boundary sizes and times required to compute them (in milliseconds).\label{fig:grigorchuk_hist}}
\end{figure}
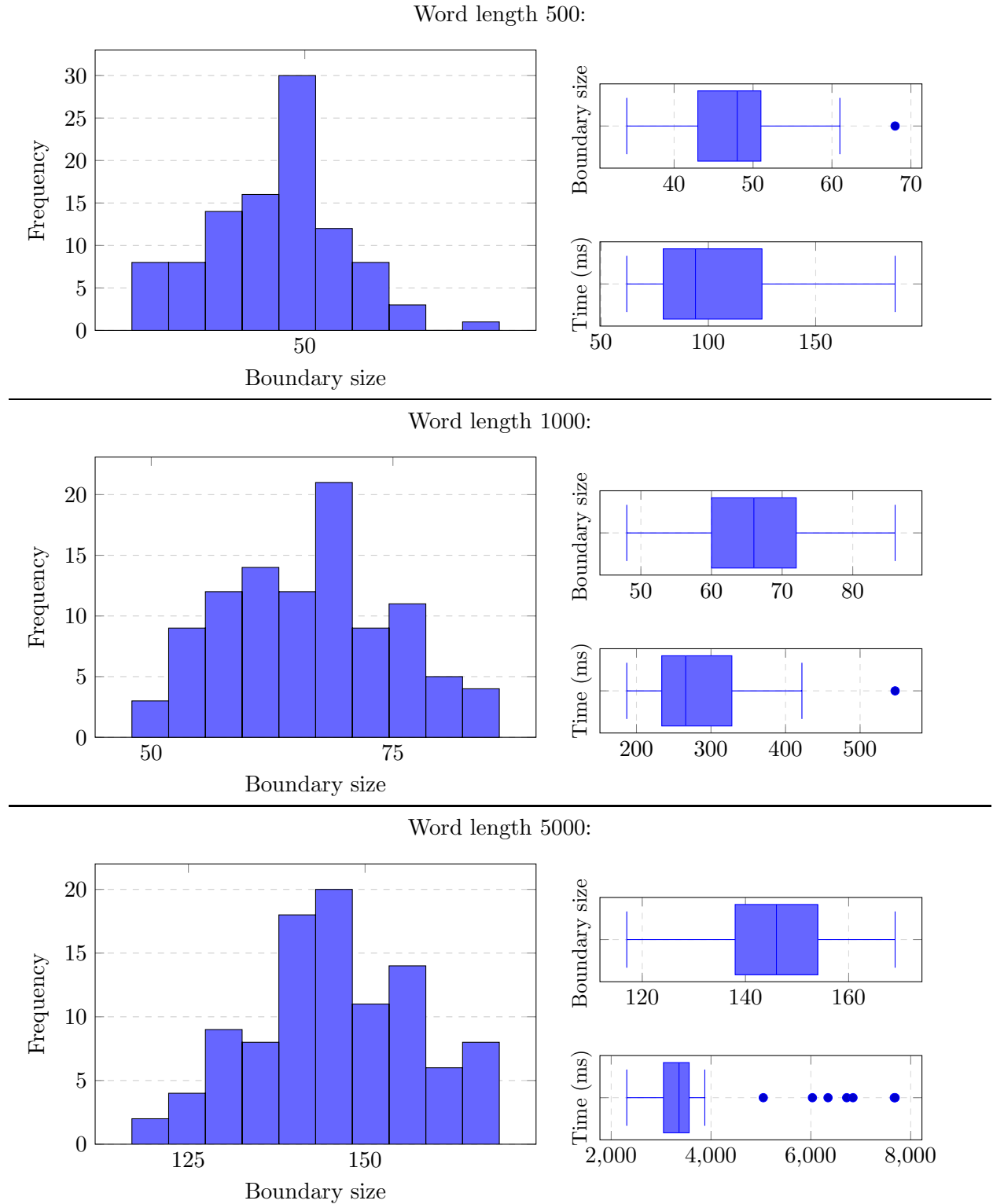

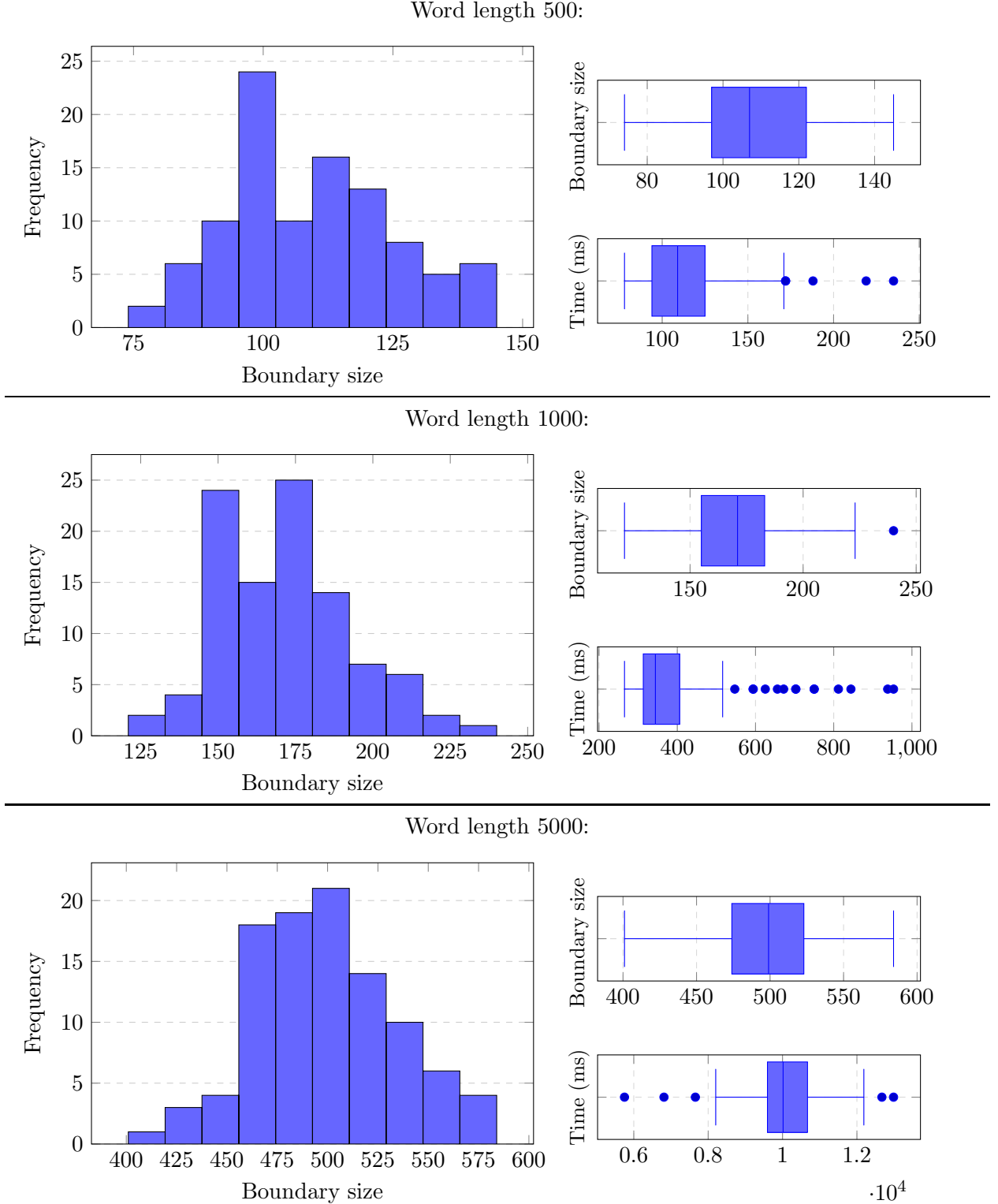
\begin{figure}[H]
\centering
Word length 500:\\
\begin{minipage}[t]{9cm}
\vspace*{0mm}
\begin{tikzpicture}
\begin{axis}[
    title={},
    xlabel={Boundary size},
    ylabel={Frequency},
    ymajorgrids=true,
    grid style={dashed,gray!40},
    width=9cm,
    height=6.3cm,
    ymin=0,
    xtick distance=25,
    ytick distance=5,
]

\addplot[
    ybar,
    hist={
        bins=10
    },
    fill=blue!60,
    draw=black
]         table [row sep=\\,y index=0] {
            data\\
88\\ 91\\ 107\\ 90\\ 116\\ 114\\ 132\\ 101\\ 138\\ 119\\ 107\\ 120\\ 92\\ 114\\ 125\\ 103\\ 119\\ 125\\ 102\\
111\\ 127\\ 115\\ 108\\ 102\\ 90\\ 124\\ 114\\ 90\\ 94\\ 140\\ 81\\ 105\\ 99\\ 111\\ 82\\ 142\\ 98\\ 123\\
98\\ 98\\ 106\\ 97\\ 111\\ 110\\ 115\\ 122\\ 145\\ 102\\ 98\\ 102\\ 97\\ 83\\ 95\\ 144\\ 118\\ 95\\ 74\\
120\\ 107\\ 134\\ 96\\ 112\\ 102\\ 122\\ 112\\ 100\\ 97\\ 96\\ 137\\ 99\\ 106\\ 97\\ 100\\ 111\\ 110\\ 132\\
122\\ 123\\ 110\\ 117\\ 131\\ 98\\ 141\\ 94\\ 93\\ 114\\ 106\\ 104\\ 96\\ 122\\ 82\\ 98\\ 125\\ 88\\ 122\\
127\\ 127\\ 125\\ 85\\ 96\\
    };
\end{axis}
\end{tikzpicture}
\end{minipage}
\begin{minipage}[t]{7cm}
\vspace*{0mm}
\begin{tikzpicture}
\begin{axis}[
    ytick={1},
    yticklabels={Boundary size\phantom(},
    yticklabel style = {align=center, font=\small, rotate=90},
    height=3cm,
    width=7cm,
    grid=major,
    grid style={dashed, gray!30}
    ]
    \addplot+ [boxplot, fill=blue!60]
        table [row sep=\\,y index=0] {
            data\\
88\\ 91\\ 107\\ 90\\ 116\\ 114\\ 132\\ 101\\ 138\\ 119\\ 107\\ 120\\ 92\\ 114\\ 125\\ 103\\ 119\\ 125\\ 102\\
111\\ 127\\ 115\\ 108\\ 102\\ 90\\ 124\\ 114\\ 90\\ 94\\ 140\\ 81\\ 105\\ 99\\ 111\\ 82\\ 142\\ 98\\ 123\\
98\\ 98\\ 106\\ 97\\ 111\\ 110\\ 115\\ 122\\ 145\\ 102\\ 98\\ 102\\ 97\\ 83\\ 95\\ 144\\ 118\\ 95\\ 74\\
120\\ 107\\ 134\\ 96\\ 112\\ 102\\ 122\\ 112\\ 100\\ 97\\ 96\\ 137\\ 99\\ 106\\ 97\\ 100\\ 111\\ 110\\ 132\\
122\\ 123\\ 110\\ 117\\ 131\\ 98\\ 141\\ 94\\ 93\\ 114\\ 106\\ 104\\ 96\\ 122\\ 82\\ 98\\ 125\\ 88\\ 122\\
127\\ 127\\ 125\\ 85\\ 96\\
    };
\end{axis}
\end{tikzpicture}\\[0.4cm]
\begin{tikzpicture}
\begin{axis}[
    ytick={1},
    yticklabels={Time (ms)},
    yticklabel style = {align=center, font=\small, rotate=90},
    height=3cm,
    width=7cm,
    grid=major,
    grid style={dashed, gray!30}
    ]
    \addplot+ [boxplot, fill=blue!60]
        table [row sep=\\,y index=0] {
            data\\
78\\ 78\\ 94\\ 78\\ 125\\ 188\\ 125\\ 109\\ 125\\ 125\\ 94\\ 125\\ 94\\ 125\\ 109\\ 109\\ 110\\ 125\\ 94\\
109\\ 125\\ 109\\ 110\\ 109\\ 94\\ 125\\ 109\\ 94\\ 94\\ 125\\ 78\\ 94\\ 93\\ 110\\ 93\\ 141\\ 109\\ 110\\
109\\ 94\\ 109\\ 157\\ 140\\ 156\\ 235\\ 219\\ 171\\ 172\\ 157\\ 125\\ 125\\ 78\\ 109\\ 141\\ 109\\ 109\\
79\\ 140\\ 110\\ 125\\ 109\\ 109\\ 125\\ 125\\ 125\\ 125\\ 94\\ 109\\ 141\\ 94\\ 94\\ 109\\ 94\\ 125\\ 109\\
125\\ 125\\ 109\\ 125\\ 125\\ 110\\ 94\\ 125\\ 109\\ 94\\ 172\\ 93\\ 110\\ 93\\ 125\\ 79\\ 93\\ 125\\ 94\\
125\\ 125\\ 125\\ 125\\ 94\\ 94\\
    };
\end{axis}
\end{tikzpicture}
\end{minipage}\\
\hrulefill\\
Word length 1000:\\
\begin{minipage}[t]{9cm}
\vspace*{0mm}
\begin{tikzpicture}
\begin{axis}[
    title={},
    xlabel={Boundary size},
    ylabel={Frequency},
    ymajorgrids=true,
    grid style={dashed,gray!40},
    width=9cm,
    height=6.3cm,
    ymin=0,
    xtick distance=25,
    ytick distance=5,
]

\addplot[
    ybar,
    hist={
        bins=10
    },
    fill=blue!60,
    draw=black
]         table [row sep=\\,y index=0] {
            data\\
208\\ 177\\ 154\\ 199\\ 155\\ 147\\ 164\\ 161\\ 162\\ 191\\ 151\\ 157\\ 177\\ 155\\ 202\\ 128\\ 162\\ 121\\
189\\ 154\\ 146\\ 151\\ 166\\ 186\\ 170\\ 175\\ 151\\ 172\\ 143\\ 151\\ 175\\ 205\\ 149\\ 182\\ 195\\ 165\\
171\\ 219\\ 148\\ 155\\ 170\\ 162\\ 154\\ 240\\ 163\\ 215\\ 204\\ 181\\ 172\\ 170\\ 197\\ 223\\ 166\\ 150\\
156\\ 156\\ 192\\ 156\\ 185\\ 146\\ 165\\ 182\\ 206\\ 189\\ 160\\ 154\\ 197\\ 175\\ 173\\ 141\\ 172\\ 153\\
179\\ 138\\ 180\\ 174\\ 152\\ 183\\ 172\\ 173\\ 172\\ 162\\ 185\\ 179\\ 153\\ 188\\ 175\\ 177\\ 140\\ 209\\
148\\ 174\\ 168\\ 192\\ 173\\ 196\\ 171\\ 206\\ 182\\ 168\\
    };
\end{axis}
\end{tikzpicture}
\end{minipage}
\begin{minipage}[t]{7cm}
\vspace*{0mm}
\begin{tikzpicture}
\begin{axis}[
    ytick={1},
    yticklabels={Boundary size\phantom(},
    yticklabel style = {align=center, font=\small, rotate=90},
    height=3cm,
    width=7cm,
    grid=major,
    grid style={dashed, gray!30}
    ]
    \addplot+ [boxplot, fill=blue!60]
        table [row sep=\\,y index=0] {
            data\\
208\\ 177\\ 154\\ 199\\ 155\\ 147\\ 164\\ 161\\ 162\\ 191\\ 151\\ 157\\ 177\\ 155\\ 202\\ 128\\ 162\\ 121\\
189\\ 154\\ 146\\ 151\\ 166\\ 186\\ 170\\ 175\\ 151\\ 172\\ 143\\ 151\\ 175\\ 205\\ 149\\ 182\\ 195\\ 165\\
171\\ 219\\ 148\\ 155\\ 170\\ 162\\ 154\\ 240\\ 163\\ 215\\ 204\\ 181\\ 172\\ 170\\ 197\\ 223\\ 166\\ 150\\
156\\ 156\\ 192\\ 156\\ 185\\ 146\\ 165\\ 182\\ 206\\ 189\\ 160\\ 154\\ 197\\ 175\\ 173\\ 141\\ 172\\ 153\\
179\\ 138\\ 180\\ 174\\ 152\\ 183\\ 172\\ 173\\ 172\\ 162\\ 185\\ 179\\ 153\\ 188\\ 175\\ 177\\ 140\\ 209\\
148\\ 174\\ 168\\ 192\\ 173\\ 196\\ 171\\ 206\\ 182\\ 168\\
    };
\end{axis}
\end{tikzpicture}\\[0.4cm]
\begin{tikzpicture}
\begin{axis}[
    ytick={1},
    yticklabels={Time (ms)},
    yticklabel style = {align=center, font=\small, rotate=90},
    height=3cm,
    width=7cm,
    grid=major,
    grid style={dashed, gray!30}
    ]
    \addplot+ [boxplot, fill=blue!60]
        table [row sep=\\,y index=0] {
            data\\
390\\ 344\\ 313\\ 406\\ 297\\ 297\\ 328\\ 297\\ 328\\ 375\\ 328\\ 344\\ 343\\ 375\\ 375\\ 313\\ 297\\ 265\\
375\\ 297\\ 297\\ 297\\ 344\\ 375\\ 344\\ 343\\ 313\\ 344\\ 328\\ 297\\ 406\\ 390\\ 313\\ 359\\ 407\\ 328\\
343\\ 422\\ 297\\ 328\\ 360\\ 312\\ 313\\ 468\\ 313\\ 406\\ 406\\ 360\\ 344\\ 328\\ 359\\ 406\\ 313\\ 297\\
343\\ 282\\ 359\\ 297\\ 359\\ 297\\ 344\\ 375\\ 437\\ 391\\ 313\\ 328\\ 375\\ 328\\ 359\\ 360\\ 343\\ 313\\
359\\ 281\\ 344\\ 360\\ 281\\ 328\\ 406\\ 328\\ 375\\ 375\\ 375\\ 453\\ 375\\ 516\\ 672\\ 750\\ 625\\ 844\\
547\\ 656\\ 812\\ 953\\ 938\\ 703\\ 594\\ 594\\ 750\\ 703\\
    };
\end{axis}
\end{tikzpicture}
\end{minipage}\\
\hrulefill\\
Word length 5000:\\
\begin{minipage}[t]{9cm}
\vspace*{0mm}
\begin{tikzpicture}
\begin{axis}[
    title={},
    xlabel={Boundary size},
    ylabel={Frequency},
    ymajorgrids=true,
    grid style={dashed,gray!40},
    width=9cm,
    height=6.3cm,
    ymin=0,
    xtick distance=25,
    ytick distance=5,
]

\addplot[
    ybar,
    hist={
        bins=10
    },
    fill=blue!60,
    draw=black
]         table [row sep=\\,y index=0] {
            data\\
508\\ 510\\ 572\\ 473\\ 526\\ 475\\ 525\\ 546\\ 500\\ 523\\ 465\\ 537\\ 476\\ 509\\ 528\\ 526\\ 561\\ 489\\
457\\ 542\\ 460\\ 471\\ 571\\ 547\\ 558\\ 459\\ 489\\ 560\\ 464\\ 478\\ 501\\ 576\\ 560\\ 513\\ 435\\ 454\\
559\\ 482\\ 482\\ 479\\ 500\\ 482\\ 584\\ 469\\ 472\\ 462\\ 540\\ 503\\ 504\\ 445\\ 486\\ 474\\ 512\\ 553\\
489\\ 485\\ 450\\ 452\\ 493\\ 467\\ 541\\ 497\\ 509\\ 540\\ 506\\ 529\\ 466\\ 496\\ 435\\ 532\\ 469\\ 499\\
522\\ 467\\ 507\\ 401\\ 513\\ 534\\ 457\\ 488\\ 429\\ 488\\ 509\\ 497\\ 478\\ 505\\ 517\\ 482\\ 491\\ 474\\
521\\ 511\\ 538\\ 482\\ 518\\ 506\\ 459\\ 505\\ 509\\ 479\\
    };
\end{axis}
\end{tikzpicture}
\end{minipage}
\begin{minipage}[t]{7cm}
\vspace*{0mm}
\begin{tikzpicture}
\begin{axis}[
    ytick={1},
    yticklabels={Boundary size\phantom(},
    yticklabel style = {align=center, font=\small, rotate=90},
    height=3cm,
    width=7cm,
    grid=major,
    grid style={dashed, gray!30}
    ]
    \addplot+ [boxplot, fill=blue!60]
        table [row sep=\\,y index=0] {
            data\\
508\\ 510\\ 572\\ 473\\ 526\\ 475\\ 525\\ 546\\ 500\\ 523\\ 465\\ 537\\ 476\\ 509\\ 528\\ 526\\ 561\\ 489\\
457\\ 542\\ 460\\ 471\\ 571\\ 547\\ 558\\ 459\\ 489\\ 560\\ 464\\ 478\\ 501\\ 576\\ 560\\ 513\\ 435\\ 454\\
559\\ 482\\ 482\\ 479\\ 500\\ 482\\ 584\\ 469\\ 472\\ 462\\ 540\\ 503\\ 504\\ 445\\ 486\\ 474\\ 512\\ 553\\
489\\ 485\\ 450\\ 452\\ 493\\ 467\\ 541\\ 497\\ 509\\ 540\\ 506\\ 529\\ 466\\ 496\\ 435\\ 532\\ 469\\ 499\\
522\\ 467\\ 507\\ 401\\ 513\\ 534\\ 457\\ 488\\ 429\\ 488\\ 509\\ 497\\ 478\\ 505\\ 517\\ 482\\ 491\\ 474\\
521\\ 511\\ 538\\ 482\\ 518\\ 506\\ 459\\ 505\\ 509\\ 479\\
    };
\end{axis}
\end{tikzpicture}\\[0.4cm]
\begin{tikzpicture}
\begin{axis}[
    ytick={1},
    yticklabels={Time (ms)},
    yticklabel style = {align=center, font=\small, rotate=90},
    height=3cm,
    width=7cm,
    grid=major,
    grid style={dashed, gray!30}
    ]
    \addplot+ [boxplot, fill=blue!60]
        table [row sep=\\,y index=0] {
            data\\
6812\\ 8204\\ 11578\\ 9640\\ 10141\\ 9766\\ 9984\\ 10328\\ 10047\\ 10812\\ 9907\\ 10953\\ 10062\\ 10000\\
11203\\ 10782\\ 10812\\ 9547\\ 8688\\ 9906\\ 9609\\ 10547\\ 11359\\ 11250\\ 11422\\ 9094\\ 9469\\ 12187\\
9953\\ 9907\\ 11562\\ 12985\\ 11250\\ 11265\\ 8422\\ 9594\\ 11047\\ 5750\\ 9281\\ 11891\\ 10984\\ 9641\\
11062\\ 9297\\ 10469\\ 9906\\ 10016\\ 9484\\ 10609\\ 10282\\ 10078\\ 10343\\ 10282\\ 11437\\ 10297\\ 10422\\
8750\\ 7656\\ 10047\\ 9188\\ 10031\\ 9328\\ 9844\\ 11297\\ 9796\\ 11454\\ 9328\\ 11312\\ 9203\\ 10172\\
9063\\ 9609\\ 12672\\ 10391\\ 9765\\ 9125\\ 9610\\ 10015\\ 9781\\ 9329\\ 8843\\ 9625\\ 10672\\ 9844\\ 9578\\
9859\\ 9782\\ 10390\\ 10094\\ 9625\\ 10281\\ 9813\\ 10078\\ 10203\\ 10813\\ 9750\\ 8406\\ 11156\\ 10328\\
9281\\
    };
\end{axis}
\end{tikzpicture}
\end{minipage}
\caption{The histograms of the sizes of boundaries of contracting portraits of 100 elements in Basilica group $\B$ represented by words of length 500 (top), 1000 (middle), and 5000 (bottom), and corresponding boxplots for the boundary sizes and times required to compute them (in milliseconds).\label{fig:basilica_hist}}
\end{figure}

\begin{figure}[H]
\centering
Word length 500:\\
\begin{minipage}[t]{9cm}
\vspace*{0mm}
\begin{tikzpicture}
\begin{axis}[
    title={},
    xlabel={Boundary size},
    ylabel={Frequency},
    ymajorgrids=true,
    grid style={dashed,gray!40},
    width=9cm,
    height=6.3cm,
    ymin=0,
    xtick distance=50,
    ytick distance=5,
]

\addplot[
    ybar,
    hist={
        bins=10
    },
    fill=blue!60,
    draw=black
]         table [row sep=\\,y index=0] {
            data\\
451\\ 205\\ 391\\ 277\\ 337\\ 367\\ 349\\ 325\\ 289\\ 253\\ 259\\ 337\\ 397\\ 283\\ 355\\ 301\\ 343\\ 295\\
379\\ 295\\ 313\\ 319\\ 277\\ 325\\ 283\\ 241\\ 289\\ 271\\ 439\\ 265\\ 421\\ 277\\ 223\\ 337\\ 265\\ 319\\
397\\ 511\\ 415\\ 301\\ 319\\ 133\\ 331\\ 373\\ 247\\ 349\\ 325\\ 391\\ 163\\ 397\\ 313\\ 397\\ 361\\ 283\\
235\\ 319\\ 343\\ 307\\ 307\\ 337\\ 379\\ 349\\ 259\\ 199\\ 319\\ 307\\ 313\\ 205\\ 259\\ 319\\ 391\\ 421\\
361\\ 271\\ 217\\ 433\\ 301\\ 223\\ 445\\ 325\\ 313\\ 259\\ 325\\ 337\\ 247\\ 451\\ 391\\ 307\\ 313\\ 295\\
331\\ 361\\ 229\\ 289\\ 319\\ 337\\ 253\\ 247\\ 421\\ 313\\
    };
\end{axis}
\end{tikzpicture}
\end{minipage}
\begin{minipage}[t]{7cm}
\vspace*{0mm}
\begin{tikzpicture}
\begin{axis}[
    ytick={1},
    yticklabels={Boundary size\phantom(},
    yticklabel style = {align=center, font=\small, rotate=90},
    height=3cm,
    width=7cm,
    grid=major,
    grid style={dashed, gray!30}
    ]
    \addplot+ [boxplot, fill=blue!60]
        table [row sep=\\,y index=0] {
            data\\
451\\ 205\\ 391\\ 277\\ 337\\ 367\\ 349\\ 325\\ 289\\ 253\\ 259\\ 337\\ 397\\ 283\\ 355\\ 301\\ 343\\ 295\\
379\\ 295\\ 313\\ 319\\ 277\\ 325\\ 283\\ 241\\ 289\\ 271\\ 439\\ 265\\ 421\\ 277\\ 223\\ 337\\ 265\\ 319\\
397\\ 511\\ 415\\ 301\\ 319\\ 133\\ 331\\ 373\\ 247\\ 349\\ 325\\ 391\\ 163\\ 397\\ 313\\ 397\\ 361\\ 283\\
235\\ 319\\ 343\\ 307\\ 307\\ 337\\ 379\\ 349\\ 259\\ 199\\ 319\\ 307\\ 313\\ 205\\ 259\\ 319\\ 391\\ 421\\
361\\ 271\\ 217\\ 433\\ 301\\ 223\\ 445\\ 325\\ 313\\ 259\\ 325\\ 337\\ 247\\ 451\\ 391\\ 307\\ 313\\ 295\\
331\\ 361\\ 229\\ 289\\ 319\\ 337\\ 253\\ 247\\ 421\\ 313\\
    };
\end{axis}
\end{tikzpicture}\\[0.4cm]
\begin{tikzpicture}
\begin{axis}[
    ytick={1},
    yticklabels={Time (ms)},
    yticklabel style = {align=center, font=\small, rotate=90},
    height=3cm,
    width=7cm,
    grid=major,
    grid style={dashed, gray!30}
    ]
    \addplot+ [boxplot, fill=blue!60]
        table [row sep=\\,y index=0] {
            data\\
500\\ 282\\ 562\\ 375\\ 391\\ 672\\ 515\\ 500\\ 438\\ 328\\ 422\\ 578\\ 625\\ 453\\ 469\\ 547\\ 531\\ 390\\
579\\ 421\\ 657\\ 718\\ 579\\ 593\\ 485\\ 375\\ 390\\ 313\\ 531\\ 375\\ 688\\ 437\\ 328\\ 438\\ 453\\ 515\\
563\\ 562\\ 454\\ 390\\ 563\\ 203\\ 453\\ 609\\ 313\\ 390\\ 375\\ 422\\ 219\\ 516\\ 437\\ 578\\ 625\\ 453\\
469\\ 360\\ 421\\ 500\\ 516\\ 375\\ 422\\ 453\\ 469\\ 281\\ 359\\ 438\\ 391\\ 296\\ 563\\ 1141\\ 1156\\
1109\\ 844\\ 969\\ 765\\ 1250\\ 1032\\ 921\\ 1563\\ 1219\\ 1078\\ 1031\\ 1266\\ 1156\\ 687\\ 1079\\ 1140\\
938\\ 1109\\ 1016\\ 1109\\ 1063\\ 859\\ 984\\ 985\\ 1078\\ 1109\\ 922\\ 1563\\ 968\\
    };
\end{axis}
\end{tikzpicture}
\end{minipage}\\
\hrulefill\\
Word length 1000:\\
\begin{minipage}[t]{9cm}
\vspace*{0mm}
\begin{tikzpicture}
\begin{axis}[
    title={},
    xlabel={Boundary size},
    ylabel={Frequency},
    ymajorgrids=true,
    grid style={dashed,gray!40},
    width=9cm,
    height=6.3cm,
    ymin=0,
    xtick distance=50,
    ytick distance=5,
]

\addplot[
    ybar,
    hist={
        bins=10
    },
    fill=blue!60,
    draw=black
]         table [row sep=\\,y index=0] {
            data\\
613\\ 421\\ 529\\ 469\\ 565\\ 385\\ 493\\ 463\\ 541\\ 529\\ 457\\ 439\\ 457\\ 553\\ 463\\ 463\\ 487\\ 601\\
595\\ 445\\ 523\\ 565\\ 547\\ 481\\ 475\\ 409\\ 493\\ 487\\ 469\\ 421\\ 505\\ 595\\ 469\\ 487\\ 481\\ 523\\
601\\ 421\\ 457\\ 493\\ 457\\ 601\\ 373\\ 517\\ 481\\ 391\\ 523\\ 481\\ 505\\ 517\\ 595\\ 457\\ 535\\ 541\\
343\\ 343\\ 523\\ 403\\ 583\\ 547\\ 457\\ 415\\ 409\\ 517\\ 463\\ 457\\ 475\\ 487\\ 493\\ 565\\ 451\\ 535\\
475\\ 541\\ 679\\ 397\\ 577\\ 487\\ 499\\ 625\\ 481\\ 559\\ 391\\ 493\\ 607\\ 559\\ 481\\ 427\\ 505\\ 583\\
517\\ 463\\ 505\\ 589\\ 553\\ 517\\ 511\\ 493\\ 541\\ 493\\
    };
\end{axis}
\end{tikzpicture}
\end{minipage}
\begin{minipage}[t]{7cm}
\vspace*{0mm}
\begin{tikzpicture}
\begin{axis}[
    ytick={1},
    yticklabels={Boundary size\phantom(},
    yticklabel style = {align=center, font=\small, rotate=90},
    height=3cm,
    width=7cm,
    grid=major,
    grid style={dashed, gray!30}
    ]
    \addplot+ [boxplot, fill=blue!60]
        table [row sep=\\,y index=0] {
            data\\
613\\ 421\\ 529\\ 469\\ 565\\ 385\\ 493\\ 463\\ 541\\ 529\\ 457\\ 439\\ 457\\ 553\\ 463\\ 463\\ 487\\ 601\\
595\\ 445\\ 523\\ 565\\ 547\\ 481\\ 475\\ 409\\ 493\\ 487\\ 469\\ 421\\ 505\\ 595\\ 469\\ 487\\ 481\\ 523\\
601\\ 421\\ 457\\ 493\\ 457\\ 601\\ 373\\ 517\\ 481\\ 391\\ 523\\ 481\\ 505\\ 517\\ 595\\ 457\\ 535\\ 541\\
343\\ 343\\ 523\\ 403\\ 583\\ 547\\ 457\\ 415\\ 409\\ 517\\ 463\\ 457\\ 475\\ 487\\ 493\\ 565\\ 451\\ 535\\
475\\ 541\\ 679\\ 397\\ 577\\ 487\\ 499\\ 625\\ 481\\ 559\\ 391\\ 493\\ 607\\ 559\\ 481\\ 427\\ 505\\ 583\\
517\\ 463\\ 505\\ 589\\ 553\\ 517\\ 511\\ 493\\ 541\\ 493\\
    };
\end{axis}
\end{tikzpicture}\\[0.4cm]
\begin{tikzpicture}
\begin{axis}[
    ytick={1},
    yticklabels={Time (ms)},
    yticklabel style = {align=center, font=\small, rotate=90},
    height=3cm,
    width=7cm,
    grid=major,
    grid style={dashed, gray!30}
    ]
    \addplot+ [boxplot, fill=blue!60]
        table [row sep=\\,y index=0] {
            data\\
1578\\ 1297\\ 1469\\ 1203\\ 1469\\ 1297\\ 1359\\ 1110\\ 2531\\ 3437\\ 3157\\ 2484\\ 2969\\ 3265\\ 3266\\
2719\\ 2968\\ 3532\\ 4125\\ 2609\\ 3297\\ 3437\\ 3844\\ 3344\\ 3219\\ 2609\\ 2984\\ 3047\\ 2578\\ 2563\\
3000\\ 3375\\ 2734\\ 3266\\ 2875\\ 3281\\ 3610\\ 2828\\ 2875\\ 2984\\ 2781\\ 3141\\ 2453\\ 3438\\ 3187\\
2438\\ 3656\\ 2953\\ 3047\\ 3390\\ 3329\\ 2859\\ 3453\\ 2828\\ 2422\\ 2359\\ 2907\\ 2328\\ 2844\\ 3250\\
2609\\ 2469\\ 2453\\ 3078\\ 2891\\ 2828\\ 2750\\ 2750\\ 2765\\ 3125\\ 2782\\ 3328\\ 3250\\ 3265\\ 3860\\
2343\\ 2844\\ 2938\\ 2718\\ 3219\\ 3063\\ 3078\\ 2375\\ 3594\\ 3578\\ 3484\\ 2641\\ 2703\\ 3578\\ 3812\\
3516\\ 2672\\ 3359\\ 3313\\ 2937\\ 3141\\ 2906\\ 3125\\ 3047\\ 3109\\
};
\end{axis}
\end{tikzpicture}
\end{minipage}\\
\hrulefill\\
Word length 5000:\\
\begin{minipage}[t]{9cm}
\vspace*{0mm}
\begin{tikzpicture}
\begin{axis}[
    title={},
    xlabel={Boundary size},
    ylabel={Frequency},
    ymajorgrids=true,
    grid style={dashed,gray!40},
    width=9cm,
    height=6.3cm,
    ymin=0,
    xtick distance=100,
    ytick distance=5,
]

\addplot[
    ybar,
    hist={
        bins=10
    },
    fill=blue!60,
    draw=black
]         table [row sep=\\,y index=0] {
            data\\
1183\\ 1129\\ 1291\\ 1027\\ 937\\ 1093\\ 1093\\ 1261\\ 1045\\ 1039\\ 1219\\ 1327\\ 1315\\ 1069\\ 1183\\
1309\\ 1255\\ 1273\\ 1117\\ 1141\\ 1225\\ 1099\\ 1441\\ 1225\\ 1135\\ 1387\\ 1177\\ 1099\\ 1069\\ 1117\\
1141\\ 1159\\ 1123\\ 1171\\ 997\\ 1027\\ 1111\\ 1189\\ 1081\\ 1231\\ 1105\\ 1351\\ 1141\\ 1105\\ 1057\\
1069\\ 1195\\ 1105\\ 1321\\ 1123\\ 1051\\ 1189\\ 1153\\ 1195\\ 1183\\ 1117\\ 1105\\ 1111\\ 1027\\ 1177\\
1237\\ 1063\\ 1165\\ 1135\\ 1309\\ 1123\\ 1147\\ 1225\\ 1261\\ 1141\\ 1141\\ 1129\\ 1153\\ 1489\\ 1201\\
1213\\ 1111\\ 1261\\ 1165\\ 1255\\ 1165\\ 1123\\ 1057\\ 1261\\ 1345\\ 1027\\ 1135\\ 1225\\ 1135\\ 1117\\
1153\\ 1015\\ 1063\\ 1069\\ 1087\\ 1159\\ 1153\\ 1183\\ 1021\\ 1171\\
    };
\end{axis}
\end{tikzpicture}
\end{minipage}
\begin{minipage}[t]{7cm}
\vspace*{0mm}
\begin{tikzpicture}
\begin{axis}[
    ytick={1},
    yticklabels={Boundary size\phantom(},
    yticklabel style = {align=center, font=\small, rotate=90},
    height=3cm,
    width=7cm,
    grid=major,
    grid style={dashed, gray!30}
    ]
    \addplot+ [boxplot, fill=blue!60]
        table [row sep=\\,y index=0] {
            data\\
1183\\ 1129\\ 1291\\ 1027\\ 937\\ 1093\\ 1093\\ 1261\\ 1045\\ 1039\\ 1219\\ 1327\\ 1315\\ 1069\\ 1183\\
1309\\ 1255\\ 1273\\ 1117\\ 1141\\ 1225\\ 1099\\ 1441\\ 1225\\ 1135\\ 1387\\ 1177\\ 1099\\ 1069\\ 1117\\
1141\\ 1159\\ 1123\\ 1171\\ 997\\ 1027\\ 1111\\ 1189\\ 1081\\ 1231\\ 1105\\ 1351\\ 1141\\ 1105\\ 1057\\
1069\\ 1195\\ 1105\\ 1321\\ 1123\\ 1051\\ 1189\\ 1153\\ 1195\\ 1183\\ 1117\\ 1105\\ 1111\\ 1027\\ 1177\\
1237\\ 1063\\ 1165\\ 1135\\ 1309\\ 1123\\ 1147\\ 1225\\ 1261\\ 1141\\ 1141\\ 1129\\ 1153\\ 1489\\ 1201\\
1213\\ 1111\\ 1261\\ 1165\\ 1255\\ 1165\\ 1123\\ 1057\\ 1261\\ 1345\\ 1027\\ 1135\\ 1225\\ 1135\\ 1117\\
1153\\ 1015\\ 1063\\ 1069\\ 1087\\ 1159\\ 1153\\ 1183\\ 1021\\ 1171\\
    };
\end{axis}
\end{tikzpicture}\\[0.4cm]
\begin{tikzpicture}
\begin{axis}[
    ytick={1},
    yticklabels={Time (ms)},
    yticklabel style = {align=center, font=\small, rotate=90},
    height=3cm,
    width=7cm,
    grid=major,
    grid style={dashed, gray!30}
    ]
    \addplot+ [boxplot, fill=blue!60]
        table [row sep=\\,y index=0] {
            data\\
18671\\ 19438\\ 42703\\ 37984\\ 36094\\ 39125\\ 28875\\ 44969\\ 36297\\ 37781\\ 44125\\ 41375\\ 47438\\
41281\\ 45703\\ 46594\\ 41218\\ 44922\\ 43360\\ 40547\\ 45968\\ 40782\\ 47781\\ 46265\\ 40032\\ 46593\\
44641\\ 42719\\ 36781\\ 20781\\ 20953\\ 36844\\ 45016\\ 44109\\ 36563\\ 37937\\ 40578\\ 31469\\ 16547\\
22453\\ 21234\\ 22250\\ 18657\\ 19078\\ 17375\\ 18719\\ 20203\\ 18875\\ 22172\\ 18859\\ 19797\\ 19062\\
19797\\ 22422\\ 20719\\ 18875\\ 19265\\ 20579\\ 19343\\ 20938\\ 20562\\ 19266\\ 19219\\ 20843\\ 20813\\
19109\\ 19703\\ 21485\\ 19594\\ 20000\\ 19875\\ 19062\\ 18094\\ 23781\\ 20250\\ 19797\\ 19594\\ 20781\\
20312\\ 21610\\ 19344\\ 19656\\ 20844\\ 26953\\ 29156\\ 23766\\ 24515\\ 21828\\ 20828\\ 20391\\ 20922\\
19000\\ 18547\\ 19625\\ 20031\\ 21438\\ 21421\\ 22469\\ 19375\\ 21641\\
    };
\end{axis}
\end{tikzpicture}
\end{minipage}
\caption{The histograms of the sizes of boundaries of contracting portraits of 100 elements in 7-Basilica group $\B_7$ represented by words of length 500 (top), 1000 (middle), and 5000 (bottom), and corresponding boxplots for the boundary sizes and times required to compute them (in milliseconds).\label{fig:7basilica_hist}}
\end{figure}

\tikzset{
        ->,  
        >=stealth',
        node distance=3cm, 
        every state/.style={thick, fill=gray!10}, 
        initial text=$ $,
        }
        
\begin{proposition}
\label{prop:complexity_product}
    Let $G$ be a contracting self-similar group with nucleus $\N$, and let $\pi(g)$ and $\pi(h)$ be nucleus portraits of elements $g,h\in G$ with boundary sizes $l_g=s(g)$ and $l_h=s(h)$, respectively. The time complexity of Algorithm~\ref{alg:portrait_product} for computing the nucleus portrait $\pi(gh)$ given $\pi(g)$ and $\pi(h)$ is $O\big(d\log|\mathcal N|\cdot(\frac{d}{d-1}(l_g+l_h )-\frac{2}{d-1})\big)$ and the space complexity is $O\big((|\mathcal N|+d)(\frac{d}{d-1}(l_g+l_h )-\frac{2}{d-1})\big)$. For fixed $|N|$ and $d$ (when only the portraits are considered as the input for the algorithm) both these complexities reduce to $O(l_g+l_h)$.
\end{proposition}

\begin{proof}
    Let $i_g$ and $i_h$ denote the numbers of non-terminal nodes of $\pi(g)$ and $\pi(h)$, respectively.
    The labels of the leaves of a nucleus portrait belong to the nucleus $\N$ and an arbitrary extension of the nucleus portrait results in a tree with all leaves in the same finite set $\N$. 
    
    Given Algorithm~\ref{alg:portrait_product}, we note that the complexity of computing the product of portraits is proportional to the number of iterations of the function \textsc{Product}. Each iteration of this function results in one of the following three cases:
    
    \begin{description}
        \item[Case I] If both arguments are in the nucleus, it returns a precomputed value from the array $NucleusProducts$.
        \item[Case II] If exactly one of two arguments is in the nucleus, it uses a precomputed value from the array $\N_1$ to expand the other argument, calls $d$ instances of \textsc{Product}, and returns the outcomes.        
        \item[Case III] If both arguments are not in the nucleus, it calls another $d$ instances of \textsc{Product} and returns the outcomes.

    \end{description} 

    The time and space complexity of Algorithm~\ref{alg:portrait_product} consists of three stages. The precomputation stage is done only once to compute the portraits of the pairwise products of the elements of the nucleus. Each pairwise product may have at most $|\mathcal N|^2$ states. To calculate the decomposition of each state one needs to compute the product of two permutations from $\Sym(d)$ and $d$ sections. Since we need to compute $|\mathcal N|^2$ portraits, the total time and space complexity of this stage is $O(d|\mathcal N|^4)$.
    
    The first stage of the main algorithm implemented in function \textsc{PortraitProduct}, due to its recursive nature, is linear (both in space and time) in the number of calls of the function \textsc{Product}.  Each of the iterations of this function takes time bounded by $O(d\log|\mathcal N|)$, where the factor $d$ comes from calculating the product $\sigma_g\sigma_h$ in line 21, and $\log|\mathcal N|$ comes from the searches within $\mathcal N_0$ in the \textbf{if} statement. To compute the number of calls of the function \textsc{Product}, we observe that the input portraits are extended to carry out the computation. These extensions occur when we are in \textbf{Case~II}. The maximum number of nodes added to the input portrait $\pi(g)$ by such extensions is bounded above by total number of nodes of the other input portrait $\pi(h)$. Therefore, by symmetry, the total number of nodes for each of the extended portraits required for the computation is bounded above by the sum of the numbers of nodes of the input portraits, i.e. $(i_g+l_g)+(i_h+l_h)$. Let $U_g$ and $U_h$ denote the sets of all nodes of the extended portrait of $\pi(g)$ and $\pi(h)$, respectively, required for computing their product. Then $|U_g| = |U_h|\leq (i_g+l_g)+(i_h+l_h)$. Let $E$ denote an enumerating set of calls of the function \textsc{Product} for computing the portrait product $\pi(g)*\pi(h)$. Then there is bijective map $U_j \rightarrow E$ for $j\in\{g,h\}$ because each instance of the recursive function \textsc{Product} uniquely identifies to exactly one node in each of the extended portraits. Therefore, the number of calls of the function is bounded above by $(i_g+l_g)+(i_h+l_h)$. From Lemma~\ref{lem:non-terminal-nodes}, this sum can be rewritten as follows:
    
    
    

\begin{equation}
\label{eqn:nodes}
(i_g+l_g)+(i_h+l_h) = \left( \frac{l_g-1}{d-1}+ l_g \right) + \left( \frac{l_h-1}{d-1}+l_h \right) = \frac{d}{d-1}(l_g+l_h )-\frac{2}{d-1} \ .
\end{equation}

Therefore, the time complexity of the first stage is $O\big(d\log|\mathcal N|\cdot(\frac{d}{d-1}(l_g+l_h )-\frac{2}{d-1})\big)$, while the space complexity is bounded above by $O\big((|\mathcal N|+d)(\frac{d}{d-1}(l_g+l_h )-\frac{2}{d-1})\big)$ since every node is labeled either by an element of $\mathcal N$ or by a permutation from $\Sym(d)$.

The second stage of the main algorithm (pruning: Steps~(f) and~(g)) is implemented in Algorithm~\ref{alg:portrait_product} as a recursive function~\textsc{Prune}. The number of its calls is equal to the number of nodes in the obtained portrait, which is again equal to $|U_g|=|U_h|$ and bounded by~\eqref{eqn:nodes}. Each iteration of this function requires $O(\log|\mathcal N|)$ operations in the \textbf{if} statement on lines 29 and 33. The whole pruning procedure does not require any additional space beyond the space occupied by the input portrait computed in the first stage of the algorithm. 

We conclude that the time complexity of Algorithm~\ref{alg:portrait_product} for computing the product of two nucleus portraits is $O\big(d\log|\mathcal N|\cdot(\frac{d}{d-1}(l_g+l_h )-\frac{2}{d-1})\big)$ and the space complexity is $O\big((|\mathcal N|+d)(\frac{d}{d-1}(l_g+l_h )-\frac{2}{d-1})\big)$.
\end{proof}

We note that the size of the boundary sizes is a natural characterization of the size of the portrait as it is proportional to the amount of memory a portrait takes. However, one can also get a (generally less efficient) estimate of the complexity of the algorithm in terms of the depth of the portraits using an obvious inequality $s(g)\leq d^{\partial(g)}$.

\begin{corollary}
\label{cor:prod_depth}
Let $G$ be a contracting self-similar group with nucleus $\N$, and let $\pi(g)$ and $\pi(h)$ be nucleus portraits of elements $g,h\in G$ of depth $\partial(g)$ and $\partial(h)$, respectively. The time and space complexity of Algorithm~\ref{alg:portrait_product} for computing the nucleus portrait $\pi(gh)$ given $\pi(g)$ and $\pi(h)$ is bounded above by $O(d^{\partial(g)} +d^{\partial(h)})$.
\end{corollary}


    



    
    


Now we obtain similar estimates for Algorithm~\ref{alg:portrait_inverse}. 

\begin{proposition}
\label{prop:complexit_inverse}
   Let $G$ be a contracting self-similar group and $\pi(g)$ be the nucleus portrait of $g\in G$ with boundary size $l_g$. The time and space complexities of Algorithm~\ref{alg:portrait_inverse} for computing the nucleus portrait $\pi(g^{-1})$ of $g^{-1}$  is $O(dl_g)$. For fixed $d$ (when only the portraits are considered as the input for the algorithm) both these complexities reduce to $O(l_g)$.
\end{proposition}

\begin{proof}
    We recall that the labels of the boundary nodes of a nucleus portrait belong the nucleus $\N$, which is closed under taking inverses. The complexity of Algorithm~\ref{alg:portrait_inverse} is linear in the number of calls of the function \textsc{Inverse}. In order to compute the portrait inverse, the function is called at each node of the portrait. Recall that $i_g$ denote the number of non-terminal nodes of $\pi(g)$. From Lemma~\ref{lem:non-terminal-nodes}, we see that the total number of these calls equals $i_g+l_g = \frac{d}{d-1}\cdot l_g-\frac{1}{d-1}$. Each iteration of function \textsc{Inverse} requires up to $d$ operations to compute $\sigma_g^{-1}$ in line 7.  Therefore, the time and space complexities of the computation is $O(dl_g)$.
\end{proof}

Similarly to Corollary~\ref{cor:prod_depth}, we can rewrite the complexity estimates in terms of the depth of the portrait.

\begin{corollary}
Let $G$ be a contracting self-similar group and $\pi(g)$ be the nucleus portrait of an element $g\in G$ of depth $\partial(g)$. The time and space complexity of Algorithm~\ref{alg:portrait_inverse} for computing the nucleus portrait $\pi(g^{-1})$ given $\pi(g)$ is bounded above by $O(d^{\partial(g)})$.
\end{corollary}

Finally, we can rewrite the complexities in terms of the lengths of the elements $g$ and $h$ as follows.

\begin{proposition}
Suppose $G$ is a contracting group generated by a finite self-similar set $S$ containing the nucleus of $G$ and let $\pi(g)$ and $\pi(h)$ be nucleus portraits of elements $g,h\in G$ with $|g|_S\leq n$ and $|h|_S\leq n$. The time and space complexities of Algorithms~\ref{alg:portrait_product} and~\ref{alg:portrait_inverse} for computing the nucleus portraits $\pi(gh)$ and $\pi(g^{-1})$, given words of length at most $n$ representing $g$ and $h$, are polynomial in $n$.
\end{proposition}

\begin{proof}
   Follows directly from Propositions~\ref{prop:complexity_product} and~\ref{prop:complexit_inverse},  and Corollary~\ref{cor:log_depth}.
\end{proof}

\algblock{BeginPrecomputation}{EndPrecomputation}
\begin{algorithm} 
\caption{Computing the portrait of the product of group elements}\label{alg:portrait_product} 

\begin{algorithmic}[1]

\BeginPrecomputation

\Comment{Precompute arrays of portraits of elements of the group nucleus $\N$}

\For{$n \in \N$}
    \State\do
    \State$\N_0[n]\gets[n]$
    \State$\N_1[n] \gets$ [\Call{Perm}{$n$}, \Call{Sections}{$n$}[1],$\ldots$, \Call{Sections}{$n$}[$d$]]
\EndFor


\Comment{Precompute the array of portraits of pairwise products of nucleus elements}

\For{$n_1\in \N$}
    \do \\
    \For{$n_2 \in \N$}
        \State\do 
            \State $NucleusProducts$[$n_1,n_2$] $\gets$ \Call{AutomPortrait}{$n_1n_2$}
        \EndFor
\EndFor



\State $d\gets$ (degree of the tree)

\EndPrecomputation
\State
\Function{PortraitProduct}{$Portrait_g,Portrait_h$}

\Function{Product}{$P_g,P_h$}
    \If{$P_g \in \N_0$ \textbf{and} $P_h \in \N_0$ }
        \State\Return $NucleusProducts$[$P_{g}[1],P_{h}[1]$]        
    \ElsIf{$P_g \notin \N_0$ \textbf{and} $P_h \notin \N_0$}
    \State $\sigma_g \gets P_g[1]$
    \State $\sigma_h \gets P_h[1]$
    
    \Comment{$\sigma_g$ and $\sigma_h$ are permutations induced by $g$ and $h$ at the first level}

    \State\Return[$\sigma_g\sigma_h$, \Call{Product}{$P_g[2],P_h[\sigma_g(1)+1]$},$\ldots$,\Call{Product}{$P_g[d+1],P_h[\sigma_g(d)+1]$}]
    
    \ElsIf{$P_g \in \N_0$ \textbf{and} $P_h \notin \N_0$}
    \State\Return\Call{Product}{$\N_1[P_g[1]],P_h$}
    \Else
    \State\Return\Call{Product}{$P_g,\N_1[P_h[1]]$}
    
    
    \EndIf
    
\EndFunction
\Function{Prune}{$P_g$}
    \If{$P_g \in \N_0$}
        \State\Return $P_{g}$
    \Else
        \State $P\gets[P_g[1],\Call{Prune}{P_g[2]},\Call{Prune}{P_g[2]},\ldots,\Call{Prune}{P_g[d+1]}]$
        \If{$P \in \N_1$}
        \Comment{Pruning}
            \State\Return $[n]$ such that $\N_1[n]=P$
        \Else
            \State\Return $P$
        \EndIf
    \EndIf
\EndFunction
\State $P\gets$\Call{Product}{$Portrait_g,Portrait_h$}
\State \Return\Call{Prune}{$P$}
\EndFunction

\end{algorithmic} 
\end{algorithm}


\begin{algorithm} 
\caption{Computing the portrait of the inverse of a group element}\label{alg:portrait_inverse} 
\begin{algorithmic}[1]
\Function{PortraitInv}{$P_g$}
    \State $d\gets$ (degree of the tree)
    \If{$Length(P_g)=1$}
        \State\Return $[(P_g[1])^{-1}]$        
    \Else
        \State $\sigma_g \gets P_g[1]$
        \State\Return[$\sigma_g^{-1}$,\Call{PortraitInv}{$P_g[\sigma_g^{-1}(1)+1]$},\ldots,\Call{PortraitInv}{$P_g[\sigma_g^{-1}(d)+1]$}]
    \EndIf
\EndFunction
\end{algorithmic} 
\end{algorithm}


\begin{figure}
    \centering
    \begin{forest}
    [  $(123)$ [ $(12)$ [$n_1$][$n_2$][$n_3$] ] [$n_4$] [$n_5$]]
    \end{forest}\qquad\qquad
    \begin{forest}
    [  $(132)$ [$n_5^{-1}$][ $(12)$ [$n_2^{-1}$][$n_1^{-1}$][$n_3^{-1}$] ] [$n_4^{-1}$]]
    \end{forest}
    \caption{An example of nucleus portraits of inverse group elements}
    \label{fig:portrait_inverse}
\end{figure}
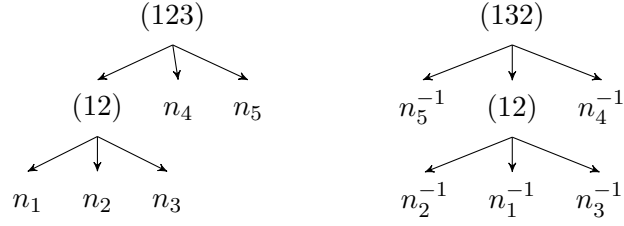

\subsection{An explicit example of AAG key generation using nucleus portraits}
\label{key-gen-portraits}

The key generation process for the AAG protocol described in Subsection~\ref{AAG} requires computations using nucleus portraits. We illustrate here the techniques from the previous subsection using the example of Grigorchuk group $\G$ to demonstrate the steps of AAG protocol. In this example, the public keys of Alice and Bob comprise of 3 elements of $\G$ represented by reduced words of length 10 each.

Following AAG key-generation mechanism using contracting automaton group $G$:

\begin{enumerate}
    \item Given the set of states $Q = \{q_1,\ldots, q_N\}$ of an automaton $\A$ generating a contracting group $G$, Alice chooses a public set $\overline{a}=(a_1,\ldots,a_{N_1})$ where $a_i$ are reduced words over $Q^{\pm1}$. Then she develops nucleus portraits for each element of the set and publishes them as her public key.
    
    For example, let Alice choose the following tuple from Grigorchuk group:
    \[ \overline{a}= \left(a_1 = (daba)^2ba , a_2 = b(ad)^2abada, a_3 = (ad)^3abac\right).\] 
    The nucleus portraits for this set as shown in Table~\ref{table:alice public} will be published as her public key.
    
    \item Bob chooses a public set $\overline{b}=(b_1,\ldots,b_{N_2})$ where $b_i$ are reduced words over $Q^{\pm1}$. Then he develops the nucleus portraits for each element of the set and publishes them as his public key.

    For example, let Bob choose the following tuple from Grigorchuk group:
    \[\overline{b}= \left( b_1 = d(ab)^2acada , b_2 = cab(ad)^3a , b_3 = bac(ab)^3a\right).\] 
    Then he develops the nucleus portrait for each word and publishes them as his public key. This is shown in Table~\ref{table:bob public}.
    
    
    \item Alice then chooses  $A=a_{s_1}^{\varepsilon_1} \cdots a_{s_L}^{\varepsilon_L}$ where $a_{s_i} \in \overline{a}$ and $\varepsilon_i \in \{\pm1\}$. The nucleus portrait of $A$ is her private key.

    For our example, Alice chooses her private key to be the nucleus portrait of $A = a_1a_2a_3$ as shown in Figure~\ref{fig:alice private}.

    \item Bob chooses $B=b_{t_1}^{\delta_1} \cdots b_{t_M}^{\delta_M}$ where $b_{t_i} \in \overline{b}$ and $\delta_i \in \{\pm1\}$. Then he develops the nucleus portrait of $B$ that serves as his private key.

    In our example, Bob chooses his private key to be the nucleus portrait of $ B = b_1b_2b_3$ as shown in Figure~\ref{fig:bob private}.

    \item Using the nucleus portraits, Alice computes $b'_i=A^{-1}b_iA$ for all $b_i \in \overline{b}$ and sends them to Bob (as nucleus portraits). Similarly, Bob computes $a'_i=B^{-1}a_iB$ for all $a_i \in \overline{a}$ as nucleus portraits and sends them to Alice. Now the shared secret key is the nucleus portrait of $K=A^{-1}B^{-1}AB$, which Alice computes as $A^{-1}a'_1a'_2a'_3$ and Bob computes as $B^{-1}b'_1b'_2b'_3$.  
    
    To conclude our example, the shared secret key is the nucleus portrait of $K=A^{-1}B^{-1}AB$ as shown in Figure~\ref{fig:shared-key}.
\end{enumerate}

We emphasize that the public keys of Alice and Bob comprise only of nucleus portraits and all computations and communication between the parties takes place in this form. The reduced words used to develop those portraits are never in public domain at any stage of the key generation process. 


\hspace{1cm}
\begin{table}[h]
\centering
\begin{tabular}{|m{3cm}||c|c|c|}
    \hline
    Group element & $a_1 = (daba)^2ba$ & $a_2 = b(ad)^2abada$ & $a_3 = (ad)^3abac$ \\
    
    \hline
\raisebox{-60pt}{\parbox{2.5cm}{Nucleus portrait}} & \begin{forest}
[  $(\sigma)$ [ $a$ ] [ $()$[ () [ $d$ ] [ $a$ ] ] [ $(\sigma)$ [ $(\sigma)$ [ $c$ ] [ $a$ ] ] [ $d$ ] ] ] 
]
\end{forest} & 
\begin{forest}
   [ $(\sigma)$ [$(\sigma)$ [$b$][1]]
               [$(\sigma)$[$c$][$(\sigma)$[a][c]]]
                   ]   
\end{forest} &
\begin{forest}
[ $()$  [$(\sigma)$ [ $b$ ] [ 1 ] ] [ $(\sigma)$ [ 1 ] [ $b$ ] ] ]
\end{forest} \\
   
    \hline
\end{tabular}
\caption{ Nucleus portraits for $a_1, a_2$ and $a_3$ that constitute the public key of Alice.} 
\label{table:alice public}
\end{table}



\hspace{1cm}
\begin{table}[h]
\centering
\begin{tabular}{|p{2.5 cm}||c|c|c|}
    \hline
    Group element & $b_1 = d(ab)^2acada$ & $b_2 = cab(ad)^3a$ & $b_3 = bac(ab)^3a$ \\
    \hline
    \raisebox{-60pt}{\parbox{2.5cm}{Nucleus portrait}} & \begin{forest}
[  $ (\sigma)$ [ $(\sigma)$ [ $(\sigma)$ [ $c$ ] [ $a$ ] ] [ $d$ ] ] [ () [ () [ $b$ ] [ 1 ] ] [ $(\sigma)$ [ $(\sigma)$ [ $b$ ] [ 1 ] ] [ $(\sigma)$ [ 1 ] [ $b$ ] ] ] ] ]
\end{forest} & 
\begin{forest}
   [  $(\sigma)$ [ $(\sigma)$ [ $b$ ] [ 1 ] ] [ $(\sigma)$ [ 1 ] [ $b$ ] ] ]  
\end{forest}   &
\begin{forest}
[  $ (\sigma)$ [ $(\sigma)$ [ $(\sigma)$ [ $a$ ] [ $c$ ] ] [ $d$ ] ] [ () [ () [ $b$ ] [ 1 ] ] [ $(\sigma)$ [ $b$ ] [ $b$ ] ] ] ]
\end{forest} \\
   
    \hline
\end{tabular}
\caption{ Nucleus portraits for $b_1, b_2$ and $b_3$ that constitute the public key of Bob.} 
\label{table:bob public}
\end{table}



\begin{figure}
    \centering
    \begin{forest}
    [  $(\sigma)$ [ $(\sigma)$ [$(\sigma)$ [$(\sigma)$ [$d$][$a$]] [$(\sigma)$[$a$][$d$]] ][$c$] ] [ $()$ [ $()$ [$d$][$a$] ] [$(\sigma)$ [$(\sigma)$ [$c$][$a$] ][$d$] ] ]]
    \end{forest}
    \caption{Nucleus portrait of Alice's private key $A = a_1a_2a_3$.}
    \label{fig:alice private}
\end{figure}
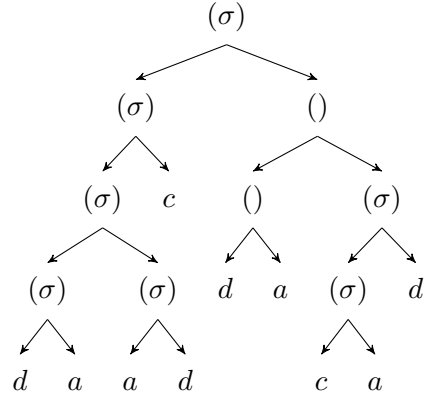



\begin{figure}
    \centering

    \begin{forest}
    [$(\sigma$) [$(\sigma$) [$()$ [$c$][$a$]][1] ] [$(\sigma$) [$(\sigma$) [$(\sigma$)[$(\sigma$)[$d$][$a$] ][$(\sigma$)[$a$][$d$] ] ][$d$] ] [$(\sigma$) [$(\sigma$)[$b$][1] ][$(\sigma$) [$c$][$(\sigma$) [$a$][$c$] ] ] ] ] ]
    \end{forest}
    \caption{Nucleus portrait of Bob's private key $ B = b_1b_2b_3$.}
    \label{fig:bob private}
\end{figure}
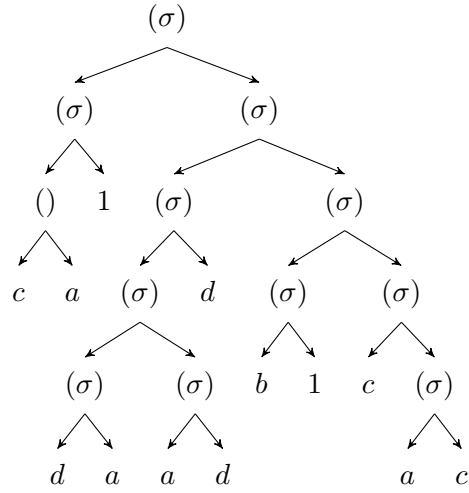


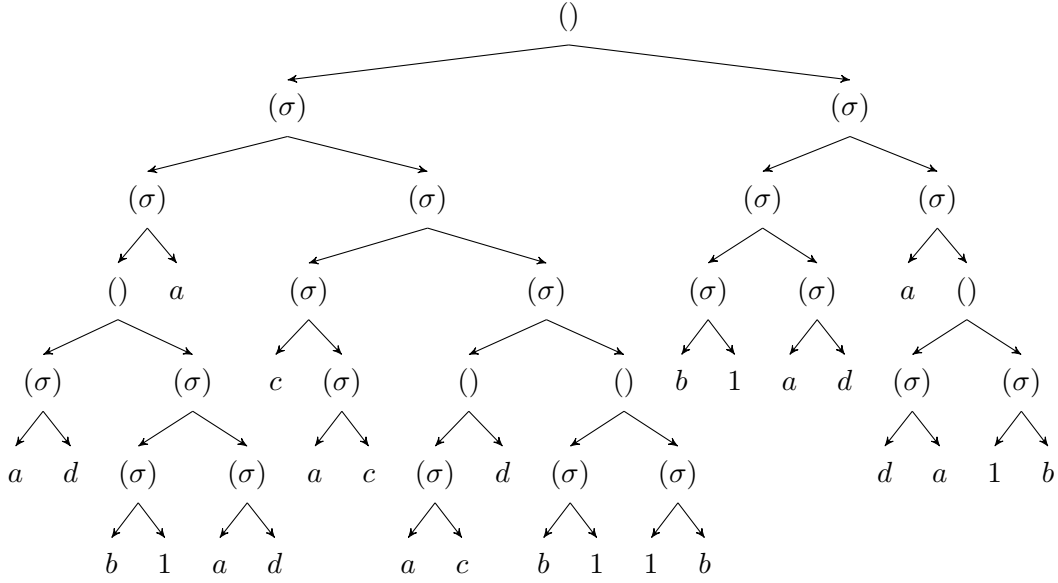
\begin{figure}
    \centering
    \begin{forest}
    [()[$(\sigma)$ [$(\sigma)$ [$()$ [$(\sigma)$[$a$][$d$]][$(\sigma)$ [$(\sigma)$[$b$][1]][$(\sigma)$[$a$][$d$]]]][$a$]][$(\sigma)$[$(\sigma)$[$c$][$(\sigma)$[$a$][$c$]]][$(\sigma)$ [() [$(\sigma)$[$a$][$c$]][$d$]][() [$(\sigma)$[$b$][1]][$(\sigma)$[$1$][$b$]]]]]]
     [$(\sigma)$[$(\sigma)$[$(\sigma)$[$b$][1]][$(\sigma)$[$a$][$d$]]][$(\sigma)$[$a$][()[$(\sigma)$[$d$][$a$]][$(\sigma)$[1][$b$]]]]]]   
    \end{forest}
    \caption{Nucleus portrait of the shared key $K = A^{-1}B^{-1}AB$.}
    \label{fig:shared-key}
\end{figure}


\section{Recovering Words from Nucleus Portraits in Regular Branch Groups}
\label{sec:word_from_portrait}
In this section we provide an algorithm that recovers a word representing an element $g$ of a regular branch group $G$ from its portrait of a fixed depth or, if a group is contracting, from its nucleus portrait. Moreover, we show that if $G$ is contracting, then, given the nucleus portrait of an element $g$ of length $n$ in $G$, this algorithm constructs a word of polynomial in $n$ length. The algorithm is a generalization of a similar algorithm for Grigorchuk group described in~\cite{grigorch:solved,lysenok_mu:conjugacy_problem_in_grigorchuk_group_polynomial10}. The authors are grateful to Rostislav Grigorchuk for pointing out the connection to the original algorithm that served as a basis for the result below.

We begin from recalling the definition of regular branch groups. Consider a self-similar group $G$ and its normal subgroup $\St_G(1)$ consisting
of all elements in $G$ that stabilize the first level of $X^*$.
There is a natural embedding
\[
\Psi\colon \St_G(1)\hookrightarrow G\times G\times\dots\times G
\]
given by
\[
g\stackrel{\Psi}{\mapsto} (g|_0,g|_1,\dots,g|_{d-1}).
\]

\begin{definition}
Let $K,K_0,\dots,K_{d-1}$ be subgroups of a self-similar group $G$
acting on $X^*$. We say that $K$ \emph{geometrically contains} $K_0
\times \dots \times K_{d-1}$ and write
\[
K_0 \times \dots \times K_{d-1} \preceq K
\]
if $K_0 \times \dots \times K_{d-1} \leq \Psi(\St_G(1) \cap K)$.
\end{definition}

\begin{definition}
A group $G$ of tree automorphisms of $X^*$ that acts transitively on
the levels of the tree $X^*$ is called a \emph{regular weakly branch
group} over its nontrivial subgroup $K$ if
\[
K \times \dots \times K \preceq K.
\]
If, in addition, $[G:K]<\infty$ and $[\Psi(\St_G(1) \cap K):K \times \dots \times K]<\infty$  (i.e., $K$ geometrically contains $K\times\cdots\times K$ as a subgroup of finite index), then $G$ is called a \emph{regular branch group} over $K$. The group $K$ is then called a \emph{branching subgroup} in $G$.
\end{definition}

\begin{lemma}
\label{lem:uniform_liftingK}
Let $G$ be a regular branch group generated by a finite set $S$ acting on a $d$-ary rooted tree $X^*$ with a branching subgroup $K$. Then, given a tuple $(g_1,g_2,\ldots,g_d)\in K^d$, where $g_i$'s are given as words in $S^{\pm1}$, there is a word $w$ in $S^{\pm1}$ of length at most $C\sum_{i=1}^d |g_i|_S$ that represents $(g_1,g_2,\ldots,g_d)\in K<G$ for some $C>0$ not depending on $g_i$'s.
\end{lemma}

\begin{proof}
We first construct the Schreier graph $\Gamma=\Gamma(G,K,S)$ of $K$ in $G$ with respect to $S$ (which is finite since $[G:K]<\infty$) and use it to construct a finite generating set $S_K=\{a_1,a_2,\ldots,a_r\}$ of $K$ via the Reidemeister-Schreier process. For each $1\leq i\leq r$ and $1\leq j\leq d$ we define 
\[a_{i,j}=(1,\ldots,1,a_i,1,\ldots,1)\in K\times\cdots\times K,\]
where the nontrivial component $a_i$ is located in the $j$-th coordinate. By construction
\[K\times\cdots\times K=\langle a_{i,j}\colon 1\leq i\leq r, 1\leq j\leq d\rangle.\]

Since $K \times \dots \times K \preceq K$, for each $a_{i,j}$ there is a  word $w_{i,j}$ in $S_K^{\pm 1}$ representing $a_{i,j}$. Let
\[L=\max_{i,j}|w_{i,j}|_{S_K}.\]

Let $M=\max_{i}|a_i|_S$ be the maximal length of generators $a_i$. Then for every $h\in H$
\begin{equation}
\label{eqn:bound_S}
|h|_{S}\leq M|h|_{S_K}.
\end{equation}
On the other hand, since $S_K$ is the generating set constructed by the Reidemeister-Schreier process, every word $w$ in $S^{\pm1}$ representing an element of $K$ corresponds to the loop in $\Gamma$ that starts and ends at the coset $K$. Following this loop and recording labels of edges that are not in the spanning tree of $\Gamma$ that was used to construct $S_K$, yields a word in $S_K$ representing the same element of $K$ as $w$ represents of length not larger than the length of $w$. Therefore, for each $h\in K$
\begin{equation}
|h|_{S_K}\leq |h|_{S}.
\end{equation}

Now, given $(g_1,g_2,\ldots,g_d)\in K^d$, where $g_i$'s are given as words in $S^{\pm1}$, we can rewrite each of $g_i$'s as a word in $S_K^{\pm1}$ of length at most $|g_i|_S$, then lift every letter $a_{i,j}$ in each of these words to a word $w_{i,j}$ of length at most $L$ in $S_K$ and obtain a word $w'$ in $S_K^{\pm1}$ of length at most 
\[|w'|_{S_K}=L\sum_{i=1}^d |g_i|_S\] 
representing $(g_1,g_2,\ldots,g_d)$. Finally, rewriting $w'$ as a word in $S^{\pm1}$ will create a word $w$ in $S^{\pm1}$ of length at most $ML\sum_{i=1}^d |g_i|_S$ by~\eqref{eqn:bound_S}. We finish the proof by setting $C=ML$.
\end{proof}

\begin{lemma}
\label{lem:uniform_lifting}
Let $G$ be a regular branch group generated by a finite set $S$ acting on a $d$-ary rooted tree $X^*$ with a branching subgroup $K$. Suppose $(u_1,\ldots, u_d), (v_1,\ldots, v_d)\in G^d$ be such that $Ku_i=Kv_i$ for all $1\leq i\leq d$. If there is $u\in\St_G(1)$ such that $\Psi(u)=(u_1,\ldots, u_d)$, then there is $v\in \St_G(1)$ such that $\Psi(v)=(v_1,\ldots, v_d)$.
\end{lemma}

\begin{proof}
By assumption since $u_i$ and $v_i$ lie in the same $K$-coset, there is $k_i\in K$ such that $u_i=k_iv_i$. Then
\[\Psi(u)=(u_1,\ldots, u_d)=(k_1v_1,\ldots, k_dv_d)=(k_1,\ldots,k_d)(v_1,\ldots, v_d).\]
Since $K \times \dots \times K \preceq K$, there is $k\in K$ such that $\Psi(k)=(k_1,\ldots,k_d)$ and thus, for $v=k^{-1}u$ we have 
\[\Psi(v)=\Psi(k)^{-1}\Psi(u)=(v_1,\ldots,v_d)\]
as required.
\end{proof}

The key component of the algorithm is the following Lemma
\begin{lemma}
\label{lem:lifting}
Suppose $G$ is a regular branch group generated by a finite set $S$ acting on a $d$-ary rooted tree $X^*$. Then, given a tuple $(g_1,g_2,\ldots,g_d)\in G^d$ and $\sigma\in\Sym(X)$ there is an algorithm deciding in a linear time of $\sum_{i=1}^d |g_i|_S$ if there is $g\in G$ such that $g=(g_1,g_2,\ldots,g_d)\sigma$, and finding a word in $S^{\pm1}$ representing $g$ of length at most $D\cdot\left(\sum_{i=1}^d |g_i|_S\right)+D$ for some $D>0$ not depending on $g_i$'s.
\end{lemma}

\begin{proof}
Suppose $K$ is the branching subgroup in $G$. We first observe
\begin{equation}
\label{eqn:index}
K\times\cdots\times K<\Psi(\St_G(1))<G\times\cdots\times G.
\end{equation}
Since $[G:K]<\infty$ we get $[G\times\cdots\times G:K\times\cdots\times K]<\infty$ and thus~\eqref{eqn:index} yields $[G\times\cdots\times G:\Psi(\St_G(1))]<\infty$. Consider the Schreier graph $\widetilde\Gamma=\Gamma(G^d, \Psi(\St_G(1)),\tilde S)$ of $\Psi(\St_G(1))$ in $G^d$ with respect to the generating set $\tilde S$ of $G^d$:
\[\tilde S=\{(1,\ldots,1,s,1,\ldots,1)\colon s\in S, 1\leq i\leq d\},\]
where the nontrivial component in the expression above is at position $i$.

With the help of $\widetilde\Gamma$ we can determine if for a tuple $(g_1,\ldots,g_d)\in G^d$, where $g_i$'s are given by words over $S^{\pm1}$, there exists $g\in\St_G(1)$ with $\Psi(g)=(g_1,\ldots,g_d)$. This can be done in a linear time in $\sum_{i=1}^d|g_i|_S$ as we have to follow the path in $\widetilde\Gamma$ of that length. Moreover, for each particular tuple $(g_1,\ldots,g_d)\in\Psi(\St_G(1))$ we can find $g\in\Psi^{-1}(G^d)$ as a word in $S^{\pm1}$ simply by enumerating all elements in $\St_G(1)$ and checking whether their images under $\Psi$ coincide with $(g_1,\ldots,g_d)$. 

Let $\mathcal T$ denote the right transversal of $K$ in $G$ and let also $L$ be the maximal length of words representing elements of $\mathcal T$. Then $\mathcal T^d<G^d$ is a finite list of tuples such that for each $(t_1,\ldots,t_d)$ we can precompute if there exists $g\in\St_G(1)$ such that $\Psi(g)=(t_1,\ldots,t_d)$ and find the corresponding word in $S^{\pm1}$ representing such a $g$. Let $N$ denote the maximal length of all of such words.

Let $P$ be the finite subgroup of $\Sym(X)$ generated by permutations on the first level of the tree $X^*$ of all generators of $G$. For each $\sigma\in P$ we find an element $h_{\sigma}\in G$ as a word in $S^{\pm1}$ that induces permutation $\sigma$ on the first level of $X^*$. Let also $K$ be the maximal length of all $h_\sigma$, $\sigma\in P$ and of all their sections on the first level.

Now, suppose we have a tuple $(g_1,\ldots, g_d)\in G^d$, where $g_i$'s are given as words in $S^{\pm1}$, and $\sigma\in\Sym(X)$. This defines an automorphism $g=(g_1,\ldots, g_d)\sigma$ of $X^*$ that may or may not be in $G$. If $\sigma\notin P$, then $g\notin G$.
Otherwise, let $h=h_\sigma=(h_1,\ldots,h_d)\sigma\in G$ and let $w_1$ be a precomputed word in $S^{\pm1}$ of length at most $K$  representing $h$ with the sections of $h$ also given by words of length at most $K$. By construction, $g\in G$ if and only if
\[gh^{-1}=(g_1h^{-1}_{\sigma(1)},\ldots, g_dh^{-1}_{\sigma(d)})\in\St_G(1).\]

Using the Schreier graph $\Gamma$ of $K$ in $g$ for each $1\leq i\leq d$ we can find in a linear time in $|g_ih^{-1}_{\sigma(i)}|_S\leq|g_i|_S+K$ the coset representative $t_i$ of $Kg_ih^{-1}_{\sigma(i)}$ and define $k_i=g_ih^{-1}_{\sigma(i)}t_i^{-1}\in K$ and note that $k_i$ are words in $S^{\pm1}$ with
\[|k_i|_S=|g_ih^{-1}_{\sigma(i)}t_i^{-1}|_S\leq |g_i|_S+|h_{\sigma(i)}|_S+|t_i^{-1}|_S\leq |g_i|_S+K+L.\] 
Then we can write
\[gh^{-1}=(k_1t_1,\ldots,k_dt_d)=(k_1,\ldots,k_d)(t_1,\ldots,t_d).\]

By Lemma~\ref{lem:uniform_lifting} $gh^{-1}\in G$ if and only if $(t_1,\ldots,t_d)\in G$, which, as described above, can be precomputed. If $(t_1,\ldots,t_d)\notin G$ then $gh^{-1}\notin G$ and $g\notin G$. Otherwise, there exists a (precomputed) word $w_2$ in $S^{\pm1}$ of length at most $N$ such that \begin{equation}
\label{eqn:t}
    \Psi(w_2)=(t_1,\ldots,t_d).
\end{equation}

By Lemma~\ref{lem:uniform_liftingK} there is a constant $C>0$ (not dependent on $k_i$'s) and a word $w_3$ in $S^{\pm1}$ of length 
\[|w_3|_S\leq C\sum_{i=1}^d |k_i|_S\leq C\sum_{i=1}^\partial(|g_i|_S+K+L)=C\sum_{i=1}^d|g_i|_S+dC(K+L),\]
such that 
\begin{equation}
\label{eqn:k}
\Psi(w_3)=(k_1,\ldots,k_d).
\end{equation}
Finally, since 
\[g=gh^{-1}\cdot h=(k_1,\ldots,k_d)(t_1,\ldots,t_d)h ,
\]
we obtain a word $w=w_3w_2w_1$ in $S^{\pm1}$ representing $g$ of length 
\[|w|_S\leq |w_3|_S+|w_2|_S+|w_1|_S\leq C\sum_{i=1}^d|g_i|_S+ \bigl(dC(K+L) + N + K\bigr).\]
We finish the proof by setting
\[D=\max\{C,dC(K+L) + N + K\}.\]
\end{proof}

In the discussion above nucleus portraits are defined only for elements of contracting self-similar groups. However, we can extend this definition to a more general case and describe by nucleus portraits arbitrary automorphisms of $T_d$ that contract to the nucleus of $G$ along every branch of the tree. 

\begin{definition}
\label{def:cont_closure}
The \emph{contracting closure} $\mathcal N G$ of a contracting self-similar group $G<\Aut(T_d)$ with the nucleus $\mathcal N$ is a subgroup of $\Aut(T_d)$ consisting of automorphisms of $T_d$ that contract to $\mathcal N$ along every branch of the tree. In other words,
\[\mathcal NG=\{g\in\Aut(T_d)\mid \exists N\geq 1\ \text{ such that } \forall v\in X^N\colon g|_v\in\mathcal N\}.\]
\end{definition}

We define the nucleus portraits of elements of $\mathcal NG$ in the same way as for elements of $G$. Clearly, every element of $\mathcal NG$ can be uniquely defined by its nucleus portrait. Note that the fact that $\mathcal NG$ is a subgroup of $\Aut(T_d)$ follows from Algorithms~1 and~2 in Subsection~\ref{ssec:using_contracting}.
The following theorem addresses the membership problem of $G$ in $\mathcal NG$.  

\begin{theorem}
\label{thm:word_from_portrait_depth}
Let $G$ be a contracting self-similar regular branch group acting on $T_d$, generated by a symmetric set $S$ that contains the nucleus of $G$. There is an algorithm that, given a nucleus portrait of of an element $g\in\mathcal NG$ of depth $t$, decides whether $g\in G$ and in this case constructs a word in $S$ of length exponential in $t$, representing $g$. The complexity of this algorithm is exponential in $t$.
\end{theorem}

\begin{proof}
We will build the word representing $g$ from bottom up using Lemma~\ref{lem:lifting}. In the worst case, all leaves of the portrait of $g$ are at level $t$. Therefore, at the level $t$ all sections of $g$ belong to the nucleus, so their length with respect to $S$ is equal to 1. For every vertex of level $t-1$ by Lemma~\ref{lem:lifting} we can construct words of length $Dd+D$ representing section of $g$ at that vertex. Applying Lemma~\ref{lem:lifting} on each level, we inductively construct words of length up to
\[(Dd)^i+(Dd)^{i-1}D+(Dd)^{i-2}D+\cdots+(Dd)D+D=(Dd)^i\left(1-\frac{D}{1-Dd}\right)+\frac{D}{1-Dd}\]
that represent the sections of $g$ at vertices of level $t-i$, $0\leq i\leq t$.

Therefore, at level $0$, for $i=t$, we get a word $w$ representing $g$ of length
\begin{equation}
\label{eqn:exp_length_bound}
|w|_s\leq (Dd)^t\left(1-\frac{D}{1-Dd}\right)+\frac{D}{1-Dd},
\end{equation}
which is exponential in $t$ as desired.

The exponential-time complexity estimate also follows from iterative applications of Lemma~\ref{lem:lifting}.
\end{proof}


The next theorem provides implicit justification for the use of the portraits in cryptographic applications.  As stated in the introduction, we propose to use nucleus portraits to transmit elements of a contracting group since sending them as words in generators would potentialy reveal some information about the secret conjugator. Therefore, it is natural to ask if this is possible for the adversary to recover some word representing an element from the nucleus portrait of this element. Generally, we believe that this is a difficult task and no such fast algorithm is known. Theorem~\ref{thm:word_from_portrait_depth} implies that such an algorithm exists in the class of regular branch contracting groups, and that the length of the constructed word is polynomial in the length of the original word that was used to build the portrait. More specifically, we prove the following theorem.

\begin{theorem}
\label{thm:word_from_portrait}
Let $G$ be a contracting self-similar regular branch group acting on $T_d$, generated by a symmetric set $S$ that contains the nucleus of $G$. There is an algorithm that, given the nucleus portrait of an element $g\in G$ that can be defined by an (unknown) word in $S$ of length up to $n$, constructs a word in $S$ of polynomial in $n$ length, representing $g$. The complexity of this algorithm is polynomial in $n$.
\end{theorem}



\begin{proof}
Since $G$ is contracting, by Proposition~\ref{prop:log_depth} there is $R>0$ such that for each element of length $n$ the depth of nucleus portrait of this element is at most 
\begin{equation}
\label{eqn:t_from_n}
t=\ceil{R\log_2 n+R}.
\end{equation} 
Using the argument from the proof of Theorem~\ref{thm:word_from_portrait_depth} we build the word $w$ of length $|w|_S$ satisfying equation~\eqref{eqn:exp_length_bound} that  represents $g$. Expressing $t$ in terms of $n$ via equation~\eqref{eqn:t_from_n} yields:
\begin{multline*}
|w|_s\leq (Dd)^t\left(1-\frac{D}{1-Dd}\right)+\frac{D}{1-Dd}\leq (Dd)^{R\log_2 n+R+1}\left(1-\frac{D}{1-Dd}\right)+\frac{D}{1-Dd}\\
=(Dd)^{\log_{Dd}n^{R\log_2(Dd)}}(Dd)^{R+1}\left(1-\frac{D}{1-Dd}\right)+\frac{D}{1-Dd} \\
= n^{R\log_2(Dd)}(Dd)^{R+1}\left(1-\frac{D}{1-Dd}\right)+\frac{D}{1-Dd},
\end{multline*}
which is polynomial in $n$ as desired. The polynomial-time complexity estimate also follows from iterative applications of Lemma~\ref{lem:lifting}.
\end{proof}

\section{Length-Based Attack}
\label{sec:LBA}
The length-based attack (LBA) is a probabilistic attack that can be applied against many computational problems and cryptographic protocols, in particular those that are based on the variants of the SCSP, such as AAG and Ko-Lee protocols described in Section~\ref{sec:cryptosystems}.

Different variants of LBA against SCSP and other related problems were studied in the cases of braid groups~\cite{HughesTann02,GarberProbabilistic05,GarberLBA06,myasnikov_u:length-based_attack_on_braid_groups07}, Thompson's group~\cite{Thompson-LBA}, free groups and groups with generic free basis property~\cite{myasnikov_u:random_subgroups_and_analysis_of_length-based_and_quotient_attacks08}, and polycyclic groups~\cite{graber_kn:length-based_attack15}.

Let $G$ be a finitely generated group. By a \emph{length function} on $G^n$ we will call a function $l\colon G^n\to \mathbb N\cup\{0\}$ satisfying $l\big((g_1,\ldots,g_n)\big)=0$ if and only if $g_i=1_G$ for all $i$. The idea of LBA against the SCSP in $G$ (respectively, AAG problem over $G$) is the following. Given two conjugated tuples of group elements $\overline{a}=(a_1,\ldots,a_n)$ and $\overline{b}=(b_1,\ldots,b_n)$ LBA tries to find a conjugator $r$ in $G$ (respectively, in $\langle a_1,\ldots,a_n\rangle$) that conjugates $\overline{a}$ to $\overline{b}$ by iteratively choosing elements $r_1,r_2,\ldots, r_t$ from a precomputed finite set $\mathcal R$ such that for some predefined length function $l$:
\[l\left(\overline{a}^{r_1r_2\cdots r_kr_{k+1}}\overline{a}^{-1}\right)<l\left(\overline{a}^{r_1r_2\cdots r_k}\overline{b}^{-1}\right)\]
for all $0\leq k< t$ and such that 
\[l\left(\overline{a}^{r_1r_2\cdots r_{t}}\overline{b}^{-1}\right)=0,\] 
in which case $r=r_1r_2\cdots r_t$ is the required conjugator.

Since the set $\mathcal R$ is finite and $l(\overline{a}^{r_1r_2\cdots r_k}\overline{b}^{-1})$ decreases discretely by at least 1 every time, the whole attack represents a finite search. Note that under the assumption that the search is completed in its entirety, there is no difference in the success rates of the attacks but there may be some differences in time required to conduct these attacks depending on implementation details. In Algorithm~\ref{alg:lba} we recurrently go through all possible branches as soon as the length is decreased. Some of the cited implementations also impose additional memory restrictions that will affect the success rate, but also limit the time/space required for the attack. In our implementation we do not have any such limitations.

We will call elements of $\mathcal R$ \emph{conjugator factors}. The choice of the set $\mathcal R$ varies in different versions of LBA to improve the speed or effectiveness of the attack. In some versions of the attack only finite amount of memory is allocated, so that the search is not done over the entire search space. The algorithm for LBA that we implement and analyze in this paper, Algorithm~\ref{alg:lba}, does not make any assumptions on the finite memory size and includes one additional parameter $R=SearchRadius$ controlling the set $\mathcal R$ of conjugating factors. Namely, before running the attack we precompute $\mathcal R$ as the set of nontrivial elements in the ball $\mathcal B_R$ of radius $R$ in $G$. Thus, for $R=3$, our attack covers all of the attacks listed in~\cite{graber_kn:length-based_attack15} in the sense that the percentage of successes will be not smaller than the one in the listed attacks. We will also refer to the set $\mathcal B_R$ as the \emph{search ball} of the attack. Our source code for the attack written in GAP can be found in the GitHub repository~\cite{kahrobaei_ms:lba_github}.

\subsection{Complexity estimate for the LBA} 
Suppose we are given an instance of SCSP with the initial value of the length function $l(\overline{a}\overline{b}^{-1})=L$. To estimate the worst case complexity of LBA we observe that each step of the attack can be parameterized by a word over $\mathcal R$ of length at most $L$ since every additional conjugator $r\in\mathcal R$ reduces the length function by at least $1$. Therefore, the total number of steps that the attack can take is at most 
\begin{equation}
\label{eqn:complexity}
\sum_{i=0}^{L}|\mathcal R|^i=\frac{|\mathcal R|^{L+1}-1}{|\mathcal R|-1}.
\end{equation}
Thus, the worst time complexity will be exponential in $L$. In practice the number of steps is much smaller as the worst case complexity estimate assumes that each additional conjugator $r\in\mathcal R$ will decrease the length function exactly by 1, so that we have to go through all possible branches of the algorithm. However, if $r$ does not reduce the length, then the whole branch of the algorithm will be terminated at this stage.

In our implementation of LBA, as stated above, we use $\mathcal R=\mathcal B_R$ for some radius $R\geq 1$. Estimate~\eqref{eqn:complexity} suggests that the LBA may generally work faster for larger values of $R$ for groups that have smaller $\mathcal B_R$ (for example, if the group grows subexponentially or when the generators are involutions). However, we observed that this is not a sufficient condition as we found some instances of SCSP in $\textrm{IMG}(z^2+i)$ with $|r|=20$, $|a|=10$, and $R=5$ (with $|\mathcal B_5|=421$) that were terminated due to longer than 24 hour computation time.


\subsection{Length functions for LBA}
In the implementation of LBA, the choice of the length function is important (see, for example, \cite{GarberProbabilistic05,HT10}). Historically, the sum of the word lengths of the elements of was the main length function used in LBA attacks in the literature. However, since we propose to use nucleus portraits to uniquely identify elements of contracting groups, this length function is not available in this situation. Instead, we propose 2 length functions that can be computed from nucleus portraits of the involved elements.

\subsubsection{Portrait depth length function}
The main length function that we study here is the \emph{portrait depth length function}: 
\begin{equation}
\label{eqn:lengthd}
l_d\big((g_1,\ldots,g_n)\big)=\sum_{i=1}^n\partial(g_i),
\end{equation}
which is just the sum of the depths of the nucleus portraits of all the $g_i$'s.

\subsubsection{Portrait boundary size length function}

Another candidate for complexity function when group elements are given as nucleus portraits is the  \emph{portrait boundary size length function} based on the number of leaves of a nucleus portrait $s(g)$ as defined in Definition~\ref{portrait-depth-def}:
\begin{equation}
    \label{eqn:lengthb}
    l_b\big((g_1,\ldots,g_n)\big)=\sum_{i=1}^ns(g_i).
\end{equation}
While the range of values of this function is significantly larger than the one for the portrait depth length function (for the same input size), our experimental results showed that using $l_b$ for LBA was infeasible for even relatively short tuples $(g_1,\ldots,g_n)\in G^n$. The attack took too much time as there was a lot of branching of the algorithm involved.

\begin{algorithm} 
\caption{The function LengthBasedAttack implementing the LBA}\label{alg:lba} 
\begin{algorithmic}[1]
\Function{LengthBasedAttack}{$a,b,SearchRadius$}
\State $B\gets $ (List of nontrivial elements in $G$ of length $\leq SearchRadius$) 
\Comment{compute conjugator factors}
\State $InitialConjugator\gets 1$\Comment{Initiate the candidate for the conjugator}
\State $InitialLength\gets$\Call{Length}{$a^{InitialConjugator}b^{-1}$}
\newline 
\Function{LengthBasedAttackRecursive}{$CurrentConjugator, CurrentLength$}
    \If{$CurrentLength=0$}
        \State \Return{$CurrentConjugator$}
    \EndIf
    \For{$ConjugatorFactor$ in $B$}
        \State $NewConjugator\gets CurrentConjugator * ConjugatorFactor$
        \State $NewSize\gets$\Call{Length}{$a^{NewConjugator}b^{-1}$}
        \If{$NewLength<CurrentLength$}
            \State $Outcome\gets$\Call{LengthBasedAttackRecursive}{$NewConjugator, NewSize$}
            \If{$Outcome\neq$ Fail}
                \State\Return{$Outcome$}
            \EndIf
        \EndIf
    \EndFor
    \State\Return{Fail}
\EndFunction 
\newline 
\State \Return{\Call{LengthBasedAttackRecursive}{$InitialConjugator, InitialLength$}}
\EndFunction
\end{algorithmic} 
\end{algorithm}

\section{Experimental Results for LBA against SCSP in Contracting Groups}
\label{sec:statistics}

In this section we present experimental results for LBA against SCSP when group elements are only provided as nucleus portraits. The results incorporate variations in parameters of the SCSP and LBA such as the size of conjugator and radius $R$ of search ball $\B_R$ among others. Furthermore, groups are chosen to ensure variation in nucleus size and in complexity of algebraic structure to identify conspicuous associations between these parameters and degree of success of the attacks. In this regard, we performed calculations for groups with nucleus size ranging from 5 to 26 and the algebraic structure varying from virtually abelian groups to more complex group structures. Some of the selected groups are taken from the classification of groups generated by 3-state automata over 2-letter alphabet~\cite{bondarenko_gkmnss:full_clas32}. In that paper, every 3-state automaton over 2-letter alphabet is assigned a number $m$ from 1 to 5832 and is denoted by $\mathcal A_m$. The group generated by $\mathcal A_m$ is denoted by $\mathbb G(\mathcal A_m)$. Note that the notation in~\cite{bondarenko_gkmnss:full_clas32} was different: unlike in the present paper, left actions were used there, which explains the difference in group descriptions (e.g., the groups $(\Z \times \Z)\rtimes C_2$ and $\Z\wr C_2$, where $C_2$ is the acting group, are denoted in~\cite{bondarenko_gkmnss:full_clas32} by $C_2\ltimes (\Z \times \Z)$ and $C_2\wr\Z$, respectively). Additionally, all words in generators given in~\cite{bondarenko_gkmnss:full_clas32} need to be reversed in the notation of the present paper.

\subsection{Notation for the attack} 
\label{notation}

\subsubsection{Simultaneous Conjugacy Search Problem (SCSP)}

We will describe our results using the notation for SCSP given in the Introduction. Namely, we assume that we have two $n$-tuples $(a_1, a_2 \ldots, a_n)$ and $(b_1,b_2,\ldots,b_n)$ of elements of $G$ that are conjugate by an element $r\in G$ (i.e., $a_i^r=b_i$ for $1\leq i\leq n$). Following the description of LBA in Section~\ref{sec:LBA}, the radius of the search ball $\B_R$ is referred to as $R$.

\subsubsection{Runtime}

We report the time taken for computation as GAP-seconds. This is the \textit{runtime} in GAP (recorded in milliseconds) converted into seconds and rounded to the nearest integer. As per GAP-Reference Manual, Section~7.6, this is the time spent by the Central Processing Unit (CPU) with GAP functions (without child processes). We note that both, the overall time taken and GAP-seconds for calculations vary depending on the system specifications. 

Table~\ref{runtime} provides specifications of the system used for running the LBA and the comparison of runtime in GAP with the actual time taken for the computations.

\hspace{1cm}
\begin{table}[h]
\centering
\begin{tabular}{|p{6.0 cm}||c|c|c|c|c|}
    \hline
    
    Processor-specifications & \multicolumn{5}{|c|}{Intel(R) Core(TM)i7-5500U CPU @ 2.40GHz }  \\
    \hline
    RAM & \multicolumn{5}{|c|}{Installed RAM = 8.00 GB - Usable RAM = 7.89 GB} \\
    \hline
    System type & \multicolumn{5}{|c|}{64-bit operating system, x64 - based processor} \\
    \hline
    \hline
    GAP-Runtime & 263562 & 206844  & 253406 & 116250 & 191781 \\
    \hline
    Overall time (ms) & 264938 & 207710  & 253660 & 116424 & 191931 \\
    \hline 
    Time difference (ms) & 1376 & 866 &  254 & 174  & 150 \\
    \hline
    Time difference as \% of GAP-Runtime & 0.52\% & 0.42\% & 0.10 \% & 0.15 \% & 0.08\% \\ 
    \hline
\end{tabular}
\caption{ Comparison of GAP-Runtime with overall time taken in milliseconds to complete computations.} 
\label{runtime}
\end{table}

\subsubsection{Parameters}

We use the following set of parameters for each computation (corresponding to each table entry) recorded in the following subsections: 
\begin{itemize}
    \item In all experiments in this section we use the portrait depth length function $l_d$ defined in~\eqref{eqn:lengthd}.
    \item In all experiments we set $n = 5$, i.e., computations are performed using randomly generated 5-tuples $(a_1, a_2, \ldots, a_5)$ from the respective groups.
    \item Random elements of groups are generated by random freely reduced words over the symmetric generating sets of groups with additional general reduction that for each generator $s$ of a group of order 2, $s^{-1}$ is replaced by $s$. Additionally, in the case of Grigorchuk group and Universal Grigorchuk group, the random words are reduced further using the fact that $\langle b,c,d\rangle$ is the Klein group of size 4, to the standard reduced words of the form $(*)a*a*\cdots *a(*)$ with $*\in\{b,c,d\}$.
    \item The reduced word length of each $a_i$, that we denote by $|a|$, was set to either $10$ or $100$.
    \item Based upon an automaton group, the radius $R$ of search ball $\B_R$ is varied to analyze efficacy of the attack. Overall, the search radius varies from $1$ to $5$ in our study. Larger values of $R$ increase the success rate of the attack but considerably slow down the calculations. 
    \item The reduced word length $|r|$ of the conjugator $r$ is varied from $5$ up to $100$ to examine effectiveness of the attack.
\end{itemize} 

\subsection{The group $\mathbb G(\A_{750})\cong\Z\wr C_2$}
The automaton generating this group is shown in Figure~\ref{fig:750} and is of exponential growth in the sense of Sidki. The nucleus size $|\N|$ is 9. The group itself is virtually abelian with all the sections of generators belonging to an abelian subgroup $\langle a,c\rangle$. The nucleus $\N$ can be found using \verb"AutomGrp" as follows: 

\vspace{2mm}

\begin{verbatim}
G_750:= AutomatonGroup("a = (c,a)(1,2), b = (c,a), c = (a,a)");
< a, b, c >
gap> GeneratingSetWithNucleus(G_750);
[ 1, a, b, c, a^-1, b^-1, c^-1, c*a^-1, a*c^-1 ]
\end{verbatim}
\vspace{2mm}

Experimental results for LBA attack for $\mathbb G(\A_{750})$ are recorded in Table~\ref{table-750}. 

We note that, given the parameters, the smallest non-trivial ball $\B_1$ (that only contains the generators and their inverses) succeeds about half of the time in finding the conjugator $r$. Moreover, increasing the radius $R$ of the search ball $\B_R$ from 1 to 2 increases the success of attack by at least 20 percentage points resulting in at least 70 percent success. This shows that the attack works very well for this contracting group. Given the very high success rate across our choice of parameters for SCSP, we have omitted results for search balls with larger radii for this group. The time taken to perform iterations for $\A_{750}$ is relatively small compared to most other contracting groups in our study (when compared with similar test parameters of course).

\begin{figure}
    \centering
    \begin{tikzpicture}

    \node[state] (q1) {$a$};
    \node[state, right of=q1] (q2) {$b$};
    \node[state] (q3) at (1.5, 2.5) {$c$};

    \draw (q1) edge[bend left , left] node{$0|1$} (q3);
    \draw (q1) edge[loop left] node{$1|0$} (q1);

    \draw (q2) edge[bend left ,below] node{$1|1$} (q1);
    \draw (q2) edge[bend right , right] node{$0|0$} (q3);

    \draw (q3) edge[right] node{$0|0 , 1|1$} (q1);

    \end{tikzpicture}
\caption{Automaton $\A_{750}$ generating the group $\Z \wr C_2$}
    \label{fig:750}
\end{figure}
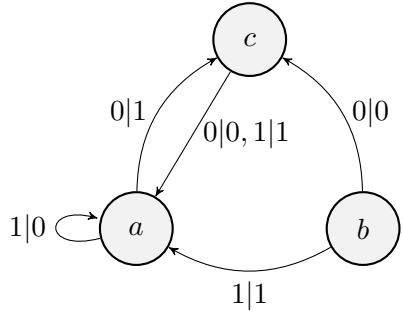

\hspace{1cm}
\small{
\begin{table}[h!]
\centering
\begin{tabular}{|p{6.2cm}||c|c||c|c||c|c|}
    \hline
    Automaton group & \multicolumn{6}{|c|}{$\mathbb G(\mathcal A_{750})\cong\Z \wr C_2$} \\
    \hline
    \hline
    Reduced word length of conjugator $r$ & \multicolumn{2}{|c||}{20} & \multicolumn{2}{|c||}{30} & \multicolumn{2}{|c|}{100} \\ 
    \hline
    Radius $R$ of the search ball $\B_R$ & 1 & 2 & 1 & 2 & 1 & 2 \\
    \hline
    Percentage of success (\%) & 70.00 & 96.67 & 60.00 & 86.67 & 46.67 & 70.00 \\
    \hline
    GAP-Seconds (average) & 10.87 & 13.53 & 8.26 & 965.33 & 157.73 & 525.63  \\
    \hline 
\end{tabular}

\caption{Average time taken and success rate for 30 iterations each for different radii of search ball $\mathcal{B}_R$ and reduced word length of conjugator  for the contracting automaton group $\Z \wr C_2$ generated by $\A_{750}$; $|a| = 10$ and $n = 5$ for all iterations.} 
\label{table-750}
\end{table}
}

\subsection{The group $\mathbb G(\A_{775})\cong \textrm{IMG} \left( \left( {\frac{z-1}{z+1}}\right)^2\right)\rtimes C_2$}

The automaton generating this group is shown in Figure~\ref{fig:775} and is of exponential growth in the sense of Sidki. The nucleus size $|\N|$ is 8. The nucleus $\N$ can be found using \verb"AutomGrp" as follows: 

\vspace{2mm}

\begin{verbatim}
G_775 := AutomatonGroup("a = (a,a)(1,2), b = (c,b), c = (a,a)");
< a, b, c >
gap> GeneratingSetWithNucleus(G_775);
[ 1, a, b, c, a*b, a*c, b^-1*a^-1, a*b^-1*a^-1 ]
\end{verbatim}
\vspace{2mm}

Experimental results for LBA attack for $\mathbb G(\A_{775})$ are recorded in Table~\ref{table-775}.

Although both contracting groups, generated by automata $\A_{750}$ and $\A_{775}$ have comparable nucleus sizes $|\N|$, there is a notable difference in the success of the LBA. This is evident by little to no success when the conjugator has a relatively large reduced word length such as $|r| = 100$. Here we observe that although there is increase in the success rate and time taken for LBA as we increase the search radius $R$, the pairs of search balls $\{\B_2 , \B_3 \}$, and $\{\B_4 , \B_5\}$ produce very similar results.    
  
Given the low success rate of LBA for the aforementioned parameters, we performed further simulations for a different set of parameters that have significantly greater reduced length $|a|$ of conjugated elements. These results are recorded in Table~\ref{tab:table-775-big}. In our simulations LBA is completely unsuccessful for these parameters for $\mathbb G(\A_{775})$.

\hspace{1cm}
\begin{center}
\begin{table}
\centering
\setlength{\tabcolsep}{1.8pt}
\begin{tabular}{|p{5.1cm}||c|c|c|c||c|c|c|c||c|c|c|c|}
    \hline    
    Automaton group & \multicolumn{12}{|c|}{$\mathbb G(\mathcal A_{775})\cong \textrm{IMG} \left( \left( {\frac{z-1}{z+1}}\right)^2\right)\rtimes C_2$} \\
    \hline
    \hline
    Reduced word length of conjugator $r$ & \multicolumn{4}{|c||}{20} & \multicolumn{4}{|c||}{30} & \multicolumn{4}{|c|}{100} \\
    \hline
    Radius $R$ of the search ball $\B_R$ & 2 & 3 & 4 & 5 & 2 & 3 & 4 & 5 & 2 & 3 & 4 & 5 \\
    \hline
    Percentage of success (\%) & 10 & 13.33 & 46.67 & 46.67 & 13.33 & 13.33 & 33.33 & 36.67 & 0 & 0 & 3.33 & 0.00 \\
    \hline
    GAP-seconds (average) & 3.7 & 9.97 & 32.67 & 47.9 & 7.23 & 10.93 & 62.5 & 83.8 & 11.3 & 28.23 & 72.47 & 234 \\

    \hline 
\end{tabular}
\caption{Average time taken and success rate for 30 iterations each for different radii of search ball $\mathcal{B}_R$ and reduced word length of conjugator  for the automaton group $\textrm{IMG} \left( \left( \frac{z-1}{z+1}\right)^2\right)\rtimes C_2$ generated by $\A_{775}$; $|a| = 10$ and $n = 5$ for all iterations.} 
\label{table-775}
\end{table}
\end{center}


\hspace{1cm}
\begin{table}[h]
    \centering
    \begin{tabular}{|p{8.0cm}||c|c|c|}
    \hline
    Automaton group & \multicolumn{3}{|c|}{$\mathbb G(\mathcal A_{775})\cong \textrm{IMG} \left( \left( {\frac{z-1}{z+1}}\right)^2\right)\rtimes C_2$} \\
    \hline
    \hline
    Reduced word length of conjugator $r$ & 100 & 100 & 100  \\
    \hline
    Reduced word length of conjugated elements, $|a|$ & 100 & 100 & 100  \\
    \hline
    Radius $R$ of the search ball $\B_R$ & 2 & 3 & 4 \\
    \hline 
    Percentage of success (\%) & 0 & 0 & 0 \\
    \hline
    GAP-seconds (average) & 64.6 & 274.17 & 3009.33 \\
    \hline
    \end{tabular}
    \caption{Average time taken and success rate for 30 iterations each for different radii of search ball $\mathcal{B}_R$ and SCSP parameters  for the automaton group $\textrm{IMG} \left( \left( {\frac{z-1}{z+1}}\right)^2\right)\rtimes C_2$ generated by $\A_{775}$. Number of conjugated elements $n = 5$ for all iterations.}
    \label{tab:table-775-big}
\end{table}



\begin{figure}
    \centering
    \begin{tikzpicture}

    \node[state] (q1) {$a$};
    \node[state, right of=q1] (q2) {$b$};
    \node[state] (q3) at (1.5, 2.5) {$c$};

    \draw (q1) edge[loop left] node{$1|0 , 0|1$} (q1);

    \draw (q2) edge[loop right] node{$1|1$} (q2);
    \draw (q2) edge[right] node{$0|0$} (q3);

    \draw (q3) edge[left] node{$0|0 , 1|1$} (q1);

    \end{tikzpicture}
\caption{Automaton $\A_{775}$ generating the group $\textrm{IMG} \left( \left( {\frac{z-1}{z+1}}\right)^2\right)\rtimes C_2$.}
    \label{fig:775}
\end{figure}
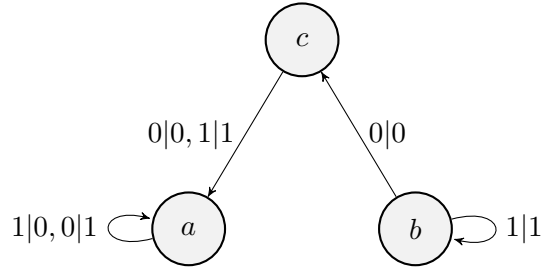

\subsection{The group $\mathbb G(\A_{2277})\cong (\Z \times \Z)\rtimes C_2$}

The automaton generating this group is shown in Figure~\ref{fig:2277} and is of exponential growth in the sense of Sidki. The nucleus size $|\N|$ is 16. The group contains $\Z^2$ as a subgroup of index 2. The nucleus $\N$ can be found using \verb"AutomGrp" as follows: 

\vspace{2mm}

\begin{verbatim}
gap> G_2277:= AutomatonGroup("a = (c,c)(1,2), b = (a,a)(1,2), c = (b,a)");
< a, b, c >
gap> GeneratingSetWithNucleus(G_2277);
[ 1, a, b, c, a*b, c*a, b*c, a*c, c*b, b*a, b*c*a, a*b*c, c*a*b, a*c*a, c*a*c, c*b*c ]
\end{verbatim}
\vspace{2mm}

Experimental results for LBA attack for $\mathbb G(\A_{2277})$ are recorded in Table~\ref{table-2277}.

Similarly to $\mathbb G(\A_{750})$, the smallest non-trivial ball $\B_1$ succeeds every time in finding the conjugator $r$. Given the hundred percent success rate across all SCSP parameters displayed in the Table~\ref{table-2277}, no further simulations were performed using search balls with larger radii.
 
\hspace{1cm}
\begin{table}
 \centering
\begin{tabular}{|l||c||c||c|}
    \hline
    
    Automaton group & \multicolumn{3}{|c|}{$\mathbb G(\mathcal A_{2277})\cong(\Z \times \Z)\rtimes C_2$} \\
    \hline
    \hline
    Reduced word length of conjugator $r$& 20 & 30 & 100 \\
    \hline
    Radius $R$ of the search ball $\B_R$ & 1 & 1 & 1  \\
    \hline
    Percentage of success (\%) & 100 & 100 & 100  \\
    \hline
    GAP-seconds (average) & 3.93 & 11.3 & 23.97  \\

    \hline 
\end{tabular}
\caption{
Average time taken and success rate for 30 iterations each for different radii of search ball $\mathcal{B}_R$ and SCSP parameters  for the automaton group $(\Z \times \Z)\rtimes C_2$ generated by $\A_{2277}$; $|a| = 10$ and $n = 5$ for all iterations. } 
\label{table-2277}
\end{table}

\begin{figure}
    \centering
    \begin{tikzpicture}

    \node[state] (q1) {$a$};
    \node[state, right of=q1] (q2) {$b$};
    \node[state] (q3) at (1.5, 2.5) {$c$};

    \draw (q1) edge[bend left , left] node{$0|1 , 1|0$} (q3);

    \draw (q2) edge[bend left , below] node{$0|1 , 1|0$} (q1);

    \draw (q3) edge[bend left , right] node{$0|0$} (q2);
    \draw (q3) edge[right] node{$1|1$} (q1);

    \end{tikzpicture}
\caption{Three-state Automaton $\A_{2277}$ generating the group $(\Z \times \Z)\rtimes C_2$}
    \label{fig:2277}
\end{figure}
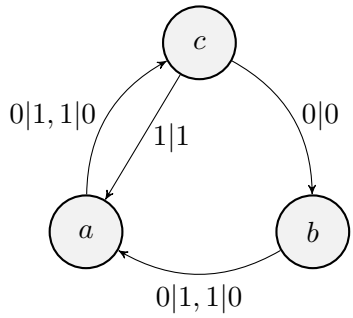

\subsection{The group $\mathbb G(\A_{2287})\cong \textrm{IMG}  \left( {\frac{z^2+2}{1-z^2}}\right)$}

The automaton generating this group is shown in Figure~\ref{fig:2287} and is of exponential growth in the sense of Sidki. The nucleus size $|\N|$ is 26 which is the largest among groups generated by 3-state automata in our study. There is not much information about the algebraic structure of this group. The nucleus $\N$ can be found using \verb"AutomGrp" as follows: 

\vspace{2mm}

\begin{verbatim}
gap> G_2287:= AutomatonGroup("a = (a,a)(1,2), b = (c,a)(1,2), c = (b,a)");
< a, b, c >
gap> GeneratingSetWithNucleus(G_2287);
[ 1, a, b, c, b^-1, c^-1, a*b, a*c, b^-1*a^-1, c^-1*a^-1, a*b^-1, a*c^-1, 
  b*a^-1, c*a^-1, a*b^-1*a^-1, a*c^-1*a^-1, a*b*a^-1, a*c*a^-1, b*a*c^-1, 
  a*b*a*c^-1, c*a^-1*b^-1*a^-1, a*c*a^-1*b^-1*a^-1, b^-1*c, c^-1*b, a*b^-1*c, 
  a*c^-1*b ]
\end{verbatim}
\vspace{2mm}

Experimental results for LBA attack for $\mathbb G(\A_{2287})$ are recorded in Table~\ref{table-2287}. 
Among the 3-state automaton generated groups in our study, the simulations take the greatest time to complete for this group. One reason for this is the big nucleus size $|\N|$. We note that the attack is not successful for modest SCSP parameters. The simulations involving greater search radius $R$ and conjugator reduced length $|r|$ were terminated after 24 hours. We conclude that this group has the highest cost of running LBA attack simulations.

\hspace{1cm}
\begin{table}
\centering
\begin{tabular}{|l||c||c||c|}
    \hline
    Automaton group & \multicolumn{3}{|c|}{$\mathbb G(\mathcal A_{2287})\cong \textrm{IMG}  \left( {\frac{z^2+2}{1-z^2}}\right)$} \\
    \hline
    \hline
    Reduced word length of conjugator $r$ & 5 & 10 & 20 \\
    \hline
    Radius of the search ball $\B_R$ & 2 & 2 & 2  \\
    \hline
    Percentage of success (\%) & 36.67 & 3.33 & 0.00  \\
    \hline
    GAP-seconds (average) & 103.13 & 154.67 & 235.7  \\

    \hline 
\end{tabular}
\caption{Average time taken and success rate for 30 iterations each for different radii of search ball $\mathcal{B}_r$ and reduced word length of conjugator  for the contracting group $\textrm{IMG}  \left( {\frac{z^2+2}{1-z^2}}\right)$ generated by automaton $\A_{2287}$; $|a| = 10$ and $n = 5$ for all iterations.} 
\label{table-2287}
\end{table}

\begin{figure}
    \centering
    \begin{tikzpicture}

    \node[state] (q1) {$a$};
    \node[state, right of=q1] (q2) {$b$};
    \node[state] (q3) at (1.5, 2.5) {$c$};

    \draw (q1) edge[loop left] node{$0|1 , 1|0$} (q1);

    \draw (q2) edge[left] node{$0|1$} (q3);
    \draw (q2) edge[below] node{$1|0$} (q1);

    \draw (q3) edge[bend left , right] node{$0|0$} (q2);
    \draw (q3) edge[left] node{$1|1$} (q1);

    \end{tikzpicture}
\caption{Automaton $\A_{2287}$ generating the group $\textrm{IMG}  \left( {\frac{z^2+2}{1-z^2}}\right)$}
    \label{fig:2287}
\end{figure}
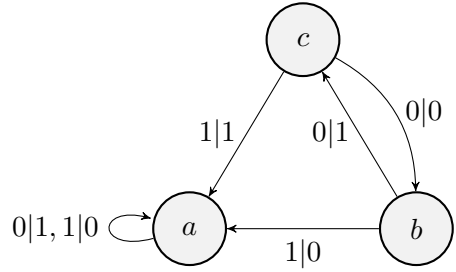

\subsection{The group $\textrm{IMG}(z^2 + i)$}

The group $G$ generated by a 4-state automaton in Figure~\ref{fig:img(z^2+i)} of bounded growth in the sense of Sidki, is isomorphic to $\textrm{IMG}(z^2+i)$ and was shown to have intermediate growth by Bux and P\'erez in~\cite{bux_p:iter_monodromy}. A finite $L$-presentation of $G$ was explicitly calculated by Grigorchuk, \v Suni\'c, and the third author in~\cite{grigorch_ss:img}. In the last paper it was also shown that $G$ is regular branch. The nucleus $\mathcal N$ has size 4 and can be found using \verb"AutomGrp" as follows: 

\vspace{2mm}

\begin{verbatim}
G := AutomatonGroup("a = (1,1)(1,2), b = (a,c), c = (b,1)");
< a, b, c >
gap> GeneratingSetWithNucleus(G);
[ 1, a, b, c ]
\end{verbatim}
\vspace{2mm}

Experimental results for LBA attack for $\textrm{IMG}(z^2+i)$ are recorded in Table~\ref{table-img(z^2+i)}. We also note that we ran some tests for this group for $|r|=20$, $|a|=10$, and $R=5$ and observed a higher success rate (up to 30\%), but also a large time variability. While some of the runs finished successfully, there was one run that we had to terminate after about 50 hours of computation without. The initial length function value for that instance was $38$.

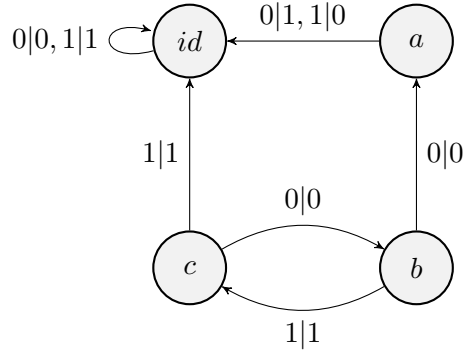
\begin{figure}
    \centering
    \begin{tikzpicture}

    \node[state] (q1) {$c$};
    \node[state, right of=q1] (q2) {$b$};
    \node[state, above of=q2] (q3) {$a$};
    \node[state, above of=q1] (q4) {$id$};

    \draw (q1) edge[bend left , above] node{$0|0$} (q2);
    \draw (q1) edge[left] node{$1|1$} (q4);

    \draw (q2) edge[bend left , below] node{$1|1$} (q1);
    \draw (q2) edge[right] node{$0|0$} (q3);

    \draw (q3) edge[above] node{$0|1 , 1|0$} (q4);

    \draw (q4) edge[loop left] node{$0|0 , 1|1$} (q4);

    \end{tikzpicture}
\caption{Automaton generating the group $\textrm{IMG}(z^2 +i)$}
    \label{fig:img(z^2+i)}
\end{figure}

\hspace{1cm}
\begin{table}[h]
\centering
\begin{tabular}{|p{6.2cm}||c|c||c|c||c|c|}
    \hline    
    Automaton group & \multicolumn{6}{|c|}{$\textrm{IMG}(z^2 + i)$} \\
    \hline
    \hline
    Reduced word length of conjugator $r$ &  \multicolumn{2}{|c||}{10} & \multicolumn{2}{|c||}{20} & \multicolumn{2}{|c|}{30} \\
    \hline
    Radius $R$ of the search ball $\B_R$ & 3 & 4 & 3 & 4 & 3 & 4 \\
    \hline
    Percentage of success (\%) & 16.67 & 33.33 & 0.00 & 3.33 & 0.00 & 0.00 \\
    \hline
    GAP-seconds (average) & 42.27 & 115.63 & 50.6 & 660.2 & 111.87 & 355.1  \\

    \hline 
\end{tabular}
\caption{Average time taken and success rate for 30 iterations each for different radii of search ball $\mathcal{B}_R$ and reduced word length of conjugator  for the contracting group $\textrm{IMG}(z^2 + i)$; $|a| = 10$ and $n = 5$ for all iterations.} 
\label{table-img(z^2+i)}
\end{table}

\subsection{Grigorchuk group $\G$}
\label{ssec:grig_empirical}
Grigorchuk group $\G$ is with no doubt the most famous and studied automaton group. We refer the reader to~\cite{grigorch:solved} for a detailed survey of properties and problems related to this group. We will point out here that it is not a linear group as a group of intermediate growth. It was suggested as a possible platform group for AAG cryptosystem in~\cite{grigorchuk2019keyagreement}. It is generated by the bounded (in the sense of Sidki) automaton displayed in Figure~\ref{fig:aut_grig} and has the nucleus of size 5. The nucleus $\N$ can be found using \verb"AutomGrp" as follows: 

\vspace{2mm}

\begin{verbatim}
gap> G := AG_Groups.GrigorchukGroup;
< a, b, c, d >
gap> AssignGeneratorVariables(G);
#I  Assigned the global variables [ a, b, c, d ]
gap> GeneratingSetWithNucleus(G);
[ 1, a, b, c, d ]
\end{verbatim}
\vspace{2mm}

Experimental results for LBA attack for $\G$ are recorded in Table~\ref{table-grig}.
Given the wide interest in Grigorchuk group, we ran further simulations for additional set of parameters that have significantly greater reduced length $|a|$ of conjugated elements. These results are recorded in Table~\ref{tab:table-grig-big}. In these simulations LBA fails to find a conjugator for SCSP.

\hspace{1cm}
\begin{table}
\setlength{\tabcolsep}{2pt}
\centering
\begin{tabular}{|p{6.2cm}||c|c|c||c|c|c||c|c|c|}
    \hline
    Automaton group & \multicolumn{9}{|c|}{Grigorchuk group $\G$} \\
    \hline
    \hline
    Reduced word length of conjugator $r$& \multicolumn{3}{|c||}{20} & \multicolumn{3}{|c||}{30} & \multicolumn{3}{|c|}{100} \\
    \hline
    Radius $R$ of the search ball $\B_R$ & 2 & 3 & 4 & 2 & 3 & 4 & 2 & 3 & 4 \\
    \hline
    Percentage of success (\%) & 0.00 & 3.33 & 0.00 & 0.00 & 3.33 & 3.33 & 0.00 & 0.00 & 0.00 \\
    \hline
    GAP-seconds (average) & 3.63 & 14.5 & 40.53 & 4.5 & 22.23 & 39.8 & 6.7 & 13.9 & 88.13 \\

    \hline 
\end{tabular}
\caption{Average time taken and success rate for 30 iterations each for different radii of search ball $\mathcal{B}_R$ and reduced word length of conjugator  for the Grigorchuk group; $|a| = 10$ and $n = 5$ for all iterations.} 
\label{table-grig}
\end{table}

\hspace{1cm}
\begin{table}[h]
    \centering
    \begin{tabular}{|p{8.0cm}||c|c|c|}
    \hline
    Automaton group & \multicolumn{3}{|c|}{Grigorchuk group} \\
    \hline
    \hline
    Reduced word length of conjugator $r$ & 100 & 100 & 100  \\
    \hline
    Reduced word length of conjugated elements, $|a|$ & 100 & 100 & 100  \\
    \hline
    Radius $R$ of the search ball $\B_R$ & 2 & 3 & 4 \\
    \hline 
    Percentage of success (\%) & 0.00 & 0.00 & 0.00 \\
    \hline
    GAP-seconds (average) & 8.2 & 59.53 & 200.2 \\
    \hline
    \end{tabular}
    \caption{Average time taken and success rate for 30 iterations each for different radii of search ball $\mathcal{B}_R$ and SCSP parameters for the Grigorchuk group. Number of conjugated elements $n = 5$ for all iterations.}
    \label{tab:table-grig-big}
\end{table}

\subsection{Universal Grigorchuk group $\G_U$}

The Universal Grigorchuk group $\G_U$ was introduced in~\cite[Problem~9.5]{grigorch:solved} as the group that has all Grigorchuk groups $G_{\omega}$ from~\cite{grigorch:degrees} as quotients. Therefore, it is not a linear group as it has a quotient $\G$ of intermediate growth. Moreover, it is generated by an automaton of exponential growth, so the results from~\cite{bondarenko_bsz:conjugacy} cannot be used to construct a conjugator in a larger group. It was observed by Nekrashevych that $\G_U$ has a self-similar action on the 6-ary tree defined by the following wreath recursion:
\[\begin{array}{l}
a = (1, 1, 1, 1, 1, 1)(1,4)(2,5)(3,6),\\
b = (a, a, 1, b, b, b),\\
c = (a, 1, a, c, c, c),\\
d = (1, a, a, d, d, d).
\end{array}\]
The nucleus $\N$ can be found using \verb"AutomGrp" as follows: 
\vspace{2mm}
\begin{verbatim}
gap> G_U := AG_Groups.UniversalGrigorchukGroup;
< a, b, c, d >
gap> GeneratingSetWithNucleus(G_U);
[ 1, a, b, c, d ]
\end{verbatim}
\vspace{2mm}

Experimental results for LBA attack for $\G_U$ are recorded in Table~\ref{table-uni-grig}. We note that due to the large degree of the tree on which $\G_U$ acts, the simulations for this group are the most time consuming among all contracting groups in our study. Some of the reported simulations took over 48 hours to complete. Therefore, the experimental results are reported for relatively small range of parameters. 

\hspace{1cm}
\begin{table}[h]
\centering
\begin{tabular}{|p{6.2cm}||c|c||c|c|}
    \hline
    Automaton group & \multicolumn{4}{|c|}{Universal Grigorchuk group $\G_U$} \\
    \hline
    \hline
    Reduced word length of conjugator $r$ & \multicolumn{2}{|c||}{5} & \multicolumn{2}{|c|}{7}  \\
    \hline
    Radius $R$ of the search ball $\B_R$ & 2 & 3 & 3 & 4   \\
    \hline
    Percentage of success (\%) & 0.00 & 50.00 & 0.00 & 30.00   \\
    \hline
    GAP-seconds (average) & 458.79 & 2174.71 & 4009.83 & 9408.09    \\

    \hline 
\end{tabular}
\caption{Average time taken and success rate for 30 iterations each for different radii of search ball $\mathcal{B}_R$ and reduced word length of conjugator  for the Universal Grigorchuk group; $|a| = 10$ and $n = 5$ for all iterations.} 
\label{table-uni-grig}
\end{table}

\subsection{Basilica group $\mathcal{B}\cong \textrm{IMG}(z^2-1)$}
\label{ssec:basilica_empirical}

Basilica group $\B=\textrm{IMG}(z^2-1)$ is among the automaton groups proposed as platform for AAG protocol in~\cite{grigorchuk2019keyagreement}. The group was first studied in~\cite{grigorch_z:basilica,grigorch_z:basilica_sp} and has been an important example in many situations. It is torsion free, has exponential growth, amenable but not elementary (and not even subexponentially) amenable~\cite{bartholdi_v:amenab}, has trivial Poisson boundary, is weakly branch, and has many other interesting properties. It is generated by bounded automaton in the The automaton generating the Basilica group is shown in Figure~\ref{fig:basilica}.

The nucleus $\N$ can be found in GAP using the package \verb"AutomGrp" as follows: 

\vspace{2mm}

\begin{verbatim}
gap> B := AutomatonGroup("u = (v,1)(1,2), v = (u,1)");
< u, v >
gap> GeneratingSetWithNucleus(B);
[ 1, u, v, u^-1, v^-1, u^-1*v, v^-1*u ]
\end{verbatim}
\vspace{2mm}

Experimental results for LBA attack for $\B$ are recorded in Table~\ref{table-basilica}, which shows that the attack had no success for conjugatros of length at least 20. We note that there is no essential time difference for computations with search radius $2$ even when the reduced word length of conjugator $r$ is increased from $30$ to $100$. This is because in most cases iterations terminate with relatively little search for the conjugator since very few (if any) elements in $\B_R$ tend to decrease the value of the length function.

Moreover, the time taken for LBA iterations to conclude significantly increases as the radius of search ball $\B_R$ is increased from $2$ to $3$. In particular, for $|r| = 30$ and $R = 3$, a set of 10 iterations was terminated after running for over $60$ hours without conclusion.

\hspace{1cm}
\begin{table}[h]
\centering
\begin{tabular}{|p{6.2cm}||c|c||c|c||c||c|}
    \hline
    
    Automaton group & \multicolumn{6}{|c|}{Basilica group $\B$} \\
    \hline
    \hline
    Reduced word length of conjugator $r$ & \multicolumn{2}{|c||}{10} & \multicolumn{2}{|c||}{20} & \multicolumn{1}{|c||}{30} & \multicolumn{1}{|c|}{100} \\
    \hline
    Radius $R$ of the search ball $\B_R$ & 2 &3 & 2 & 3 & 2  & 2\\
    \hline
    Percentage of success (\%) & 6.67 & 33.33 & 0.00 & 0.00 & 0.00 & 0.00 \\
    \hline
    GAP-seconds (average) & 41.04 & 6871.21 & 62.6 & 5089.57 & 80.1  & 51.77  \\

    \hline 
\end{tabular}
\caption{Average time taken and success rate for 30 iterations each for different radii of search ball $\mathcal{B}_R$ and reduced word length of conjugator  for the Basilica group; $|a| = 10$ and $n = 5$ for all iterations.} 
\label{table-basilica}
\end{table}

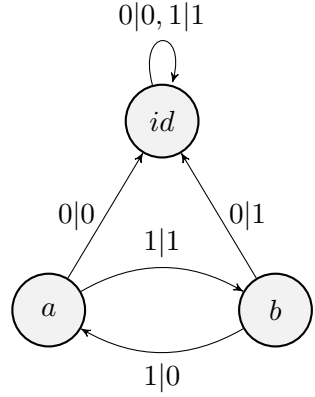
\begin{figure}
    \centering
    \begin{tikzpicture}

    \node[state] (q1) {$a$};
    \node[state, right of=q1] (q2) {$b$};
    \node[state] (q3) at (1.5, 2.5) {$id$};

    \draw (q1) edge[left] node{$0|0$} (q3);
    \draw (q1) edge[bend left , above] node{$1|1$} (q2);

    \draw (q2) edge[right] node{$0|1$} (q3);
    \draw (q2) edge[bend left , below] node{$1|0$} (q1);

    \draw (q3) edge[loop above] node{$0|0 , 1|1$} (q3);

    \end{tikzpicture}
\caption{Automaton $\A_{\mathcal{B}}$ generating the Basilica group.}
    \label{fig:basilica}
\end{figure}

\subsection{$p$-Basilica groups}
\label{ssec:7basilica_empirical}
A natural generalization of the Basilica group to the $p$-ary tree, called $p$-Basilica groups, was recently proposed in~\cite{didomenico_fnt:p-basilica22}. The groups in this class share many properties with the original Basilical group, such as being torsion-free, just-non-solvable, etc.

\[\begin{array}{l}
a = (1, \ldots, 1, b),\\
b = (1, \ldots, 1, a)\sigma,\\
\end{array}\]
where $\sigma=(1,2,\ldots,p)$ represents the cycle in $\Sym(p)$.

It is also shown in~\cite[Theorem~3.8]{didomenico_fnt:p-basilica22} that these groups are contracting with the nucleus $\{1,a,b,a^{-1},b^{-1}\}$.

Experimental results for LBA attack for $\B_p$ for $p=3,7,11$ are recorded in Tables~\ref{table-3-basilica},~\ref{table-7-basilica} and~\ref{table-11-basilica} respectively.

We also mention another generalization of Basilica group defined and studied in~\cite{petschick_r:on_the_basilica_operation23}. Namely, given any self-similar group $G$ and $s\geq 1$ the authors define the group $\mathrm{Bas}_s(G)$ that preserves many properties of $G$, including being contracting.     

\hspace{1cm}
\begin{table}[h]
\centering
\begin{tabular}{|p{6.2cm}||c|c||c|c||c||c|}
    \hline
    
    Automaton group & \multicolumn{6}{|c|}{$3$-Basilica group} \\
    \hline
    \hline
    Reduced word length of conjugator $r$ & \multicolumn{2}{|c||}{10} & \multicolumn{2}{|c||}{20} & \multicolumn{1}{|c||}{30} & \multicolumn{1}{|c|}{100} \\
    \hline
    Radius $R$ of the search ball $\B_R$ & 2 &3 & 2 & 3 & 2  & 2\\
    \hline
    Percentage of success (\%) & 0.00 & 33.33 & 0.00 & 6.67 & 0.00 & 0.00 \\
    \hline
    GAP-seconds (average) & 6.8 & 82.5 & 7.2 & 1884.9 & 11.3  & 10.8  \\

    \hline 
\end{tabular}
\caption{Average time taken and success rate for 30 iterations each for different radii of search ball $\mathcal{B}_R$ and reduced word length of conjugator  for the $3$-Basilica group; $|a| = 10$ and $n = 5$ for all iterations.} 
\label{table-3-basilica}
\end{table}

\hspace{1cm}
\begin{table}[h]
\centering
\begin{tabular}{|p{6.2cm}||c|c||c|c||c||c|}
    \hline
    
    Automaton group & \multicolumn{6}{|c|}{$7$-Basilica group} \\
    \hline
    \hline
    Reduced word length of conjugator $r$ & \multicolumn{2}{|c||}{10} & \multicolumn{2}{|c||}{20} & \multicolumn{1}{|c||}{30} & \multicolumn{1}{|c|}{100} \\
    \hline
    Radius $R$ of the search ball $\B_R$ & 2 &3 & 2 & 3 & 2  & 2\\
    \hline
    Percentage of success (\%) & 0.00 & 23.3 & 0.00 & 3.3 & 0.00 & 0.00 \\
    \hline
    GAP-seconds (average) & 6.0 & 219.0 & 9.2 & 2311.0 & 11.6 & 9.7  \\

    \hline 
\end{tabular}
\caption{Average time taken and success rate for 30 iterations each for different radii of search ball $\mathcal{B}_R$ and reduced word length of conjugator  for the $7$-Basilica group; $|a| = 10$ and $n = 5$ for all iterations.} 
\label{table-7-basilica}
\end{table}

\hspace{1cm}
\begin{table}[h]
\centering
\begin{tabular}{|p{6.2cm}||c|c||c|c||c||c|}
    \hline
    
    Automaton group & \multicolumn{6}{|c|}{$11$-Basilica group} \\
    \hline
    \hline
    Reduced word length of conjugator $r$ & \multicolumn{2}{|c||}{10} & \multicolumn{2}{|c||}{20} & \multicolumn{1}{|c||}{30} & \multicolumn{1}{|c|}{100} \\
    \hline
    Radius $R$ of the search ball $\B_R$ & 2 &3 & 2 & 3 & 2  & 2\\
    \hline
    Percentage of success (\%) & 0.00 & 3.33 & 0.00 & 0.00 & 0.00 & 0.00 \\
    \hline
    GAP-seconds (average) & 1.6 & 20.8 & 4.2 & 59.9 & 8.2 & 37.6  \\

    \hline 
\end{tabular}
\caption{Average time taken and success rate for 30 iterations each for different radii of search ball $\mathcal{B}_R$ and reduced word length of conjugator  for the $11$-Basilica group; $|a| = 10$ and $n = 5$ for all iterations.} 
\label{table-11-basilica}
\end{table}

\subsection{Discussion of experimental results}
From the experimental results provided in subsections above we draw the following conclusions:
\begin{itemize}
\item The variant of LBA attack against SCSP described in Algorithm~\ref{alg:lba} is ineffective against all tested groups except virtually abelian. Generally, virtually abelian automaton groups are not recommended platforms and appear here only as negative controls in experiments.

\item The attack takes generally longer for groups acting on trees of larger degrees. However, this is not always the case as can be seen for Basilica and $p$-Basilica groups for $p=3,7,11$. 

\item There is often no essential time difference for computations with a fixed search radius  when the reduced word length of conjugator $r$ is increased. This is because in most cases iterations terminate with relatively little search for the conjugator since very few (if any) elements in $\B_R$ tend to decrease the value of the length function. However, the success rate of the attack generally decreases to 0 except for the virtually abelian groups.
 
\item  The time taken for LBA iterations to conclude significantly increases as the radius $R$ of search ball $\B_R$ is increased, which is expected since in the number of branches at each step of the LBA attack is equal to the size of $\B_R$. 

\end{itemize}

\subsection{Key generation and recommendation for security parameters}
In this subsection we discuss possible practical parameters for AAG cryptosystem based on contracting groups.

Suppose that a contracting group $G$ generated by set $S$ with nucleus $\N$ is selected as a platform group for AAG cryptosystem as described in Subsection~\ref{AAG}. To make analysis simpler, we will assume that $N_1=N_2=N$. 

First, Alice and Bob need to select their public keys $\overline a=(a_1,\ldots,a_N)\in G^N$ and $\overline b=(b_1,\ldots,b_N)\in G^N$. They can do it by randomly choosing reduced words by following a simple random walk in the group corresponding to the uniform measure concentrated on the generating set. There is usually no efficient way to find the shortest word representing the obtained element, though.  The length of a random element of $G$ obtained with a simple random walk of length $n$ depends on the speed of this random walk. This is an active area of study in the theory of random walks and, in the context of contracting groups, Tianyi Zheng  in a recent paper~\cite[Proposition~1.5]{zheng:randow_walks_on_grigorchuk_group25} proved asymptotic bounds for the speed of simple random walk $W_n$ on the Grigorchuk group $\G$: 
\[n^{\frac1{2\alpha_0}}\lesssim \mathbb E|W_n|\lesssim n^{\frac34},\] 
where $\alpha_0=\frac{\log 2}{\log 2/\eta}\approx 0.7674$, $\eta$ is the positive root of the polynomial $X^3+X^2+X-2$. Such bounds can give theoretical justification of how long should a random word be in order to achieve the desired length of an element.

However, we suggest to use the benefits of the contracting portraits to generate a key of desired size since by Corollary~\ref{cor:log_depth} the size of a portrait (proportional to its boundary size) of a word of length $n$ is bounded above by a polynomial in $n$ of degree that is equal to the level at which all pairwise products of nucleus elements contract to the nucleus. Moreover, empirical results in Figures~\ref{fig:grigorchuk_hist},~\ref{fig:basilica_hist}, and~\ref{fig:7basilica_hist} show that the actual boundary size of a portrait generated by a random word is significantly smaller than the length of this word. For example, in the case of random words of length 5,000 representing elements of Grigorchuk group, the average boundary size of the contracting portraits was only $145.99$ and the maximum was $169$ (the bottom of Figure~\ref{fig:grigorchuk_hist}). Moreover, the ratio $\frac{|w|}{s(g_w)}$ of the length of the word $w$ (which may be not a word of the shortest length) representing an element $g_w$ to the size of the portrait boundary of this element is increasing monotonically with $n$. This is generally an anticipated phenomenon since the expected distance of a random walk of length $n$ on the group from the identity, depending on a group, may be asymptotically negligible compared to $n$ as, for example, was shown by Zheng in the case of Grigorchuk group. Therefore, random words will represent elements that can be presented by much shorter words and thus have much smaller portraits. Thus contracting portraits represent much more memory efficient way to store group element and operate with them, compared to random words in generators. Finally, the size of the contracting portrait of an element gives a linear lower bound for the size of the word representing this element. Thus, unlike random words, the portraits can guarantee the desired complexity of an element.

Thus, to generate the keys, we suggest to sample random words in the generators of the group in such a way that the resulting contracting portrait will be of a desired size. The random generation can be done with any cryptographically secure pseudorandom number generator, such as Blum Blum Shub (B.B.S.) generator~\cite{blum_bs:pseudorandom_number_generator86}. We also refer the reader to the NIST recommendations~\cite{nist800-90a:random_generation}. Since both sampling of random words and contracting portrait computations are fast, it is very a very practical procedure. Once the words are obtained, Alice and Bob convert them into the portraits. The memory required to store the portrait is linear in the size of the portrait boundary by Corollary~\ref{cor:memory}. Note that we are unaware of any other easy way to take a random group element with the contracting portrait of a given size without generating random words of sufficiently large length.

Now, for the AAG cryptosystem, we estimate the memory required to store the private and public keys, and to transmit the information needed to establish the shared secret key. Of course, these parameters depend on the chosen group, so we provide some particular parameters for Grigorchuk group, Basilica group, and 7-Basilica group.

To estimate the bit-size of the public key, we specialize estimate~\eqref{eqn:portrait_size} for each of the mentioned groups. The values of the coefficient $\lceil\log_2|\N|\rceil+\frac{d(\lceil\log_2d\rceil+2)}{d-1}$ that converts the size of the portrait boundary to bits for four groups are given in Table~\ref{tab:coeff}. These values are used for empirical calculations given in Tables~\ref{tab:key_grig},~\ref{tab:key_basilica}, and~\ref{tab:key_7basilica} that show various security options and required transmission sizes for Grigorchuk group, Basilica group, and 7-Basilica group, respectively. Each of these tables has 4 data columns corresponding to 4 different combinations of private and public key sizes. 

Recall that the public key that Alice chooses is 
a tuple $\overline{a}(a_1,\ldots,a_N)\in G^N$. We propose to use $N=8$ for all cases and 2 options for the sizes of boundaries of $a_i$'s: 10 and 20. The bit-sizes of the public keys are computed according to coefficients from Table~\ref{tab:coeff}. For example, for Grigorchuk group 8 elements with boundary size 10 will take $8\cdot10\cdot 9=720$ bits.

The private key of Alice is a word $A=a_{s_1}^{\varepsilon_1} \cdots a_{s_L}^{\varepsilon_L}$ of length $L$, where $a_{s_i} \in \overline{a}$ and $\varepsilon_i \in \{\pm1\}$. Storing this word requires 
\begin{equation}
    \label{eqn:private_key_bits}
    L\cdot(\lceil\log N\rceil+1)
\end{equation} 
bits, where the term $\lceil\log N\rceil$ comes from the selection of $a_{s_i}\in\overline{a}$, and one additional bit is required to store $\varepsilon_i$. Note that while the portrait size of $A$ depends on the size of portraits of $a_i$'s in the public key, the quantity~\eqref{eqn:private_key_bits}, representing the security level of the private key, does not depend on these sizes. Based on empirical results on the ineffectiveness of the LBA attack against Grigorchuk and Basilica groups in Sebsections~\ref{ssec:grig_empirical},~\ref{ssec:basilica_empirical}, and~\ref{ssec:7basilica_empirical}, we propose 2 options for the (word) length $L$ of the private key: 32 and 64. According to expression~\ref{eqn:private_key_bits}, for $N=8$ these choices result in 128- and 256-bit private keys, respectively. We note that there is not enough information at this time on what security levels such a length of the private key corresponds to in general, or in the cases of specific groups, as the security level depends on existence of general attacks against the SCSP in the class of contracting groups, which were not yet proposed or studied. Thus, the length of the private key represents only the upper bound on the security level of the AAG system.

Tables~\ref{tab:key_grig},~\ref{tab:key_basilica}, and~\ref{tab:key_7basilica} show also the average (of 10 runs) of one-way transmission size and the shared secret key size in the number of leaves of corresponding portraits and in bits. This information is useful for bandwidth selection and choosing appropriate applications of the scheme. Recall that in order to establish a shared secret key, Alice sends to Bob contracting portraits of $b'_i=A^{-1}b_iA$ for all $b_i$ in Bob's public key $\overline{b}=(b_1,\ldots,b_N)$ and Bob does the symmetric operation. The average sum of the boundary sizes of $b_i'$ and the corresponding bit-size are given in rows 5 and 6 of the tables. Finally, the average boundary size and the bit-size of the shared secret key are given in the last 2 rows of the tables.

For this analysis we did not chose other values for the number $N$ of elements in the public keys as it will only slightly affect the size of the private key (due to $\log N$ term in~\eqref{eqn:private_key_bits}) and will not affect the transmission size and the shared secret key size. However, it may  have an impact on the general security of the system since it changes the number of conjugate pairs for the SCSP, which can be important for some of the potential attacks.

\begin{table}[]
\centering
\begin{tabular}{|l|c|c|c|}
\hline
 \cellcolor[HTML]{DDDDDD}~& \cellcolor[HTML]{FFFFC7}$d$ 
 & \cellcolor[HTML]{FFFFC7}$|\mathcal{N}|$ 
 & \cellcolor[HTML]{FFFFC7}$\left\lceil\log_2|\mathcal{N}|\right\rceil+\frac{d(\left\lceil\log_2 d\right\rceil+2)}{d-1}$\rule{0pt}{3ex} 
\\[1mm] \hline
\multicolumn{1}{|l|}{\cellcolor[HTML]{FFFFC7}Grigorchuk group} & 2 & 5 & 9 \\ \hline
\multicolumn{1}{|l|}{\cellcolor[HTML]{FFFFC7}Basilica group}   & 2 & 7 & 9 \\ \hline \multicolumn{1}{|l|}{\cellcolor[HTML]{FFFFC7}7-Basilica group\rule{0pt}{2.7ex}} & 7 & 7 & $\frac{53}{6}\approx 8.83$ \\[1mm] \hline
\multicolumn{1}{|l|}{\cellcolor[HTML]{FFFFC7}11-Basilica group\rule{0pt}{2.7ex}} & 11 & 7 & $\frac{48}{5}=9.6$ \\[1mm] \hline
\end{tabular}
\caption{Conversion constants from the portrait boundary size to bits\label{tab:coeff}}
\end{table}

\begin{center}
\begin{table}[]
\centering
\begin{tabular}{|
>{\columncolor[HTML]{FFFFC7}}l |l|l|l|l|}
\hline
\textbf{Private key length $L$}    & 32    & 32     & 64    & 64     \\ \hline
\textbf{Private key size (bits)}    & 128    & 128     & 256    & 256     \\ \hline
\textbf{Boundary size of portraits in public key}     & 10    & 20    & 10    & 20    \\ \hline
\textbf{Public key size (bits)}     & 720    & 1440    & 720    & 1440    \\ \hline
\textbf{Average transmission size (leaves)} & 638.9  & 1165.1  & 890.5  & 1710.4  \\ \hline
\textbf{Average transmission size (bits)}   & 5750.1 & 10485.9 & 8014.5 & 15393.6 \\ \hline
\textbf{Average shared secret key size (leaves)}   & 114.2  & 213.2   & 163.1  & 295.5   \\ \hline
\textbf{Average shared secret key size (bits)}     & 1027.8 & 1918.8  & 1467.9 & 2659.5  \\ \hline
\end{tabular}
\caption{Keys and transmission sizes of AAG cryptosystem based on Grigorchuk group\label{tab:key_grig}}
\end{table}
\end{center}

\begin{center}
\begin{table}[]
\centering
\begin{tabular}{|
>{\columncolor[HTML]{FFFFC7}}l |l|l|l|l|}
\hline
\textbf{Private key length $L$}    & 32    & 32     & 64    & 64     \\ \hline
\textbf{Private key size (bits)}    & 128    & 128     & 256    & 256     \\ \hline
\textbf{Boundary size of portraits in public key}     & 10    & 20    & 10    & 20    \\ \hline
\textbf{Public key size (bits)}     & 720    & 1440    & 720    & 1440    \\ \hline
\textbf{Average transmission size (leaves)} & 1144.4  & 2294.5  & 1660.5  & 3617.6  \\ \hline
\textbf{Average transmission size (bits)}   & 10299.6 & 20650.5 & 14944.5 & 32558.4 \\ \hline
\textbf{Average shared secret key size (leaves)}   & 226.9   & 458.1   & 372.8   & 731     \\ \hline
\textbf{Average shared secret key size (bits)}     & 2042.1  & 4122.9  & 3355.2  & 6579    \\ \hline
\end{tabular}
\caption{Keys and transmission sizes of AAG cryptosystem based on Basilica group\label{tab:key_basilica}}
\end{table}
\end{center}

\begin{center}
\begin{table}[]
\centering
\begin{tabular}{|
>{\columncolor[HTML]{FFFFC7}}l |l|l|l|l|}
\hline
\textbf{Private key length $L$}    & 32    & 32     & 64    & 64     \\ \hline
\textbf{Private key size (bits)}    & 128    & 128     & 256    & 256     \\ \hline
\textbf{Boundary size of portraits in public key}     & 10    & 20    & 10    & 20    \\ \hline
\textbf{Public key size (bits)}     & 707    & 1414    & 707    & 1414    \\ \hline
\textbf{Average transmission size (leaves)} & 1239.90   & 2345     & 1825.30   & 3673.70   \\ \hline
\textbf{Average transmission size (bits)}   & 10952.45 & 20714.17 & 16123.48 & 32451.02 \\ \hline
\textbf{Average shared secret key size (leaves)}   & 239      & 464.60    & 379.50    & 728.20    \\ \hline
\textbf{Average shared secret key size (bits)}     & 2111.17 & 4103.97 & 3352.25  & 6432.43 \\ \hline\end{tabular}
\caption{Keys and transmission sizes of AAG cryptosystem based on 7-Basilica group\label{tab:key_7basilica}}
\end{table}
\end{center}

\subsection{Cryptanalysis challenge}
We have provided some challenge problems in the GitHub repository~\cite{kahrobaei_ms:lba_github}. The problems provide the publicly available information that is used to establish the AAG-key and asks to compute the private key or the shared secret key. The publicly available information comprises of the public keys of Alice and Bob, and the conjugated elements transmitted over the open channel by Alice and Bob as described in Subsection~\ref{AAG}. The platform groups used for these problems are $7$-Basilica, $11$-Basilica and Universal Grigorchuk groups.
\vspace{5mm}

\noindent\textbf{Acknowledgement.} The authors would like to express sincere gratitude to Rostislav Grigorchuk for fruitful discussions that helped to improve the paper. The first author conducted this work partially with the support of ONR Grant 62909-24-1-2002. The first author thanks Institut des Hautes \'Etudes Scientifiques - IHES for providing stimulating environment while this project was partially done. The third author was partially supported by NSF grant DMS-2342254. The second and the third authors greatly acknowledge the support of American Institute of Mathematics (AIM). Part of this work was done in June 2024 during the workshop ``Groups of dynamical origin'' funded by AIM. All authors greatly appreciate a meticulous work of anonymous referees that helped to improve exposition of the paper.

The authors acknowledge the support from
the Institut Henri Poincare (UAR 839 CNRS-Sorbonne Universite) and LabEx CARMIN (ANR-10-LABX-59-01).

\newcommand{\etalchar}[1]{$^{#1}$}
\def\cprime{$'$} \def\cydot{\leavevmode\raise.4ex\hbox{.}} \def\cprime{$'$} \def\cprime{$'$} \def\cprime{$'$} \def\cprime{$'$} \def\cprime{$'$} \def\cprime{$'$} \def\cprime{$'$} \def\cprime{$'$} \def\cprime{$'$} \def\cprime{$'$} \def\cprime{$'$} \def\cprime{$'$} \def\cprime{$'$} \def\cprime{$'$}

\end{document}